\documentclass[12pt, a4paper]{llncs}
\usepackage[margin=3cm]{geometry}
\usepackage[T1]{fontenc}
\usepackage{graphicx}
\usepackage{enumerate}
\usepackage{amsmath}
\usepackage{amssymb}
\usepackage{mathtools}
\usepackage{bm}
\usepackage{nicefrac}
\usepackage{pgfplots}
\usepgfplotslibrary{external}
\newcommand{\transpose}{^{\mathrm{T}}}
\newcommand{\Sref}{Section \ref}
\newcommand{\eref}{\eqref}
\newcommand{\Eref}{Equation \eqref}
\newcommand{\Tref}{Table \ref}
\newcommand{\Fref}{Figure \ref}
\begin{document}
\title{Stable Backward Diffusion Models that Minimise Convex Energies}
\author{Leif Bergerhoff \inst{1} \and
    Marcelo C{\'a}rdenas \inst{1} \and
    Joachim Weickert \inst{1} \and 
    Martin Welk \inst{2}}
\institute{%
    Mathematical Image Analysis Group,\\
    Faculty of Mathematics and Computer Science,\\
    Saarland University,\\
    Campus E1.7, 66041 Saarbr\"ucken, Germany\\
    \email{\{bergerhoff,cardenas,weickert\}@mia.uni-saarland.de}\\[1em]
    \and
    Institute of Biomedical Image Analysis,\\
    Private University for Health Sciences, Medical Informatics and 
    Technology,\\
    Eduard-Walln\"ofer-Zentrum 1,
    6060 Hall/Tyrol, Austria\\
    \email{martin.welk@umit.at}
}
%\author{Leif Bergerhoff \and
%        Marcelo C{\'a}rdenas \and
%        Joachim Weickert \and 
%        Martin Welk}
%%
%\institute{Leif Bergerhoff \and
%           Marcelo C{\'a}rdenas \and
%           Joachim Weickert \at
%               Mathematical Image Analysis Group,\\
%               Faculty of Mathematics and Computer Science,\\
%               Saarland University\\
%               Campus E1.7\\
%               66041 Saarbr\"ucken, Germany\\
%               Tel.: +49-681-302-57341\\
%               Fax: +49-681-302-57342\\
%               \email{\{bergerhoff,cardenas,weickert\}@mia.uni-saarland.de}
%           \and
%           Martin Welk \at
%               Institute of Biomedical Image Analysis,\\
%               Private University for Health Sciences, Medical Informatics and 
%               Technology,\\
%               Eduard-Walln\"ofer-Zentrum 1,\\
%               6060 Hall/Tyrol, Austria\\
%               Tel.: +43-50-8648-3974\\
%%               Fax: +43 50-8648-673836\\
%               \email{martin.welk@umit.at}
%}
%
% ------------------------------------------------------------------------------
%
%\date{Received: date / Accepted: date}
%
\maketitle
%
% ------------------------------------------------------------------------------
%
\begin{abstract}
The inverse problem of backward diffusion is known to be ill-posed and highly 
unstable. Backward diffusion processes appear naturally in image enhancement 
and deblurring applications. It is therefore greatly desirable to establish a 
backward diffusion model which implements a smart stabilisation approach that 
can be used in combination with an easy to handle numerical scheme.
So far, existing stabilisation strategies in literature require sophisticated 
numerics to solve the underlying initial value problem.
We derive a class of space-discrete one-dimensional backward diffusion as 
gradient descent of energies where we gain stability by imposing range 
constraints. Interestingly, these energies are even convex. Furthermore, we 
establish a comprehensive theory for the time-continuous evolution and we show 
that stability carries over to a simple explicit time discretisation of our 
model. Finally, we confirm the stability and usefulness of our technique in 
experiments in which we enhance the contrast of digital greyscale and colour 
images.
\\[1em]
%\subclass{\\ 76R50 \and
%          60J60 \and
%          58J65 \and
%          37N30 \and
%          37N40 \and
%          65D18 \and
%          68U10}
%
\end{abstract}
%
% ------------------------------------------------------------------------------
%
\section{Introduction}
\label{sec:intro}
Forward diffusion processes are well-suited to describe the smoothing of a 
given signal or image. This process of blurring implies a loss of high 
frequencies or details in the original data.
As a result, the inverse problem, backward diffusion, suffers from deficient 
information which are needed to uniquely reconstruct the original data. The 
introduction of noise due to measured data increases this difficulty even 
further.
Consequently, a solution to the inverse problem -- if it exists at all -- is 
highly sensitive and heavily depends on the input data: Even the smallest 
perturbation in the initial data can have a large impact on the evolution 
and may cause large deviations.
Therefore, it becomes clear that backward diffusion processes necessitate 
further stabilisation.

\paragraph{Previous Work on Backward Diffusion.}
Already more than 60 years ago, John \cite{Jo55} discussed the quality of a 
numerical solution to the inverse diffusion problem given that a solution 
exists, and that it is bounded and non-negative.
Since then, a large number of different regularisation methods have evolved 
which achieve stability by bounding the noise of the measured and 
the unperturbed data \cite{TS96}, by operator splitting \cite{KW02}, by Fourier 
regularisation \cite{FXQ07}, or by a modified Tikhonov regulariser \cite{ZM11}.
H\`ao and Duc \cite{HD09} suggest a mollification method where stability for 
the inverse diffusion problem follows from a convolution with the Dirichlet 
kernel. In \cite{HD11} the same authors provide a regularisation method 
for backward parabolic equations with time-dependent coefficients.
Ternat et al. \cite{TOD11} suggest low-pass filters and fourth-order 
regularisation terms for stabilisation.

Backward parabolic differential equations also enjoy high popularity in the 
image analysis community where they have e.g. been used for image restoration 
and image deblurring respectively.
The first contribution to backward diffusion dates back to 1955 when 
Kov\'asznay and Joseph \cite{KJ55} proposed to use the scaled negative 
Laplacian for contour enhancement.
Gabor \cite{Ga65} observed that the isotropy of the Laplacian 
operator leads to amplification of accidental noise at contour lines at the 
same time it enhances the contour lines. As a remedy, he proposed to restrict 
the contrast enhancement to the orthogonal contour direction and -- in a second 
step -- suggested additional smoothing in tangent direction.
Lindenbaum et al. \cite{LFB94} make use of averaged derivatives in order to 
improve the directional sensitive filter by Gabor. However, the 
authors point out that smoothing in only one direction favours the 
emergence of artefacts in nearly isotropic image regions. They recommend to use 
the Perona-Malik filter \cite{PM90} instead.
Forces of Perona-Malik type are also used by Pollak et al. 
\cite{PWK00} who specify a family of evolution equations to sharpen edges and 
suppress noise in the context of image segmentation.
In \cite{HFSWV94}, ter Haar Romeny et al. stress the influence of higher order 
time derivatives on the Gaussian deblurred image. Referring to the heat 
equation, the authors express the time derivatives in the spatial domain and 
approximate them using Gaussian derivatives.
Steiner et al. \cite{St98} highlight how backward diffusion can be used for 
feature enhancement in planar curves.

In the field of image processing, a frequently used stabilisation technique 
constrains the extrema in order to 
enforce a maximum-minimum principle. This is e.g. implemented in the inverse 
diffusion filter of Osher and Rudin \cite{OR91}. It imposes zero diffusivities 
at extrema and applies backward diffusion everywhere else.
The so-called forward-and-backward (FAB) diffusion of Gilboa et al. 
\cite{GSZ02a} follows a slightly different approach.
Closely related to the Perona-Malik filter \cite{PM90} it uses negative 
diffusivities for a specific range of gradient magnitudes.
On the other hand, it imposes forward diffusion for values of low and zero 
gradient magnitude. By doing so, the filter prevents the output values from 
exploding at extrema.
However, it is worth mentioning that -- so far -- all adequate implementations 
of inverse diffusion processes with forward or zero diffusivities at extrema 
require sophisticated numerical schemes. They use e.g.~minmod discretisations 
of the Laplacian \cite{OR91}, nonstandard finite difference approximations of 
the squared gradient magnitude \cite{WGW09}, and splittings into two-pixel 
interactions \cite{WWG18}.

Another, less popular stabilisation approach implies the application of a 
fidelity term and has been used to penalise deviations from the input image 
\cite{CSH78,SZ98} or from the average grey value of the desired range 
\cite{SC97}. Consequently, both the weights of the fidelity and the diffusion 
term control the range of the filtered image.

Further methods achieve stabilisation using a regularisation strategy built on 
FFT-based operators \cite{Ca14,Ca16,Ca17} and by the restriction to polynomials 
of fixed finite degree \cite{HKZ87}.
Mair et al. \cite{MWR96} discuss the well-posedness of deblurring Gaussian blur 
in the discrete image domain based on analytic number theory.

In summary, the presented methods offer an insight into the challenge of 
handling backward diffusion in practice and underline the demand for careful 
stabilisation strategies and sophisticated numerical methods.

In our paper we are going to present an alternative
approach to deal with backward diffusion problems. It prefers
smarter modelling over smarter numerics.
To understand it better, it is useful to recapitulate
some relations between diffusion and energy minimisation.

\paragraph{Diffusion and Energy Minimisation.}
For the sake of convenience we assume a one-dimensional evolution that 
smoothes an initial signal $f: [a,b] \to \mathbb{R}$.
In this context, the original signal $f$ serves as initial state of the 
diffusion equation
\begin{equation} \label{eq:fd}
\partial_t u = \partial_x \bigl(g(u_x^2)\,u_x\bigr)
\end{equation} 
where $u=u(x,t)$ represents the filtered outcome with $u(x,0)=f(x)$.
Additionally, let $u_x = \partial_x u$ and assume reflecting boundary 
conditions at $x=a$ and $x=b$.
Given a nonnegative diffusivity function $g$,
growing diffusion times $t$ lead to simpler representations of the input signal.
From Perona and Malik's work \cite{PM90} we know
that the smoothing effect at signal edges
can be reduced if $g$ is a decreasing function of the contrast $u_x^2$.
As long as the flux function $\varPhi(u_x) := g(u_x^2)\,u_x$ is strictly 
increasing in $u_x$ the corresponding forward diffusion process $\partial_tu = 
\varPhi'(u_x) u_{xx}$ involves no edge sharpening.
This diffusion can be regarded as the gradient descent 
evolution which minimises the energy
\begin{equation} \label{eq:energy}
E[u] = \int_a^b \varPsi(u_x^2) \, \mathrm{d}x .
\end{equation} 
The potential function $\tilde{\varPsi}(u_x)=\varPsi(u_x^2)$ is strictly convex 
in $u_x$, increasing in $u_x^2$, and fulfils $\varPsi'(u_x^2)=g(u_x^2)$.
Furthermore, the energy functional has a flat minimiser which is -- due to the 
strict convexity of the energy functional -- unique.
The gradient descent / diffusion evolution is well-posed and converges towards 
this minimiser for $t \to \infty$.
Due to this classical
emergence of well-posed forward diffusion
from strictly convex energies
it seems natural to believe that backward diffusion processes are necessarily 
associated with non-convex energies.
However, as we will see, this conjecture is wrong.
\paragraph{Our Contribution.}
In our article, we show that a specific class of backward diffusion processes 
are gradient descent evolutions of energies that have one unexpected property: 
They are convex! Our second innovation is the incorporation of a specific
constraint: We impose reflecting boundary conditions in the diffusion 
\emph{co-domain}.
This means that in case of greyscale images with an allowed grey value range of 
$(0, 255)$ the occurring values are mirrored at the boundary positions $0$ and 
$255$.
While such range constraints have shown their usefulness in 
some other context (see e.g. \cite{NS14}), to our knowledge they have never 
been used for stabilising backward diffusions.
For our novel backward diffusion models, we show also a surprising numerical 
fact: A simple explicit scheme turns out to be stable and convergent.
Last but not least, we apply our models to the contrast
enhancement of greyscale and colour images.

This article is a revised version of our conference contribution 
\cite{BCWW17} which we extend in several aspects.
First, we enhance our model for convex backward diffusion to support not only a 
global and weighted setting but also a localised variant.
We analyse this extended model in terms of stability and convergence towards a 
unique minimiser.
Furthermore, we
formulate a simple explicit scheme for our newly proposed 
approach which shares all important properties with the time-continuous 
evolution.
In this context, we provide a detailed discussion on the selection of 
suitable time step sizes.
Additionally, we suggest two new applications: 
global contrast enhancement of digital colour images and local contrast 
enhancement of digital grey and colour images.

\paragraph{Structure of the Paper.}
In \Sref{sec:model}, we present our model for convex backward diffusion 
with range constraints. We describe a general approach which allows to 
formulate weighted local and global evolutions. \Sref{sec:theory} 
includes proofs for model properties such as range and rank-order preservation 
as well as convergence analysis and the derivation of explicit steady-state 
solutions. \Sref{sec:numerics} provides a simple explicit scheme which
can be used to solve the occurring initial-value problem.
In \Sref{sec:application}, we explain how to enhance the global and 
local contrast of digital images using the proposed model. Furthermore, 
we discuss the relation to existing literature on contrast enhancement.
In \Sref{sec:summary}, we draw conclusions from our findings and give an 
outlook on future research.
%
% ------------------------------------------------------------------------------
% ------------------------------------------------------------------------------
%
\section{Model}
\label{sec:model}
Let us now explore the roots of our model and derive -- in a second 
step -- the particle evolution which forms the heart of our method and which is 
given by the gradient descent of a convex energy.
%
% ------------------------------------------------------------------------------
%
\subsection{Motivation from Swarm Dynamics}
The idea behind our model goes back to the scenario of describing a 
one-dimensional 
evolution of particles within a closed system. Recent literature on 
mathematical swarm models employs a pairwise potential $U: \mathbb{R}^n \to 
\mathbb{R}$ to characterise the behaviour of individual particles (see e.g. 
\cite{CFTV10,CHDB07,OCBC06,GF07,GP04} and the references therein). The 
potential function allows to steer 
attractive and repulsive forces among swarm mates. 
Physically simplified models like \cite{GP03} neglect inertia and describe the 
individual particle velocity
%\hspace{0.0001mm}
\hphantom{d}
$\mathclap{\partial_t \bm{v}_i}$
\hphantom{v}
%\hspace{0.0001mm}
within a swarm of size $N$ directly as
\begin{equation}
\partial_t \bm{v}_i =
-\sum\limits_{\scriptstyle j = 1 \atop \scriptstyle j \neq i}^N
\bm{\nabla} U(|\bm{v}_i-\bm{v}_j|) ,
\qquad i = 1,\ldots,N ,
\end{equation}
where $\bm{v}_i$ and $\bm{v}_j$ denote particle positions in $\mathbb{R}^n$.
These models are also referred to as first order models.
Often they are inspired by biology and describe long-range attractive and 
short-range repulsive behaviour between swarm members. The interplay of 
attractive and repulsive forces leads to flocking and allows to gain stability 
for the whole swarm.
Inverting this behaviour -- resulting in short-range attractive and long-range 
repulsive forces -- leads to a highly unstable scenario in which the swarm 
splits up into small separating groups which might never reach a point where 
they stand still.
One would expect
that a restriction to repulsive forces only will increase this 
instability even further.
However, we will present a model which copes well with exactly this situation.
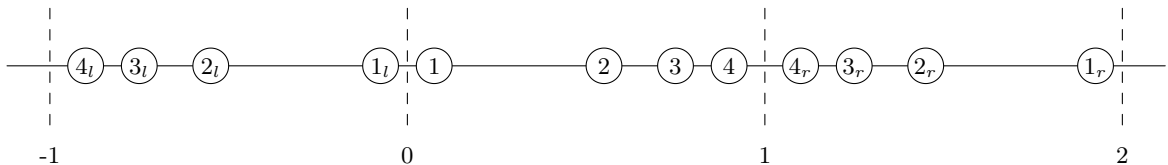
\begin{figure*}
    \centering
    \begin{tikzpicture}[x=1pt,y=1pt]
    \definecolor{fillColor}{RGB}{255,255,255}
    \path[use as bounding box,fill=fillColor,fill opacity=0.00] (0,0) rectangle 
    (433.62,108.41);
    \begin{scope}
    \path[clip] (  0.00,  0.00) rectangle (433.62,108.41);
    \definecolor{drawColor}{RGB}{0,0,0}
    
    \path[draw=drawColor,line width= 0.4pt,line join=round,line cap=round] (  
    0.00, 70.93) -- (433.62, 70.93);
    \definecolor{fillColor}{RGB}{255,255,255}
    
    \path[draw=drawColor,line width= 0.4pt,line join=round,line 
    cap=round,fill=fillColor] (146.55, 70.93) --
    (146.53, 71.35) --
    (146.49, 71.77) --
    (146.43, 72.19) --
    (146.34, 72.60) --
    (146.22, 73.00) --
    (146.08, 73.40) --
    (145.91, 73.78) --
    (145.72, 74.16) --
    (145.51, 74.52) --
    (145.27, 74.86) --
    (145.01, 75.20) --
    (144.73, 75.51) --
    (144.44, 75.81) --
    (144.12, 76.09) --
    (143.79, 76.35) --
    (143.44, 76.58) --
    (143.08, 76.80) --
    (142.71, 76.99) --
    (142.32, 77.15) --
    (141.92, 77.30) --
    (141.52, 77.41) --
    (141.11, 77.50) --
    (140.69, 77.57) --
    (140.28, 77.61) --
    (139.86, 77.62) --
    (139.44, 77.61) --
    (139.02, 77.57) --
    (138.60, 77.50) --
    (138.19, 77.41) --
    (137.79, 77.30) --
    (137.39, 77.15) --
    (137.01, 76.99) --
    (136.63, 76.80) --
    (136.27, 76.58) --
    (135.92, 76.35) --
    (135.59, 76.09) --
    (135.28, 75.81) --
    (134.98, 75.51) --
    (134.70, 75.20) --
    (134.44, 74.86) --
    (134.21, 74.52) --
    (133.99, 74.16) --
    (133.80, 73.78) --
    (133.63, 73.40) --
    (133.49, 73.00) --
    (133.37, 72.60) --
    (133.28, 72.19) --
    (133.22, 71.77) --
    (133.18, 71.35) --
    (133.16, 70.93) --
    (133.18, 70.51) --
    (133.22, 70.09) --
    (133.28, 69.68) --
    (133.37, 69.27) --
    (133.49, 68.86) --
    (133.63, 68.47) --
    (133.80, 68.08) --
    (133.99, 67.71) --
    (134.21, 67.35) --
    (134.44, 67.00) --
    (134.70, 66.67) --
    (134.98, 66.35) --
    (135.28, 66.05) --
    (135.59, 65.78) --
    (135.92, 65.52) --
    (136.27, 65.28) --
    (136.63, 65.07) --
    (137.01, 64.88) --
    (137.39, 64.71) --
    (137.79, 64.57) --
    (138.19, 64.45) --
    (138.60, 64.36) --
    (139.02, 64.29) --
    (139.44, 64.25) --
    (139.86, 64.24) --
    (140.28, 64.25) --
    (140.69, 64.29) --
    (141.11, 64.36) --
    (141.52, 64.45) --
    (141.92, 64.57) --
    (142.32, 64.71) --
    (142.71, 64.88) --
    (143.08, 65.07) --
    (143.44, 65.28) --
    (143.79, 65.52) --
    (144.12, 65.78) --
    (144.44, 66.05) --
    (144.73, 66.35) --
    (145.01, 66.67) --
    (145.27, 67.00) --
    (145.51, 67.35) --
    (145.72, 67.71) --
    (145.91, 68.08) --
    (146.08, 68.47) --
    (146.22, 68.86) --
    (146.34, 69.27) --
    (146.43, 69.68) --
    (146.49, 70.09) --
    (146.53, 70.51) --
    cycle;
    
    \node[text=drawColor,anchor=base,inner sep=0pt, outer sep=0pt, scale=  
    1.00] at 
    (139.86, 67.72) {\footnotesize $1_l$};
    
    \path[draw=drawColor,line width= 0.4pt,line join=round,line 
    cap=round,fill=fillColor] (166.62, 70.93) --
    (166.61, 71.35) --
    (166.57, 71.77) --
    (166.50, 72.19) --
    (166.41, 72.60) --
    (166.29, 73.00) --
    (166.15, 73.40) --
    (165.99, 73.78) --
    (165.79, 74.16) --
    (165.58, 74.52) --
    (165.34, 74.86) --
    (165.09, 75.20) --
    (164.81, 75.51) --
    (164.51, 75.81) --
    (164.20, 76.09) --
    (163.86, 76.35) --
    (163.52, 76.58) --
    (163.15, 76.80) --
    (162.78, 76.99) --
    (162.39, 77.15) --
    (162.00, 77.30) --
    (161.59, 77.41) --
    (161.18, 77.50) --
    (160.77, 77.57) --
    (160.35, 77.61) --
    (159.93, 77.62) --
    (159.51, 77.61) --
    (159.09, 77.57) --
    (158.68, 77.50) --
    (158.27, 77.41) --
    (157.86, 77.30) --
    (157.47, 77.15) --
    (157.08, 76.99) --
    (156.71, 76.80) --
    (156.35, 76.58) --
    (156.00, 76.35) --
    (155.67, 76.09) --
    (155.35, 75.81) --
    (155.05, 75.51) --
    (154.77, 75.20) --
    (154.52, 74.86) --
    (154.28, 74.52) --
    (154.07, 74.16) --
    (153.88, 73.78) --
    (153.71, 73.40) --
    (153.57, 73.00) --
    (153.45, 72.60) --
    (153.36, 72.19) --
    (153.29, 71.77) --
    (153.25, 71.35) --
    (153.24, 70.93) --
    (153.25, 70.51) --
    (153.29, 70.09) --
    (153.36, 69.68) --
    (153.45, 69.27) --
    (153.57, 68.86) --
    (153.71, 68.47) --
    (153.88, 68.08) --
    (154.07, 67.71) --
    (154.28, 67.35) --
    (154.52, 67.00) --
    (154.77, 66.67) --
    (155.05, 66.35) --
    (155.35, 66.05) --
    (155.67, 65.78) --
    (156.00, 65.52) --
    (156.35, 65.28) --
    (156.71, 65.07) --
    (157.08, 64.88) --
    (157.47, 64.71) --
    (157.86, 64.57) --
    (158.27, 64.45) --
    (158.68, 64.36) --
    (159.09, 64.29) --
    (159.51, 64.25) --
    (159.93, 64.24) --
    (160.35, 64.25) --
    (160.77, 64.29) --
    (161.18, 64.36) --
    (161.59, 64.45) --
    (162.00, 64.57) --
    (162.39, 64.71) --
    (162.78, 64.88) --
    (163.15, 65.07) --
    (163.52, 65.28) --
    (163.86, 65.52) --
    (164.20, 65.78) --
    (164.51, 66.05) --
    (164.81, 66.35) --
    (165.09, 66.67) --
    (165.34, 67.00) --
    (165.58, 67.35) --
    (165.79, 67.71) --
    (165.99, 68.08) --
    (166.15, 68.47) --
    (166.29, 68.86) --
    (166.41, 69.27) --
    (166.50, 69.68) --
    (166.57, 70.09) --
    (166.61, 70.51) --
    cycle;
    
    \node[text=drawColor,anchor=base,inner sep=0pt, outer sep=0pt, scale=  
    1.00] at 
    (159.93, 67.72) {\footnotesize $1$};
    
    \path[draw=drawColor,line width= 0.4pt,line join=round,line 
    cap=round,fill=fillColor] (414.21, 70.93) --
    (414.20, 71.35) --
    (414.16, 71.77) --
    (414.10, 72.19) --
    (414.00, 72.60) --
    (413.89, 73.00) --
    (413.74, 73.40) --
    (413.58, 73.78) --
    (413.39, 74.16) --
    (413.17, 74.52) --
    (412.94, 74.86) --
    (412.68, 75.20) --
    (412.40, 75.51) --
    (412.10, 75.81) --
    (411.79, 76.09) --
    (411.46, 76.35) --
    (411.11, 76.58) --
    (410.75, 76.80) --
    (410.37, 76.99) --
    (409.99, 77.15) --
    (409.59, 77.30) --
    (409.19, 77.41) --
    (408.78, 77.50) --
    (408.36, 77.57) --
    (407.94, 77.61) --
    (407.52, 77.62) --
    (407.10, 77.61) --
    (406.68, 77.57) --
    (406.27, 77.50) --
    (405.86, 77.41) --
    (405.45, 77.30) --
    (405.06, 77.15) --
    (404.67, 76.99) --
    (404.30, 76.80) --
    (403.94, 76.58) --
    (403.59, 76.35) --
    (403.26, 76.09) --
    (402.94, 75.81) --
    (402.64, 75.51) --
    (402.37, 75.20) --
    (402.11, 74.86) --
    (401.87, 74.52) --
    (401.66, 74.16) --
    (401.47, 73.78) --
    (401.30, 73.40) --
    (401.16, 73.00) --
    (401.04, 72.60) --
    (400.95, 72.19) --
    (400.88, 71.77) --
    (400.84, 71.35) --
    (400.83, 70.93) --
    (400.84, 70.51) --
    (400.88, 70.09) --
    (400.95, 69.68) --
    (401.04, 69.27) --
    (401.16, 68.86) --
    (401.30, 68.47) --
    (401.47, 68.08) --
    (401.66, 67.71) --
    (401.87, 67.35) --
    (402.11, 67.00) --
    (402.37, 66.67) --
    (402.64, 66.35) --
    (402.94, 66.05) --
    (403.26, 65.78) --
    (403.59, 65.52) --
    (403.94, 65.28) --
    (404.30, 65.07) --
    (404.67, 64.88) --
    (405.06, 64.71) --
    (405.45, 64.57) --
    (405.86, 64.45) --
    (406.27, 64.36) --
    (406.68, 64.29) --
    (407.10, 64.25) --
    (407.52, 64.24) --
    (407.94, 64.25) --
    (408.36, 64.29) --
    (408.78, 64.36) --
    (409.19, 64.45) --
    (409.59, 64.57) --
    (409.99, 64.71) --
    (410.37, 64.88) --
    (410.75, 65.07) --
    (411.11, 65.28) --
    (411.46, 65.52) --
    (411.79, 65.78) --
    (412.10, 66.05) --
    (412.40, 66.35) --
    (412.68, 66.67) --
    (412.94, 67.00) --
    (413.17, 67.35) --
    (413.39, 67.71) --
    (413.58, 68.08) --
    (413.74, 68.47) --
    (413.89, 68.86) --
    (414.00, 69.27) --
    (414.10, 69.68) --
    (414.16, 70.09) --
    (414.20, 70.51) --
    cycle;
    
    \node[text=drawColor,anchor=base,inner sep=0pt, outer sep=0pt, scale=  
    1.00] at 
    (407.52, 67.72) {\footnotesize $1_r$};
    
    \path[draw=drawColor,line width= 0.4pt,line join=round,line 
    cap=round,fill=fillColor] ( 82.98, 70.93) --
    ( 82.96, 71.35) --
    ( 82.92, 71.77) --
    ( 82.86, 72.19) --
    ( 82.77, 72.60) --
    ( 82.65, 73.00) --
    ( 82.51, 73.40) --
    ( 82.34, 73.78) --
    ( 82.15, 74.16) --
    ( 81.93, 74.52) --
    ( 81.70, 74.86) --
    ( 81.44, 75.20) --
    ( 81.16, 75.51) --
    ( 80.87, 75.81) --
    ( 80.55, 76.09) --
    ( 80.22, 76.35) --
    ( 79.87, 76.58) --
    ( 79.51, 76.80) --
    ( 79.13, 76.99) --
    ( 78.75, 77.15) --
    ( 78.35, 77.30) --
    ( 77.95, 77.41) --
    ( 77.54, 77.50) --
    ( 77.12, 77.57) --
    ( 76.71, 77.61) --
    ( 76.29, 77.62) --
    ( 75.86, 77.61) --
    ( 75.45, 77.57) --
    ( 75.03, 77.50) --
    ( 74.62, 77.41) --
    ( 74.22, 77.30) --
    ( 73.82, 77.15) --
    ( 73.44, 76.99) --
    ( 73.06, 76.80) --
    ( 72.70, 76.58) --
    ( 72.35, 76.35) --
    ( 72.02, 76.09) --
    ( 71.70, 75.81) --
    ( 71.41, 75.51) --
    ( 71.13, 75.20) --
    ( 70.87, 74.86) --
    ( 70.64, 74.52) --
    ( 70.42, 74.16) --
    ( 70.23, 73.78) --
    ( 70.06, 73.40) --
    ( 69.92, 73.00) --
    ( 69.80, 72.60) --
    ( 69.71, 72.19) --
    ( 69.65, 71.77) --
    ( 69.61, 71.35) --
    ( 69.59, 70.93) --
    ( 69.61, 70.51) --
    ( 69.65, 70.09) --
    ( 69.71, 69.68) --
    ( 69.80, 69.27) --
    ( 69.92, 68.86) --
    ( 70.06, 68.47) --
    ( 70.23, 68.08) --
    ( 70.42, 67.71) --
    ( 70.64, 67.35) --
    ( 70.87, 67.00) --
    ( 71.13, 66.67) --
    ( 71.41, 66.35) --
    ( 71.70, 66.05) --
    ( 72.02, 65.78) --
    ( 72.35, 65.52) --
    ( 72.70, 65.28) --
    ( 73.06, 65.07) --
    ( 73.44, 64.88) --
    ( 73.82, 64.71) --
    ( 74.22, 64.57) --
    ( 74.62, 64.45) --
    ( 75.03, 64.36) --
    ( 75.45, 64.29) --
    ( 75.86, 64.25) --
    ( 76.29, 64.24) --
    ( 76.71, 64.25) --
    ( 77.12, 64.29) --
    ( 77.54, 64.36) --
    ( 77.95, 64.45) --
    ( 78.35, 64.57) --
    ( 78.75, 64.71) --
    ( 79.13, 64.88) --
    ( 79.51, 65.07) --
    ( 79.87, 65.28) --
    ( 80.22, 65.52) --
    ( 80.55, 65.78) --
    ( 80.87, 66.05) --
    ( 81.16, 66.35) --
    ( 81.44, 66.67) --
    ( 81.70, 67.00) --
    ( 81.93, 67.35) --
    ( 82.15, 67.71) --
    ( 82.34, 68.08) --
    ( 82.51, 68.47) --
    ( 82.65, 68.86) --
    ( 82.77, 69.27) --
    ( 82.86, 69.68) --
    ( 82.92, 70.09) --
    ( 82.96, 70.51) --
    cycle;
    
    \node[text=drawColor,anchor=base,inner sep=0pt, outer sep=0pt, scale=  
    1.00] at 
    ( 76.29, 67.72) {\footnotesize $2_l$};
    
    \path[draw=drawColor,line width= 0.4pt,line join=round,line 
    cap=round,fill=fillColor] (230.19, 70.93) --
    (230.18, 71.35) --
    (230.14, 71.77) --
    (230.07, 72.19) --
    (229.98, 72.60) --
    (229.87, 73.00) --
    (229.72, 73.40) --
    (229.56, 73.78) --
    (229.37, 74.16) --
    (229.15, 74.52) --
    (228.92, 74.86) --
    (228.66, 75.20) --
    (228.38, 75.51) --
    (228.08, 75.81) --
    (227.77, 76.09) --
    (227.43, 76.35) --
    (227.09, 76.58) --
    (226.73, 76.80) --
    (226.35, 76.99) --
    (225.97, 77.15) --
    (225.57, 77.30) --
    (225.17, 77.41) --
    (224.76, 77.50) --
    (224.34, 77.57) --
    (223.92, 77.61) --
    (223.50, 77.62) --
    (223.08, 77.61) --
    (222.66, 77.57) --
    (222.25, 77.50) --
    (221.84, 77.41) --
    (221.43, 77.30) --
    (221.04, 77.15) --
    (220.65, 76.99) --
    (220.28, 76.80) --
    (219.92, 76.58) --
    (219.57, 76.35) --
    (219.24, 76.09) --
    (218.92, 75.81) --
    (218.62, 75.51) --
    (218.35, 75.20) --
    (218.09, 74.86) --
    (217.85, 74.52) --
    (217.64, 74.16) --
    (217.45, 73.78) --
    (217.28, 73.40) --
    (217.14, 73.00) --
    (217.02, 72.60) --
    (216.93, 72.19) --
    (216.86, 71.77) --
    (216.82, 71.35) --
    (216.81, 70.93) --
    (216.82, 70.51) --
    (216.86, 70.09) --
    (216.93, 69.68) --
    (217.02, 69.27) --
    (217.14, 68.86) --
    (217.28, 68.47) --
    (217.45, 68.08) --
    (217.64, 67.71) --
    (217.85, 67.35) --
    (218.09, 67.00) --
    (218.35, 66.67) --
    (218.62, 66.35) --
    (218.92, 66.05) --
    (219.24, 65.78) --
    (219.57, 65.52) --
    (219.92, 65.28) --
    (220.28, 65.07) --
    (220.65, 64.88) --
    (221.04, 64.71) --
    (221.43, 64.57) --
    (221.84, 64.45) --
    (222.25, 64.36) --
    (222.66, 64.29) --
    (223.08, 64.25) --
    (223.50, 64.24) --
    (223.92, 64.25) --
    (224.34, 64.29) --
    (224.76, 64.36) --
    (225.17, 64.45) --
    (225.57, 64.57) --
    (225.97, 64.71) --
    (226.35, 64.88) --
    (226.73, 65.07) --
    (227.09, 65.28) --
    (227.43, 65.52) --
    (227.77, 65.78) --
    (228.08, 66.05) --
    (228.38, 66.35) --
    (228.66, 66.67) --
    (228.92, 67.00) --
    (229.15, 67.35) --
    (229.37, 67.71) --
    (229.56, 68.08) --
    (229.72, 68.47) --
    (229.87, 68.86) --
    (229.98, 69.27) --
    (230.07, 69.68) --
    (230.14, 70.09) --
    (230.18, 70.51) --
    cycle;
    
    \node[text=drawColor,anchor=base,inner sep=0pt, outer sep=0pt, scale=  
    1.00] at 
    (223.50, 67.72) {\footnotesize $2$};
    
    \path[draw=drawColor,line width= 0.4pt,line join=round,line 
    cap=round,fill=fillColor] (350.64, 70.93) --
    (350.63, 71.35) --
    (350.59, 71.77) --
    (350.52, 72.19) --
    (350.43, 72.60) --
    (350.32, 73.00) --
    (350.17, 73.40) --
    (350.01, 73.78) --
    (349.82, 74.16) --
    (349.60, 74.52) --
    (349.37, 74.86) --
    (349.11, 75.20) --
    (348.83, 75.51) --
    (348.53, 75.81) --
    (348.22, 76.09) --
    (347.88, 76.35) --
    (347.54, 76.58) --
    (347.18, 76.80) --
    (346.80, 76.99) --
    (346.42, 77.15) --
    (346.02, 77.30) --
    (345.62, 77.41) --
    (345.21, 77.50) --
    (344.79, 77.57) --
    (344.37, 77.61) --
    (343.95, 77.62) --
    (343.53, 77.61) --
    (343.11, 77.57) --
    (342.70, 77.50) --
    (342.29, 77.41) --
    (341.88, 77.30) --
    (341.49, 77.15) --
    (341.10, 76.99) --
    (340.73, 76.80) --
    (340.37, 76.58) --
    (340.02, 76.35) --
    (339.69, 76.09) --
    (339.37, 75.81) --
    (339.07, 75.51) --
    (338.80, 75.20) --
    (338.54, 74.86) --
    (338.30, 74.52) --
    (338.09, 74.16) --
    (337.90, 73.78) --
    (337.73, 73.40) --
    (337.59, 73.00) --
    (337.47, 72.60) --
    (337.38, 72.19) --
    (337.31, 71.77) --
    (337.27, 71.35) --
    (337.26, 70.93) --
    (337.27, 70.51) --
    (337.31, 70.09) --
    (337.38, 69.68) --
    (337.47, 69.27) --
    (337.59, 68.86) --
    (337.73, 68.47) --
    (337.90, 68.08) --
    (338.09, 67.71) --
    (338.30, 67.35) --
    (338.54, 67.00) --
    (338.80, 66.67) --
    (339.07, 66.35) --
    (339.37, 66.05) --
    (339.69, 65.78) --
    (340.02, 65.52) --
    (340.37, 65.28) --
    (340.73, 65.07) --
    (341.10, 64.88) --
    (341.49, 64.71) --
    (341.88, 64.57) --
    (342.29, 64.45) --
    (342.70, 64.36) --
    (343.11, 64.29) --
    (343.53, 64.25) --
    (343.95, 64.24) --
    (344.37, 64.25) --
    (344.79, 64.29) --
    (345.21, 64.36) --
    (345.62, 64.45) --
    (346.02, 64.57) --
    (346.42, 64.71) --
    (346.80, 64.88) --
    (347.18, 65.07) --
    (347.54, 65.28) --
    (347.88, 65.52) --
    (348.22, 65.78) --
    (348.53, 66.05) --
    (348.83, 66.35) --
    (349.11, 66.67) --
    (349.37, 67.00) --
    (349.60, 67.35) --
    (349.82, 67.71) --
    (350.01, 68.08) --
    (350.17, 68.47) --
    (350.32, 68.86) --
    (350.43, 69.27) --
    (350.52, 69.68) --
    (350.59, 70.09) --
    (350.63, 70.51) --
    cycle;
    
    \node[text=drawColor,anchor=base,inner sep=0pt, outer sep=0pt, scale=  
    1.00] at 
    (343.95, 67.72) {\footnotesize $2_r$};
    
    \path[draw=drawColor,line width= 0.4pt,line join=round,line 
    cap=round,fill=fillColor] ( 56.21, 70.93) --
    ( 56.20, 71.35) --
    ( 56.16, 71.77) --
    ( 56.09, 72.19) --
    ( 56.00, 72.60) --
    ( 55.88, 73.00) --
    ( 55.74, 73.40) --
    ( 55.57, 73.78) --
    ( 55.38, 74.16) --
    ( 55.17, 74.52) --
    ( 54.93, 74.86) --
    ( 54.67, 75.20) --
    ( 54.40, 75.51) --
    ( 54.10, 75.81) --
    ( 53.78, 76.09) --
    ( 53.45, 76.35) --
    ( 53.10, 76.58) --
    ( 52.74, 76.80) --
    ( 52.37, 76.99) --
    ( 51.98, 77.15) --
    ( 51.59, 77.30) --
    ( 51.18, 77.41) --
    ( 50.77, 77.50) --
    ( 50.36, 77.57) --
    ( 49.94, 77.61) --
    ( 49.52, 77.62) --
    ( 49.10, 77.61) --
    ( 48.68, 77.57) --
    ( 48.26, 77.50) --
    ( 47.85, 77.41) --
    ( 47.45, 77.30) --
    ( 47.05, 77.15) --
    ( 46.67, 76.99) --
    ( 46.29, 76.80) --
    ( 45.93, 76.58) --
    ( 45.59, 76.35) --
    ( 45.25, 76.09) --
    ( 44.94, 75.81) --
    ( 44.64, 75.51) --
    ( 44.36, 75.20) --
    ( 44.10, 74.86) --
    ( 43.87, 74.52) --
    ( 43.65, 74.16) --
    ( 43.46, 73.78) --
    ( 43.30, 73.40) --
    ( 43.15, 73.00) --
    ( 43.04, 72.60) --
    ( 42.95, 72.19) --
    ( 42.88, 71.77) --
    ( 42.84, 71.35) --
    ( 42.83, 70.93) --
    ( 42.84, 70.51) --
    ( 42.88, 70.09) --
    ( 42.95, 69.68) --
    ( 43.04, 69.27) --
    ( 43.15, 68.86) --
    ( 43.30, 68.47) --
    ( 43.46, 68.08) --
    ( 43.65, 67.71) --
    ( 43.87, 67.35) --
    ( 44.10, 67.00) --
    ( 44.36, 66.67) --
    ( 44.64, 66.35) --
    ( 44.94, 66.05) --
    ( 45.25, 65.78) --
    ( 45.59, 65.52) --
    ( 45.93, 65.28) --
    ( 46.29, 65.07) --
    ( 46.67, 64.88) --
    ( 47.05, 64.71) --
    ( 47.45, 64.57) --
    ( 47.85, 64.45) --
    ( 48.26, 64.36) --
    ( 48.68, 64.29) --
    ( 49.10, 64.25) --
    ( 49.52, 64.24) --
    ( 49.94, 64.25) --
    ( 50.36, 64.29) --
    ( 50.77, 64.36) --
    ( 51.18, 64.45) --
    ( 51.59, 64.57) --
    ( 51.98, 64.71) --
    ( 52.37, 64.88) --
    ( 52.74, 65.07) --
    ( 53.10, 65.28) --
    ( 53.45, 65.52) --
    ( 53.78, 65.78) --
    ( 54.10, 66.05) --
    ( 54.40, 66.35) --
    ( 54.67, 66.67) --
    ( 54.93, 67.00) --
    ( 55.17, 67.35) --
    ( 55.38, 67.71) --
    ( 55.57, 68.08) --
    ( 55.74, 68.47) --
    ( 55.88, 68.86) --
    ( 56.00, 69.27) --
    ( 56.09, 69.68) --
    ( 56.16, 70.09) --
    ( 56.20, 70.51) --
    cycle;
    
    \node[text=drawColor,anchor=base,inner sep=0pt, outer sep=0pt, scale=  
    1.00] at 
    ( 49.52, 67.72) {\footnotesize $3_l$};
    
    \path[draw=drawColor,line width= 0.4pt,line join=round,line 
    cap=round,fill=fillColor] (256.96, 70.93) --
    (256.95, 71.35) --
    (256.91, 71.77) --
    (256.84, 72.19) --
    (256.75, 72.60) --
    (256.63, 73.00) --
    (256.49, 73.40) --
    (256.32, 73.78) --
    (256.13, 74.16) --
    (255.92, 74.52) --
    (255.68, 74.86) --
    (255.42, 75.20) --
    (255.15, 75.51) --
    (254.85, 75.81) --
    (254.53, 76.09) --
    (254.20, 76.35) --
    (253.85, 76.58) --
    (253.49, 76.80) --
    (253.12, 76.99) --
    (252.73, 77.15) --
    (252.34, 77.30) --
    (251.93, 77.41) --
    (251.52, 77.50) --
    (251.11, 77.57) --
    (250.69, 77.61) --
    (250.27, 77.62) --
    (249.85, 77.61) --
    (249.43, 77.57) --
    (249.01, 77.50) --
    (248.60, 77.41) --
    (248.20, 77.30) --
    (247.80, 77.15) --
    (247.42, 76.99) --
    (247.04, 76.80) --
    (246.68, 76.58) --
    (246.34, 76.35) --
    (246.00, 76.09) --
    (245.69, 75.81) --
    (245.39, 75.51) --
    (245.11, 75.20) --
    (244.85, 74.86) --
    (244.62, 74.52) --
    (244.40, 74.16) --
    (244.21, 73.78) --
    (244.05, 73.40) --
    (243.90, 73.00) --
    (243.79, 72.60) --
    (243.70, 72.19) --
    (243.63, 71.77) --
    (243.59, 71.35) --
    (243.58, 70.93) --
    (243.59, 70.51) --
    (243.63, 70.09) --
    (243.70, 69.68) --
    (243.79, 69.27) --
    (243.90, 68.86) --
    (244.05, 68.47) --
    (244.21, 68.08) --
    (244.40, 67.71) --
    (244.62, 67.35) --
    (244.85, 67.00) --
    (245.11, 66.67) --
    (245.39, 66.35) --
    (245.69, 66.05) --
    (246.00, 65.78) --
    (246.34, 65.52) --
    (246.68, 65.28) --
    (247.04, 65.07) --
    (247.42, 64.88) --
    (247.80, 64.71) --
    (248.20, 64.57) --
    (248.60, 64.45) --
    (249.01, 64.36) --
    (249.43, 64.29) --
    (249.85, 64.25) --
    (250.27, 64.24) --
    (250.69, 64.25) --
    (251.11, 64.29) --
    (251.52, 64.36) --
    (251.93, 64.45) --
    (252.34, 64.57) --
    (252.73, 64.71) --
    (253.12, 64.88) --
    (253.49, 65.07) --
    (253.85, 65.28) --
    (254.20, 65.52) --
    (254.53, 65.78) --
    (254.85, 66.05) --
    (255.15, 66.35) --
    (255.42, 66.67) --
    (255.68, 67.00) --
    (255.92, 67.35) --
    (256.13, 67.71) --
    (256.32, 68.08) --
    (256.49, 68.47) --
    (256.63, 68.86) --
    (256.75, 69.27) --
    (256.84, 69.68) --
    (256.91, 70.09) --
    (256.95, 70.51) --
    cycle;
    
    \node[text=drawColor,anchor=base,inner sep=0pt, outer sep=0pt, scale=  
    1.00] at 
    (250.27, 67.72) {\footnotesize $3$};
    
    \path[draw=drawColor,line width= 0.4pt,line join=round,line 
    cap=round,fill=fillColor] (323.88, 70.93) --
    (323.86, 71.35) --
    (323.82, 71.77) --
    (323.76, 72.19) --
    (323.67, 72.60) --
    (323.55, 73.00) --
    (323.41, 73.40) --
    (323.24, 73.78) --
    (323.05, 74.16) --
    (322.83, 74.52) --
    (322.60, 74.86) --
    (322.34, 75.20) --
    (322.06, 75.51) --
    (321.77, 75.81) --
    (321.45, 76.09) --
    (321.12, 76.35) --
    (320.77, 76.58) --
    (320.41, 76.80) --
    (320.03, 76.99) --
    (319.65, 77.15) --
    (319.25, 77.30) --
    (318.85, 77.41) --
    (318.44, 77.50) --
    (318.02, 77.57) --
    (317.61, 77.61) --
    (317.19, 77.62) --
    (316.76, 77.61) --
    (316.35, 77.57) --
    (315.93, 77.50) --
    (315.52, 77.41) --
    (315.12, 77.30) --
    (314.72, 77.15) --
    (314.34, 76.99) --
    (313.96, 76.80) --
    (313.60, 76.58) --
    (313.25, 76.35) --
    (312.92, 76.09) --
    (312.60, 75.81) --
    (312.31, 75.51) --
    (312.03, 75.20) --
    (311.77, 74.86) --
    (311.54, 74.52) --
    (311.32, 74.16) --
    (311.13, 73.78) --
    (310.96, 73.40) --
    (310.82, 73.00) --
    (310.70, 72.60) --
    (310.61, 72.19) --
    (310.55, 71.77) --
    (310.51, 71.35) --
    (310.49, 70.93) --
    (310.51, 70.51) --
    (310.55, 70.09) --
    (310.61, 69.68) --
    (310.70, 69.27) --
    (310.82, 68.86) --
    (310.96, 68.47) --
    (311.13, 68.08) --
    (311.32, 67.71) --
    (311.54, 67.35) --
    (311.77, 67.00) --
    (312.03, 66.67) --
    (312.31, 66.35) --
    (312.60, 66.05) --
    (312.92, 65.78) --
    (313.25, 65.52) --
    (313.60, 65.28) --
    (313.96, 65.07) --
    (314.34, 64.88) --
    (314.72, 64.71) --
    (315.12, 64.57) --
    (315.52, 64.45) --
    (315.93, 64.36) --
    (316.35, 64.29) --
    (316.76, 64.25) --
    (317.19, 64.24) --
    (317.61, 64.25) --
    (318.02, 64.29) --
    (318.44, 64.36) --
    (318.85, 64.45) --
    (319.25, 64.57) --
    (319.65, 64.71) --
    (320.03, 64.88) --
    (320.41, 65.07) --
    (320.77, 65.28) --
    (321.12, 65.52) --
    (321.45, 65.78) --
    (321.77, 66.05) --
    (322.06, 66.35) --
    (322.34, 66.67) --
    (322.60, 67.00) --
    (322.83, 67.35) --
    (323.05, 67.71) --
    (323.24, 68.08) --
    (323.41, 68.47) --
    (323.55, 68.86) --
    (323.67, 69.27) --
    (323.76, 69.68) --
    (323.82, 70.09) --
    (323.86, 70.51) --
    cycle;
    
    \node[text=drawColor,anchor=base,inner sep=0pt, outer sep=0pt, scale=  
    1.00] at 
    (317.19, 67.72) {\footnotesize $3_r$};
    
    \path[draw=drawColor,line width= 0.4pt,line join=round,line 
    cap=round,fill=fillColor] ( 36.14, 70.93) --
    ( 36.12, 71.35) --
    ( 36.08, 71.77) --
    ( 36.02, 72.19) --
    ( 35.92, 72.60) --
    ( 35.81, 73.00) --
    ( 35.67, 73.40) --
    ( 35.50, 73.78) --
    ( 35.31, 74.16) --
    ( 35.09, 74.52) --
    ( 34.86, 74.86) --
    ( 34.60, 75.20) --
    ( 34.32, 75.51) --
    ( 34.02, 75.81) --
    ( 33.71, 76.09) --
    ( 33.38, 76.35) --
    ( 33.03, 76.58) --
    ( 32.67, 76.80) --
    ( 32.29, 76.99) --
    ( 31.91, 77.15) --
    ( 31.51, 77.30) --
    ( 31.11, 77.41) --
    ( 30.70, 77.50) --
    ( 30.28, 77.57) --
    ( 29.86, 77.61) --
    ( 29.44, 77.62) --
    ( 29.02, 77.61) --
    ( 28.60, 77.57) --
    ( 28.19, 77.50) --
    ( 27.78, 77.41) --
    ( 27.38, 77.30) --
    ( 26.98, 77.15) --
    ( 26.59, 76.99) --
    ( 26.22, 76.80) --
    ( 25.86, 76.58) --
    ( 25.51, 76.35) --
    ( 25.18, 76.09) --
    ( 24.86, 75.81) --
    ( 24.57, 75.51) --
    ( 24.29, 75.20) --
    ( 24.03, 74.86) --
    ( 23.79, 74.52) --
    ( 23.58, 74.16) --
    ( 23.39, 73.78) --
    ( 23.22, 73.40) --
    ( 23.08, 73.00) --
    ( 22.96, 72.60) --
    ( 22.87, 72.19) --
    ( 22.80, 71.77) --
    ( 22.76, 71.35) --
    ( 22.75, 70.93) --
    ( 22.76, 70.51) --
    ( 22.80, 70.09) --
    ( 22.87, 69.68) --
    ( 22.96, 69.27) --
    ( 23.08, 68.86) --
    ( 23.22, 68.47) --
    ( 23.39, 68.08) --
    ( 23.58, 67.71) --
    ( 23.79, 67.35) --
    ( 24.03, 67.00) --
    ( 24.29, 66.67) --
    ( 24.57, 66.35) --
    ( 24.86, 66.05) --
    ( 25.18, 65.78) --
    ( 25.51, 65.52) --
    ( 25.86, 65.28) --
    ( 26.22, 65.07) --
    ( 26.59, 64.88) --
    ( 26.98, 64.71) --
    ( 27.38, 64.57) --
    ( 27.78, 64.45) --
    ( 28.19, 64.36) --
    ( 28.60, 64.29) --
    ( 29.02, 64.25) --
    ( 29.44, 64.24) --
    ( 29.86, 64.25) --
    ( 30.28, 64.29) --
    ( 30.70, 64.36) --
    ( 31.11, 64.45) --
    ( 31.51, 64.57) --
    ( 31.91, 64.71) --
    ( 32.29, 64.88) --
    ( 32.67, 65.07) --
    ( 33.03, 65.28) --
    ( 33.38, 65.52) --
    ( 33.71, 65.78) --
    ( 34.02, 66.05) --
    ( 34.32, 66.35) --
    ( 34.60, 66.67) --
    ( 34.86, 67.00) --
    ( 35.09, 67.35) --
    ( 35.31, 67.71) --
    ( 35.50, 68.08) --
    ( 35.67, 68.47) --
    ( 35.81, 68.86) --
    ( 35.92, 69.27) --
    ( 36.02, 69.68) --
    ( 36.08, 70.09) --
    ( 36.12, 70.51) --
    cycle;
    
    \node[text=drawColor,anchor=base,inner sep=0pt, outer sep=0pt, scale=  
    1.00] at 
    ( 29.44, 67.72) {\footnotesize $4_l$};
    
    \path[draw=drawColor,line width= 0.4pt,line join=round,line 
    cap=round,fill=fillColor] (277.04, 70.93) --
    (277.02, 71.35) --
    (276.98, 71.77) --
    (276.92, 72.19) --
    (276.82, 72.60) --
    (276.71, 73.00) --
    (276.57, 73.40) --
    (276.40, 73.78) --
    (276.21, 74.16) --
    (275.99, 74.52) --
    (275.76, 74.86) --
    (275.50, 75.20) --
    (275.22, 75.51) --
    (274.92, 75.81) --
    (274.61, 76.09) --
    (274.28, 76.35) --
    (273.93, 76.58) --
    (273.57, 76.80) --
    (273.19, 76.99) --
    (272.81, 77.15) --
    (272.41, 77.30) --
    (272.01, 77.41) --
    (271.60, 77.50) --
    (271.18, 77.57) --
    (270.76, 77.61) --
    (270.34, 77.62) --
    (269.92, 77.61) --
    (269.50, 77.57) --
    (269.09, 77.50) --
    (268.68, 77.41) --
    (268.28, 77.30) --
    (267.88, 77.15) --
    (267.49, 76.99) --
    (267.12, 76.80) --
    (266.76, 76.58) --
    (266.41, 76.35) --
    (266.08, 76.09) --
    (265.76, 75.81) --
    (265.47, 75.51) --
    (265.19, 75.20) --
    (264.93, 74.86) --
    (264.69, 74.52) --
    (264.48, 74.16) --
    (264.29, 73.78) --
    (264.12, 73.40) --
    (263.98, 73.00) --
    (263.86, 72.60) --
    (263.77, 72.19) --
    (263.70, 71.77) --
    (263.66, 71.35) --
    (263.65, 70.93) --
    (263.66, 70.51) --
    (263.70, 70.09) --
    (263.77, 69.68) --
    (263.86, 69.27) --
    (263.98, 68.86) --
    (264.12, 68.47) --
    (264.29, 68.08) --
    (264.48, 67.71) --
    (264.69, 67.35) --
    (264.93, 67.00) --
    (265.19, 66.67) --
    (265.47, 66.35) --
    (265.76, 66.05) --
    (266.08, 65.78) --
    (266.41, 65.52) --
    (266.76, 65.28) --
    (267.12, 65.07) --
    (267.49, 64.88) --
    (267.88, 64.71) --
    (268.28, 64.57) --
    (268.68, 64.45) --
    (269.09, 64.36) --
    (269.50, 64.29) --
    (269.92, 64.25) --
    (270.34, 64.24) --
    (270.76, 64.25) --
    (271.18, 64.29) --
    (271.60, 64.36) --
    (272.01, 64.45) --
    (272.41, 64.57) --
    (272.81, 64.71) --
    (273.19, 64.88) --
    (273.57, 65.07) --
    (273.93, 65.28) --
    (274.28, 65.52) --
    (274.61, 65.78) --
    (274.92, 66.05) --
    (275.22, 66.35) --
    (275.50, 66.67) --
    (275.76, 67.00) --
    (275.99, 67.35) --
    (276.21, 67.71) --
    (276.40, 68.08) --
    (276.57, 68.47) --
    (276.71, 68.86) --
    (276.82, 69.27) --
    (276.92, 69.68) --
    (276.98, 70.09) --
    (277.02, 70.51) --
    cycle;
    
    \node[text=drawColor,anchor=base,inner sep=0pt, outer sep=0pt, scale=  
    1.00] at 
    (270.34, 67.72) {\footnotesize $4$};
    
    \path[draw=drawColor,line width= 0.4pt,line join=round,line 
    cap=round,fill=fillColor] (303.80, 70.93) --
    (303.79, 71.35) --
    (303.75, 71.77) --
    (303.68, 72.19) --
    (303.59, 72.60) --
    (303.47, 73.00) --
    (303.33, 73.40) --
    (303.16, 73.78) --
    (302.97, 74.16) --
    (302.76, 74.52) --
    (302.52, 74.86) --
    (302.27, 75.20) --
    (301.99, 75.51) --
    (301.69, 75.81) --
    (301.38, 76.09) --
    (301.04, 76.35) --
    (300.70, 76.58) --
    (300.33, 76.80) --
    (299.96, 76.99) --
    (299.57, 77.15) --
    (299.18, 77.30) --
    (298.77, 77.41) --
    (298.36, 77.50) --
    (297.95, 77.57) --
    (297.53, 77.61) --
    (297.11, 77.62) --
    (296.69, 77.61) --
    (296.27, 77.57) --
    (295.86, 77.50) --
    (295.45, 77.41) --
    (295.04, 77.30) --
    (294.65, 77.15) --
    (294.26, 76.99) --
    (293.89, 76.80) --
    (293.52, 76.58) --
    (293.18, 76.35) --
    (292.84, 76.09) --
    (292.53, 75.81) --
    (292.23, 75.51) --
    (291.95, 75.20) --
    (291.70, 74.86) --
    (291.46, 74.52) --
    (291.25, 74.16) --
    (291.06, 73.78) --
    (290.89, 73.40) --
    (290.75, 73.00) --
    (290.63, 72.60) --
    (290.54, 72.19) --
    (290.47, 71.77) --
    (290.43, 71.35) --
    (290.42, 70.93) --
    (290.43, 70.51) --
    (290.47, 70.09) --
    (290.54, 69.68) --
    (290.63, 69.27) --
    (290.75, 68.86) --
    (290.89, 68.47) --
    (291.06, 68.08) --
    (291.25, 67.71) --
    (291.46, 67.35) --
    (291.70, 67.00) --
    (291.95, 66.67) --
    (292.23, 66.35) --
    (292.53, 66.05) --
    (292.84, 65.78) --
    (293.18, 65.52) --
    (293.52, 65.28) --
    (293.89, 65.07) --
    (294.26, 64.88) --
    (294.65, 64.71) --
    (295.04, 64.57) --
    (295.45, 64.45) --
    (295.86, 64.36) --
    (296.27, 64.29) --
    (296.69, 64.25) --
    (297.11, 64.24) --
    (297.53, 64.25) --
    (297.95, 64.29) --
    (298.36, 64.36) --
    (298.77, 64.45) --
    (299.18, 64.57) --
    (299.57, 64.71) --
    (299.96, 64.88) --
    (300.33, 65.07) --
    (300.70, 65.28) --
    (301.04, 65.52) --
    (301.38, 65.78) --
    (301.69, 66.05) --
    (301.99, 66.35) --
    (302.27, 66.67) --
    (302.52, 67.00) --
    (302.76, 67.35) --
    (302.97, 67.71) --
    (303.16, 68.08) --
    (303.33, 68.47) --
    (303.47, 68.86) --
    (303.59, 69.27) --
    (303.68, 69.68) --
    (303.75, 70.09) --
    (303.79, 70.51) --
    cycle;
    
    \node[text=drawColor,anchor=base,inner sep=0pt, outer sep=0pt, scale=  
    1.00] at 
    (297.11, 67.72) {\footnotesize $4_r$};
    
    \path[draw=drawColor,line width= 0.4pt,dash pattern=on 4pt off 4pt ,line 
    join=round,line cap=round] ( 16.06, 48.63) --
    ( 16.06, 93.24);
    
    \path[draw=drawColor,line width= 0.4pt,dash pattern=on 4pt off 4pt ,line 
    join=round,line cap=round] (149.89, 48.63) --
    (149.89, 93.24);
    
    \path[draw=drawColor,line width= 0.4pt,dash pattern=on 4pt off 4pt ,line 
    join=round,line cap=round] (283.73, 48.63) --
    (283.73, 93.24);
    
    \path[draw=drawColor,line width= 0.4pt,dash pattern=on 4pt off 4pt ,line 
    join=round,line cap=round] (417.56, 48.63) --
    (417.56, 93.24);
    
    \node[text=drawColor,anchor=base,inner sep=0pt, outer sep=0pt, scale=  
    1.00] at ( 16.06, 34.27) {-1};
    
    \node[text=drawColor,anchor=base,inner sep=0pt, outer sep=0pt, scale=  
    1.00] at (149.89, 34.27) {0};
    
    \node[text=drawColor,anchor=base,inner sep=0pt, outer sep=0pt, scale=  
    1.00] at (283.73, 34.27) {1};
    
    \node[text=drawColor,anchor=base,inner sep=0pt, outer sep=0pt, scale=  
    1.00] at (417.56, 34.27) {2};
    \end{scope}
    \end{tikzpicture}
    \vspace{-1cm}
    \caption{Four particles with positions in $(0,1)$ and their reflections at 
    the left and right domain boundary (labelled $_l$ and $_r$ accordingly). 
    Particle 2, for example, gets repelled by the particles
    $1,3,4,1_l,2_r,3_r,4_r$.}
    \label{fig:particles_mirror}
\end{figure*}
In our set-up every particle moves within the open interval $(0,1)$ and has
an interaction radius of size $1$. Keeping this in mind, let us briefly examine 
the two main assumptions of the evolution.
First, there exist reflections for all particles at the left and right domain 
boundary. Secondly, the particles repel each other and -- furthermore -- get 
repelled by the reflections.
However, due to the limited viewing range, only one of the two reflections of a 
certain particle is considered at any given time, namely the one which is 
closer to the reference particle (see Figure \ref{fig:particles_mirror}). A 
special 
case occurs if the reference particle is located at position 0.5: the repulsive 
forces of both of its own reflections equal out.
%
% ------------------------------------------------------------------------------
%
\subsection{Discrete Variational Model}
\label{sec:discrete_variational_model}
We propose a dynamical system which has its roots in a spatial discretisation 
of the energy functional \eref{eq:energy}. Furthermore, we make use of a 
decreasing energy function 
$\varPsi: \mathbb{R}_0^+ \to \mathbb{R}$ and a global range constraint on $u$. 
The corresponding flux function $\varPhi$ is defined as $\varPhi(s) := 
\varPsi'(s^2) s$.\\
Our goal is to describe the evolution of one-dimensional -- not necessarily 
distinct -- particle positions $v_i \in (0,1)$, where $i = 1,\ldots,N$.
Therefore, we extend the position vector $\bm{v} = 
(v_1,\ldots,v_N)\transpose$ with the additional coordinates 
$v_{N+1},\ldots,v_{2N}$ defined as $v_{2N+1-i} := 2 - v_i \in (1,2)$. This 
extended position vector $\bm{v} \in (0,2)^{2N}$ allows to evaluate the energy 
function
\begin{equation}
\label{eq:model_energy}
E(\bm{v},\bm{W}) = \frac14 \cdot
\sum \limits_{i = 1}^{2N} \sum \limits_{j = 1}^{2N}
w_{i,j} \cdot
\varPsi ( (v_j - v_i)^2 ) ,
\end{equation}
which models the repulsion potential between all positions $v_i$ and $v_j$. The 
coefficient $w_{i,j}$ denotes entry $j$ in row $i$ of a 
constant non-negative weight matrix $\bm{W} = (w_{i,j}) \in 
(\mathbb{R}_0^+)^{2N \times 
2N}$. It models the importance of the
interaction between position $v_i$ and $v_j$.
All diagonal elements of the weight matrix are positive, i.e. $w_{i,i} 
> 0, \forall i \in \{1,2,\ldots,2N\}$.
In addition, we assume that the weights for all extended positions are the same 
as those for the original ones. Namely, we have
\begin{equation}
w_{i,j} =
w_{i, \, 2N+1-j} =
w_{2N+1-i, \, j} =
w_{2N+1-i, \, 2N+1-j}
\end{equation}
for $1 \leq i,j \leq N$.\\
For the penaliser function $\varPsi$ we impose several restrictions which we 
discuss subsequently. 
\Tref{tab:class_psi} shows one reasonable class of functions $\varPsi_{a,n}$ as 
well as the corresponding diffusivities $\varPsi'_{a,n}$ and flux functions 
$\varPhi_{a,n}$. In \Fref{fig:psi_dpsi_phi} we provide an illustration of three 
functions using $a = 1$ and $n = 1,2,3$.
\begin{table}[t]
\renewcommand{\arraystretch}{1.2}
\centering
\begin{tabular}{|p{0.15\textwidth}|p{0.15\textwidth}|p{0.15\textwidth}|}
\hline
$\varPsi_{a,n}(s^2)$ & $\varPsi'_{a,n}(s^2)$ & $\varPhi_{a,n}(s)$\\
\hline
$a \cdot \big( (s - 1)^{2n} - 1 \big)$ & 
$\frac{a \cdot n}{s} \cdot (s-1)^{2n -1}$ & 
$a \cdot n \cdot (s - 1)^{2n-1}$ \\
\hline
\end{tabular}
\caption{One exemplary class of penaliser functions $\varPsi(s^2)$ for $s \in 
[0,1]$ with $n \in \mathbb{N}$, $a > 0$ and corresponding diffusivity 
$\varPsi'(s^2)$ and flux $\varPhi(s)$ functions.}
\label{tab:class_psi}
\end{table}
\begin{figure*}
    \begin{tikzpicture}
    \begin{axis}[%
    axis lines = left,
    xlabel = $s$,
    ylabel = {$\tilde{\varPsi}(s)$},
    legend entries = {$\tilde{\varPsi}_{1,1}(s)$,
        $\tilde{\varPsi}_{1,2}(s)$,
        $\tilde{\varPsi}_{1,3}(s)$},
    legend pos = north east,
    width = 15cm,
    height = 6cm]
    \addlegendimage{no markers,red}
    \addlegendimage{no markers,green}
    \addlegendimage{no markers,blue}
    %
    % line at x = 0
    \addplot [
    dashed,
    samples=50,
    black
    ] 
    coordinates {(0,-1)(0,0)};
    %
    % line at x = 1
    \addplot [
    dashed,
    samples=50,
    black
    ] 
    coordinates {(1,-1)(1,0)};
    %
    % line at x = 2
    \addplot [
    dashed,
    samples=50,
    black
    ] 
    coordinates {(2,-1)(2,0)};
    %
    % Psi_a,n
    %
    % Psi_1,1 in [-1,0]
    \addplot [
    domain=-1:0, 
    samples=100, 
    color=red,
    ]
    {(x + 1)^2 - 1};
    % Psi_1,1 in [0,2]
    \addplot [
    domain=0:2, 
    samples=100, 
    color=red,
    ]
    {(x - 1)^2 - 1};
    % Psi_1,1 in [2,3]
    \addplot [
    domain=2:3, 
    samples=100, 
    color=red,
    ]
    {(x - 3)^2 - 1};
    %
    % Psi_1,2 in [-1,0]
    \addplot [
    domain=-1:0, 
    samples=100, 
    color=green,
    ]
    {(x + 1)^4 - 1};
    % Psi_1,2 in [0,2]
    \addplot [
    domain=0:2, 
    samples=100, 
    color=green,
    ]
    {(x - 1)^4 - 1};
    % Psi_1,2 in [2,3]
    \addplot [
    domain=2:3, 
    samples=100, 
    color=green,
    ]
    {(x - 3)^4 - 1};
    %
    % Psi_1,3 in [-1,0]
    \addplot [
    domain=-1:0, 
    samples=100, 
    color=blue,
    ]
    {(x + 1)^6 - 1};
    % Psi_1,3 in [0,2]
    \addplot [
    domain=0:2, 
    samples=100, 
    color=blue,
    ]
    {(x - 1)^6 - 1};
    % Psi_1,3 in [2,3]
    \addplot [
    domain=2:3, 
    samples=100, 
    color=blue,
    ]
    {(x - 3)^6 - 1};
    \end{axis}
    \end{tikzpicture}\\[5mm]
    %
    %-------------------------------------------------------------------------------
    %
    {
        \begin{tikzpicture}
        \begin{axis}[
        axis lines = left,
        xlabel = $s$,
        ylabel = {$\tilde{\varPsi}'(s)$},
        legend entries = {$\tilde{\varPsi}'_{1,1}(s)$,
            $\tilde{\varPsi}'_{1,2}(s)$,
            $\tilde{\varPsi}'_{1,3}(s)$},
        legend pos = north east,
        width = 15cm,
        height = 8cm,
        restrict y to domain=-100:100
        ]
        \addlegendimage{no markers,red}
        \addlegendimage{no markers,green}
        \addlegendimage{no markers,blue}
        %
        % line at x = 0
        \addplot [
        dashed,
        samples=50,
        black
        ] 
        coordinates {(0,-100)(0,100)};
        %
        % line at x = 1
        \addplot [
        dashed,
        samples=50,
        black
        ] 
        coordinates {(1,-100)(1,100)};
        %
        % line at x = 2
        \addplot [
        dashed,
        samples=50,
        black
        ] 
        coordinates {(2,-100)(2,100)};
        %
        % line at y = 0
        \addplot [
        solid,
        samples=50,
        black
        ] 
        coordinates {(-1,0)(3,0)};
        %
        % Psi'_a,n
        %
        % Psi'_1,1 in [-1,0]
        \addplot [
        domain=-1:0, 
        samples=100, 
        color=red,
        ]
        {-1 - 1/x};
        % Psi'_1,1 in [0,1]
        \addplot [
        domain=0:1, 
        samples=100, 
        color=red,
        ]
        {1 - 1/x};
        % Psi'_1,1 in [1,2]
        \addplot [
        domain=1:2, 
        samples=100, 
        color=red,
        ]
        {1/(2-x) - 1};
        % Psi'_1,1 in [2,3]
        \addplot [
        domain=2:3, 
        samples=100, 
        color=red,
        ]
        {1 - 1/(x-2)};
        %
        % Psi'_1,2 in [-1,0]
        \addplot [
        domain=-1:0, 
        samples=100, 
        color=green,
        ]
        {-2 - 2/x};
        % Psi'_1,2 in [0,1]
        \addplot [
        domain=0:1, 
        samples=100, 
        color=green,
        ]
        {2 - 2/x};
        % Psi'_1,2 in [1,2]
        \addplot [
        domain=1:2, 
        samples=100, 
        color=green,
        ]
        {2/(2-x) - 2};
        % Psi'_1,2 in [2,3]
        \addplot [
        domain=2:3, 
        samples=100, 
        color=green,
        ]
        {2 - 2/(x-2)};
        %
        % Psi'_1,3 in [-1,0]
        \addplot [
        domain=-1:0, 
        samples=100, 
        color=blue,
        ]
        {-3 - 3/x};
        % Psi'_1,3 in [0,1]
        \addplot [
        domain=0:1, 
        samples=100, 
        color=blue,
        ]
        {3 - 3/x};
        % Psi'_1,3 in [1,2]
        \addplot [
        domain=1:2, 
        samples=100, 
        color=blue,
        ]
        {3/(2-x) - 3};
        % Psi'_1,3 in [2,3]
        \addplot [
        domain=2:3, 
        samples=100, 
        color=blue,
        ]
        {3 - 3/(x-2)};
        \end{axis}
        \end{tikzpicture}\\[5mm]
    }
    \begin{tikzpicture}
    \begin{axis}[
    axis lines = left,
    xlabel = $s$,
    ylabel = {$\varPhi(s)$},
    legend entries = {$\varPhi_{1,1}(s)$,
        $\varPhi_{1,2}(s)$,
        $\varPhi_{1,3}(s)$},
    legend pos = north east,
    width = 15cm,
    height = 6cm
    ]
    \addlegendimage{no markers,red}
    \addlegendimage{no markers,green}
    \addlegendimage{no markers,blue}
    %
    % line at x = 0
    \addplot [
    dashed,
    samples=50,
    black
    ] 
    coordinates {(0,-3)(0,3)};
    %
    % line at x = 1
    \addplot [
    dashed,
    samples=50,
    black
    ] 
    coordinates {(1,-3)(1,3)};
    %
    % line at x = 2
    \addplot [
    dashed,
    samples=50,
    black
    ] 
    coordinates {(2,-3)(2,3)};
    %
    % line at y = 0
    \addplot [
    solid,
    samples=50,
    black
    ] 
    coordinates {(-1,0)(3,0)};
    %
    % Phi_a,n
    %
    % Phi_1,1 in [-1,0]
    \addplot [
    domain=-1:0, 
    samples=100, 
    color=red,
    ]
    {x + 1};
    % Phi_1,1 in [0,2]
    \addplot [
    domain=0:2, 
    samples=100, 
    color=red,
    ]
    {x - 1};
    % Phi_1,1 in [2,3]
    \addplot [
    domain=2:3, 
    samples=100, 
    color=red,
    ]
    {x - 3};
    %
    % Phi_1,2 in [-1,0]
    \addplot [
    domain=-1:0, 
    samples=100, 
    color=green,
    ]
    {2 * (x + 1)^3};
    % Phi_1,2 in [0,2]
    \addplot [
    domain=0:2, 
    samples=100, 
    color=green,
    ]
    {2 * (x - 1)^3};
    % Phi_1,2 in [2,3]
    \addplot [
    domain=2:3, 
    samples=100, 
    color=green,
    ]
    {2 * (x - 3)^3};
    %
    % Phi_1,3 in [-1,0]
    \addplot [
    domain=-1:0, 
    samples=100, 
    color=blue,
    ]
    {3 * (x + 1)^5};
    % Phi_1,3 in [0,2]
    \addplot [
    domain=0:2, 
    samples=100, 
    color=blue,
    ]
    {3 * (x - 1)^5};
    % Phi_1,3 in [2,3]
    \addplot [
    domain=2:3, 
    samples=100, 
    color=blue,
    ]
    {3 * (x - 3)^5};
    
    \end{axis}
    \end{tikzpicture}
    \caption{\textbf{Top:} Exemplary penaliser functions 
    $\tilde{\varPsi}_{1,1}$, $\tilde{\varPsi}_{1,2}$, and 
    $\tilde{\varPsi}_{1,3}$
    extended to the interval $[-1,3]$ by 
    imposing \textit{symmetry} and \textit{periodicity}
    with $\tilde{\varPsi}_{a,n}(s) := \varPsi_{a,n}(s^2)$.
    \textbf{Middle:} Corresponding diffusivities $\tilde{\varPsi}'_{1,1}$, 
    $\tilde{\varPsi}'_{1,2}$, and $\tilde{\varPsi}'_{1,3}$
    with $\tilde{\varPsi}'_{a,n}(s) := \varPsi'_{a,n}(s^2)$.
    \textbf{Bottom:} Corresponding flux functions $\varPhi_{1,1}$, 
    $\varPhi_{1,2}$, and $\varPhi_{1,3}$ with 
    $\varPhi_{a,n}(s)=\varPsi'_{a,n}(s^2)s$.}
    \label{fig:psi_dpsi_phi}
\end{figure*}
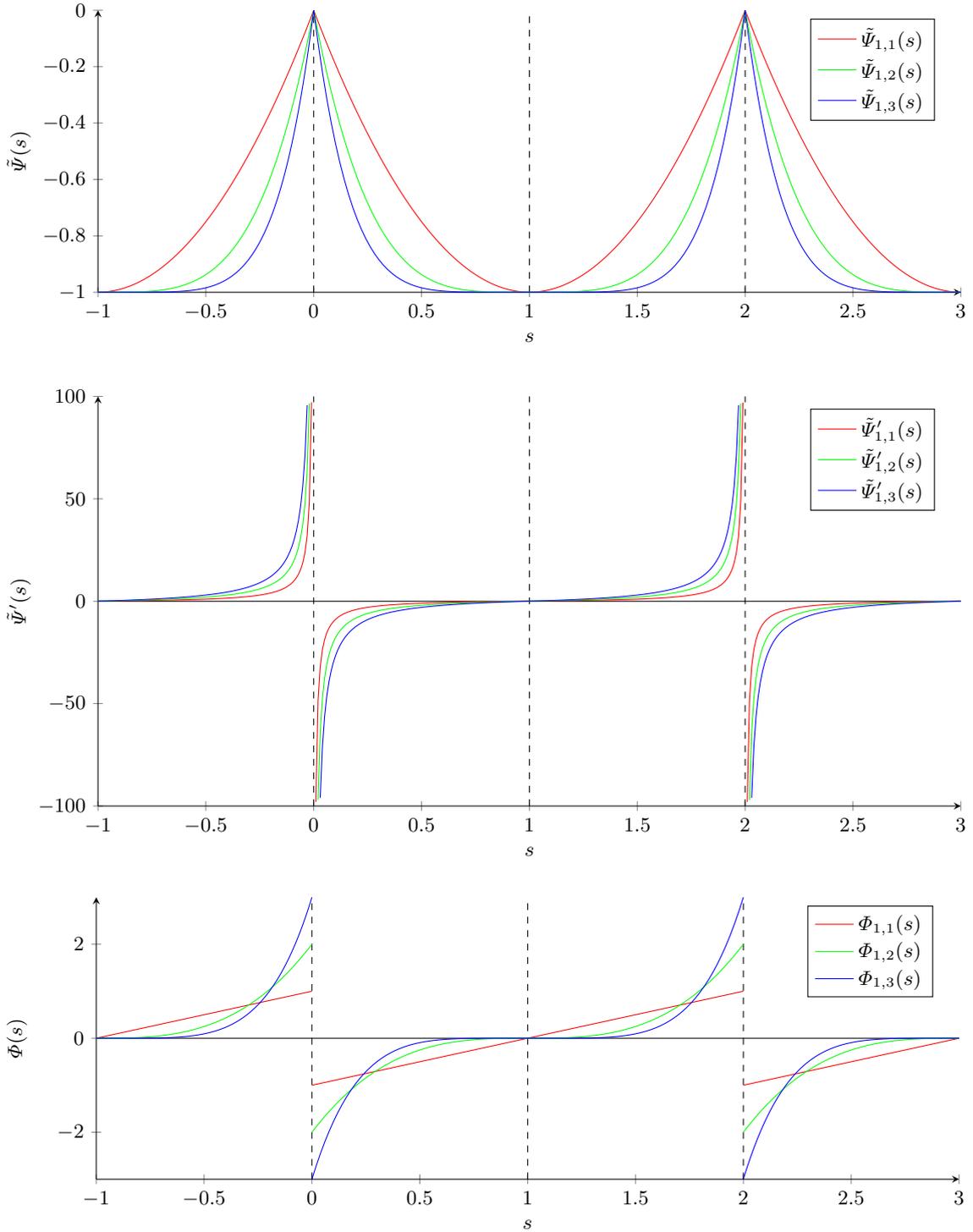
The penaliser is constructed following a three step procedure.
First, the function $\varPsi(s^2)$ is defined as a continuously 
differentiable, decreasing, and strictly convex function for $s \in [0,1]$ with 
$\varPsi(0)=0$ and
$\varPhi_-(1)=0$ (left-sided derivative).
Next, $\varPsi$ is extended to $[-1,1]$ by symmetry and to $\mathbb{R}$ by 
periodicity $\varPsi\bigl((2+s)^2\bigr)=\varPsi(s^2)$. 
This results in a penaliser $\varPsi(s^2)$ which is continuously differentiable 
everywhere except at even integers, where it is still continuous.
Note that $\varPsi(s^2)$ is increasing on $[-1,0]$ and $[1,2]$.
The flux $\varPhi$ is continuous and increasing in $(0,2)$ with
jump discontinuities at $0$ and $2$ (see \Fref{fig:psi_dpsi_phi}). 
Furthermore, we have that $\varPhi(s)=-\varPhi(-s)$ and 
$\varPhi(2+s)=\varPhi(s)$.
Exploiting the properties of $\varPsi$ allows us to rewrite 
\eref{eq:model_energy} 
without the redundant entries $v_{N+1},\ldots,v_{2N}$ as
\begin{equation}
\label{eq:model_energy_red}
%\begin{split}
E(\bm{v},\bm{W}) =
\frac12 \cdot
\sum \limits_{i = 1}^N \sum \limits_{j = 1}^N w_{i,j} \cdot
\bigg( \varPsi((v_j - v_i)^2) \, +
\varPsi((v_j + v_i)^2) \bigg) .
%\end{split}
\end{equation}
A gradient descent for \eref{eq:model_energy} is given by
\begin{equation}
\label{eq:model_gradient_descent}
%\begin{split}
\partial_t v_i = - \partial_{v_i} E(\bm{v},\bm{W})
= \sum \limits_{j \in J_1^i}
w_{i,j} \cdot \varPhi(v_j - v_i),
\quad i = 1,\ldots,2N ,
%\end{split}
\end{equation}
where $v_i$ now are functions of the time $t$ and
%$J_1^i := \{ j \in \{1,2,\ldots,2N\} \, \vert \, v_j \ne v_i \}$.
%
\begin{equation}
J_1^i := \{ j \in \{1,2,\ldots,2N\} \, \vert \, v_j \ne v_i \} .
\end{equation}
Note that for $1 \leq i,j \leq N$, thus $|v_j-v_i|<1$, 
the flux $\varPhi(v_j-v_i)$ is negative for $v_j > v_i$ and positive otherwise, 
thus driving $v_i$ always away from $v_j$. This implies that we have negative 
diffusivities $\varPsi'$ for all $|v_j-v_i| < 1$.
Due to the convexity of $\varPsi(s^2)$, the absolute values of the repulsive 
forces $\varPhi$ are decreasing with the distance between $v_i$ and $v_j$.
We remark that the jumps of $\varPhi$ at $0$ and $2$ are not problematic here,
as all positions $v_i$ and $v_j$ in the argument of $\varPhi$ are distinct by 
the definition of $J_1^i$.

Let us discuss shortly how the interval constraint for the $v_i$, 
$i=1,\ldots,N$, is enforced in \eref{eq:model_energy} and 
\eref{eq:model_gradient_descent}. First, notice that $v_{2N+1-i}$ for 
$i=1,\ldots,N$ is 
the reflection of $v_i$ on the right interval boundary $1$.
For $v_i$ and $v_{2N+1-j}$ with $1\le i,j\le N$ and $v_{2N+1-j}-v_i<1$ there is 
a repulsive force due to $\varPhi(v_{2N+1-j}-v_i)<0$ that drives $v_i$ and 
$v_{2N+1-j}$ away from the \emph{right} interval boundary. The closer $v_i$ and 
$v_{2N+1-j}$ come to this boundary, the stronger is the repulsion.
For $v_{2N+1-j}-v_i>1$, we have $\varPhi(v_{2N+1-j}-v_i)>0$.
By $\varPhi(v_{2N+1-j}-v_i) = \varPhi\bigl((2-v_j)-v_i\bigr) = 
\varPhi\bigl((-v_j)-v_i\bigr)$,
this can equally be interpreted as a repulsion between $v_i$ and $-v_j$ where 
$-v_j$ is the reflection of $v_j$ at the left interval boundary $0$.
In this case the interaction between $v_i$ and 
$v_{2N+1-j}$ drives $v_i$ and $-v_j$ away from the \emph{left} interval 
boundary.
Recapitulating both possible cases, it becomes clear that every $v_i$ is either 
repelled from the reflection of $v_j$ at the left or at the right interval 
boundary, but never from both at the same time.
As $\partial_tv_{2N+1-i}=-\partial_tv_i$ holds in 
\eref{eq:model_gradient_descent}, the symmetry of $\bm{v}$ is preserved.
Dropping the redundant entries $v_{N+1},\ldots,v_{2N}$,
\Eref{eq:model_gradient_descent} can be rewritten as
\begin{equation}
\label{eq:model_gradient_descent_redundant}
\partial_t v_i = 
\sum \limits_{j \in J_2^i}
w_{i,j} \cdot \varPhi(v_j - v_i) -
\sum \limits_{j = 1}^N w_{i,j} \cdot \varPhi(v_j + v_i) ,
\end{equation}
for $i = 1,\ldots,N$, where the second sum expresses
%emphasises
the repulsions between original and
reflected coordinates in a symmetric way.
The set $J_2^i$ is defined as
%$J_2^i := \{ j \in \{1,2,\ldots,N\} \, \vert \, v_j \ne v_i \}$.
%
\begin{equation}
J_2^i := \{ j \in \{1,2,\ldots,N\} \, \vert \, v_j \ne v_i \} .
\end{equation}
\Eref{eq:model_gradient_descent_redundant} denotes 
a formulation for pure repulsion amongst $N$ different positions $v_i$ with 
stabilisation being achieved through the consideration of their reflections at 
the domain boundary.
It is worth mentioning that within \eref{eq:model_energy_red} and 
\eref{eq:model_gradient_descent_redundant} we only make use of the first $N 
\times N$ entries of $\bm{W}$. In the following, we denote this submatrix by 
$\bm{\tilde{W}}$ and refer to its elements as $\tilde{w}_{i,j}$.
Given an initial vector $\boldsymbol{f}\in(0,1)^N$
and initialising $v_i(0)=f_i$, $v_{2N+1-i}(0)=2-f_i$ for $i=1,\ldots,N$,
the gradient descent \eref{eq:model_gradient_descent} and 
\eref{eq:model_gradient_descent_redundant}
evolves $\boldsymbol{v}$ towards a minimiser of $E$.
%
% ------------------------------------------------------------------------------
% ------------------------------------------------------------------------------
%
\section{Theory}
\label{sec:theory}
Below we provide a detailed analysis of the evolution and discuss its main 
properties. For this purpose we consider the Hessian of 
\eref{eq:model_energy_red} whose entries for $1 \leq i \leq N$ read
\begin{equation}
\label{eq:hessian_ii}
%\begin{split}
\partial_{v_iv_i} E(\bm{v},\bm{\tilde{W}}) =
\sum \limits_{\mathclap{j \in J_2^i}}
\tilde{w}_{i,j} \cdot \varPhi'(v_j-v_i)
+ \sum \limits_{j = 1}^N \tilde{w}_{i,j} \cdot \varPhi'(v_j+v_i) ,
%\end{split}
\end{equation}
\begin{equation}
\label{eq:hessian_ij}
\partial_{v_iv_j} E(\bm{v},\bm{\tilde{W}}) =
\begin{cases}
%\begin{split}
\tilde{w}_{i,j} \cdot \big( \varPhi'(v_j+v_i) \, -
\varPhi'(v_j-v_i) \big) ,
%\end{split}
 &
\forall j \in J_2^i ,\\[4mm]
\, \tilde{w}_{i,j} \cdot \varPhi'(v_j+v_i) , &
\forall j \in J_3^i ,
\end{cases}
\end{equation}
where
%$J_3^i := \{ j \in \{1,2,\ldots,N\} \, \vert \, v_i = v_j \}$.
%
\begin{equation}
J_3^i := \{ j \in \{1,2,\ldots,N\} \, \vert \, v_i = v_j \} .
\end{equation}
%
% ------------------------------------------------------------------------------
%
\subsection{General Results}
\label{sec:general_results}
In a first step, let us investigate the well-posedness of the underlying 
initial value problem in the sense of Hadamard \cite{Ha02}.
\begin{theorem}[Well-Posedness]
Let $\varPsi = \varPsi_{a,n}$ as defined in Table \ref{tab:class_psi}. Then the 
initial value problem \eqref{eq:model_gradient_descent_redundant} is well-posed 
since
\begin{enumerate}[{(}a{)}]
  \item it has a solution,
  \item the solution is unique, and
  \item it depends continuously on the initial conditions.
\end{enumerate}
\end{theorem}
\begin{proof}
The initial value problem \eqref{eq:model_gradient_descent_redundant} can be 
written as
\begin{align}
\bm{\dot{v}}(t) &
= \bm{f}(\bm{v}(t)) :=  -\bm{\nabla}_{\bm{v}} E(\bm{v}(t),\bm{W})\\
\bm{v}(0) & = \bm{v}_0
\end{align}
with $\bm{v}(t) \in \mathbb{R}^{2N}$ and $t \in \mathbb{R}_0^+$ where we 
make use of the fact that $\bm{W}$ is a constant weight matrix.\\

In case $\bm{f}(\bm{v}(t))$ is continuously differentiable and Lipschitz 
continuous all three conditions (a)--(c) hold. Existence and uniqueness 
directly 
follow from \cite[chapter 3.1, Theorem 3]{Pe01}.
Continuous dependence on the initial conditions is guaranteed due to 
\cite[chapter 2.3, Theorem 1]{Pe01} which is based on Gronwall's Lemma 
\cite{Gr19}.
Thus, let us now prove differentiability and Lipschitz continuity of 
$\bm{f}(\bm{v}(t))$.

\paragraph{Differentiability:}
Differentiability follows from the fact that all functions 
$\varPhi_{a,n}$ are continuously differentiable. As a consequence the partial 
derivatives of (8) w.r.t. $v_i$ exist for $i = 1,\ldots,2N$.

\paragraph{Lipschitz Continuity:}
The Gershgorin circle theorem \cite{Ge31} allows to estimate a valid Lipschitz 
constant $L$ as an upper bound of the spectral radius of the Jacobian of 
\eqref{eq:model_gradient_descent_redundant}.
For $1 \leq i \leq N$ the entries read
%-- consistent with \eqref{eq:hessian_ii} and \eqref{eq:hessian_ij} --
%
\begin{align}
%\nonumber
\partial_{v_i}(\partial_t v_i) = &
-\sum\limits_{j \in J_2^i}
\tilde{w}_{i,j} \cdot \big( \varPhi'(v_j-v_i) + \varPhi'(v_j+v_i) \big)%\\
%
%&
- 2 \cdot \sum \limits_{j \in J_3^i} \tilde{w}_{i,j} \cdot \varPhi'(v_j + 
v_i) ,\\
%\end{align}
%
%\begin{equation}
\partial_{v_j}(\partial_t v_i) = &
\begin{cases}
%\begin{split}
\tilde{w}_{i,j} \cdot \big( \varPhi'(v_j-v_i) - \varPhi'(v_j+v_i) \big) ,
%\end{split}
 & 
\forall j \in J_2^i ,\\[2mm]
-\tilde{w}_{i,j} \cdot \varPhi'(v_j + v_i) , & 
\forall j \in J_3^i .
\end{cases}
%\end{equation}
\end{align}
%\end{equation}\\
%
The radii of the Gershgorin discs fulfil
\begin{equation}
\begin{split}
r_i & =
\sum \limits_{\substack{j = 1\\j \neq i}}^N
\big| \partial_{v_j} (\partial_t v_i) \big|\\
&
%\begin{split}
= %&
\sum \limits_{j \in J_2^i}
\tilde{w}_{i,j} \cdot | \varPhi'(v_j - v_i) - \varPhi'(v_j + v_i) |
+ \sum \limits_{\substack{j \in J_3^i\\j \neq i}}
\tilde{w}_{i,j} \cdot |\varPhi'(v_j+v_i)|\\
%\end{split}\\
%
&
%\begin{split}
< \sum \limits_{j \in J_2^i}
\tilde{w}_{i,j} \cdot | \varPhi'(v_j - v_i) - \varPhi'(v_j + v_i) |
+ \sum \limits_{j \in J_3^i}
\tilde{w}_{i,j} \cdot |\varPhi'(v_j+v_i)|\\
%\end{split}\\
%
& =: \tilde{r}_i , \qquad i = 1,\ldots,N .
\end{split}
\end{equation}
Then we have $|\lambda_i - \partial_{v_i}(\partial_t v_i)| < \tilde{r}_i$ for 
$1 \leq i \leq N$ where $\lambda_i$ denotes the $i$-th eigenvalue of the 
Jacobian of \eqref{eq:model_gradient_descent_redundant}. This leads to the 
bounds
\begin{align}
\nonumber
\lambda_i & <
\begin{aligned}[t]
&
%\begin{aligned}[t]
\sum \limits_{j \in J_2^i} \tilde{w}_{i,j} \cdot \big(
|\varPhi'(v_j-v_i)-\varPhi'(v_j+v_i)| \, -
( \varPhi'(v_j-v_i) + \varPhi'(v_j+v_i)) \big)\\
%\end{aligned}\\[2mm]
%
& + \sum \limits_{j \in J_3^i}
\tilde{w}_{i,j} \cdot \big( |\varPhi'(v_j+v_i)| - 2 \cdot \varPhi'(v_j+v_i) 
\big)
\end{aligned}\\
& \leq
\begin{aligned}[t]
&
%\begin{aligned}[t]
 \sum \limits_{j \in J_2^i} \tilde{w}_{i,j} \cdot \big(
|\varPhi'(v_j-v_i)|+|\varPhi'(v_j+v_i)| \, +
|\varPhi'(v_j-v_i)| + |\varPhi'(v_j+v_i)| \big)\\
%\end{aligned}\\[2mm]
%
\nonumber
& +
\sum \limits_{j \in J_3^i}
\tilde{w}_{i,j} \cdot \big( |\varPhi'(v_j+v_i)| + 2 \cdot 
|\varPhi'(v_j+v_i)| \big)
\end{aligned}\\
\nonumber
& \leq
4 \cdot L_\varPhi \cdot \sum \limits_{j \in J_2^i} \tilde{w}_{i,j} +
3 \cdot L_\varPhi \cdot \sum \limits_{j \in J_3^i} \tilde{w}_{i,j}\\
& <
4 \cdot L_\varPhi \cdot \sum \limits_{j = 1}^N \tilde{w}_{i,j},
\qquad i = 1,\ldots,N,
\end{align}
where $L_{\varPhi}$ represents the Lipschitz constant of the flux function 
$\varPhi$.
Using the same reasoning one can show that
\begin{equation}
\lambda_i >
-4 \cdot L_\varPhi \cdot \sum \limits_{j= 1}^N \tilde{w}_{i,j},
\qquad i = 1,\ldots,N.
\end{equation}
Consequently, an upper bound for the spectral radius -- and thus for the 
Lipschitz constant $L$ of the gradient of 
\eqref{eq:model_gradient_descent_redundant} -- reads
\begin{equation}
\label{eq:lipschitz_max}
L \leq \max \limits_{1 \leq i \leq N}|\lambda_i| <
4 \cdot L_\varPhi \cdot \max \limits_{1 \leq i \leq N}
\sum \limits_{j = 1}^N \tilde{w}_{i,j}
=: L_{\mathrm{max}} .
\end{equation}
For our specific class of flux functions $\varPhi_{a,n}$ a valid 
Lipschitz constant $L_{\varPhi}$ is given by
\begin{equation}
L_{\varPhi} = a \cdot n \cdot (2n - 1) \cdot 2^{2n-2}
\end{equation}
such that we have
\begin{equation}
\label{eq:lipschitz_an}
L < 4 \cdot a \cdot n \cdot (2n-1) \cdot 2^{2n-2} \cdot
\max \limits_{1 \leq i \leq N}
\sum \limits_{j = 1}^N \tilde{w}_{i,j}.
\end{equation}
This concludes the proof.
\qed
\end{proof}
Next, let us show that no position $v_i$ can ever reach or cross 
the interval boundaries $0$ and $1$.
%
% ------------------------------------------------------------------------------
%
\begin{theorem}[Avoidance of Range Interval Boundaries]
\label{thm:boundaries}
For any weighting matrix $\bm{\tilde{W}} \in (\mathbb{R}_0^+)^{N \times N}$ all 
$N$ positions $v_i$ which evolve according to 
\eref{eq:model_gradient_descent_redundant} and have an arbitrary initial value 
in $(0,1)$ do not reach the domain boundaries 0 and 1 for any time $t \geq 0$.
\end{theorem}
\begin{proof}
\Eref{eq:model_gradient_descent_redundant} can be written as
\begin{equation}
\label{eq:model_gradient_descent_redundant_alternative}
%\begin{split}
\partial_tv_i = %&
\sum\limits_{j \in J_2^i}
\tilde{w}_{i,j} \cdot \bigg( \varPhi(v_j-v_i) - \varPhi(v_j+v_i) \bigg) \\
- \sum\limits_{j \in J_3^i}
\tilde{w}_{i,j} \cdot \varPhi(2v_i) ,
%\end{split}
\end{equation}
where $1 \leq i \leq N$.
Notice that for $j \in J_2^i$ we have
\begin{eqnarray}
\label{eq:thm_boundaries:02}
\lim\limits_{v_i \to 0^+} \varPhi(v_j-v_i) - \varPhi(v_j+v_i) & = & 
0\enspace,\\[2mm]
\label{eq:thm_boundaries:03}
\lim\limits_{v_i \to 1^-} \varPhi(v_j-v_i) - \varPhi(v_j+v_i) & = & 0\enspace,
\end{eqnarray}
where the latter follows from the periodicity of $\varPhi$.
Consequently, any position $v_i$ which gets arbitrarily close to one of the 
domain boundaries 0 or 1 experiences no impact by positions $v_j$ with $j \in 
J_2^i$, and the first sum in 
\eref{eq:model_gradient_descent_redundant_alternative} gets zero.
The definition of $\varPsi(s^2)$ implies that
\begin{align}
\varPsi'(s^2) < 0 , & \quad \forall s \in (0,1) ,\\
\varPsi'(s^2) > 0 , & \quad \forall s \in (1,2) ,
\end{align}
from which it follows for $1 \leq i \leq N$ that
\begin{align}
\label{eq:thm_boundaries:00}
- \varPhi(2v_i) > 0 , & \quad \forall v_i \in \left(0,\frac12\right) ,\\
\label{eq:thm_boundaries:01}
- \varPhi(2v_i) < 0 , & \quad \forall v_i \in \left(\frac12,1\right) .
\end{align}
Now remember that
$\bm{\tilde{W}} \in (\mathbb{R}_0^+)^{N \times N}$ and
$\tilde{w}_{i,i} > 0$. In combination with 
\eref{eq:thm_boundaries:00}
and \eref{eq:thm_boundaries:01}
we get
\begin{equation}
\lim\limits_{v_i \to 0^+} \partial_t v_i > 0
\quad \mbox{and} \quad
\lim\limits_{v_i \to 1^-} \partial_t v_i < 0\enspace,
\end{equation}
which concludes the proof.\qed
\end{proof}
%
% ------------------------------------------------------------------------------
%
Let us for a moment assume that the penaliser function is given by $\varPsi = 
\varPsi_{a,n}$ from \Tref{tab:class_psi}. Below, we prove that this implies 
convergence to the global minimum of the energy $E(\bm{v},\bm{\tilde{W}})$. 
\begin{theorem}[Convergence for $\varPsi = \varPsi_{a,n=1}$]
\label{thm:convergence_n_one}
For $t \to \infty$,
given a penaliser $\varPsi_{a,1}$ with arbitrary $a > 0$, any initial 
configuration $\bm{v} \in (0,1)^N$ converges to a unique 
steady state $\bm{v}^*$ which is the global minimiser of the energy given in 
\eref{eq:model_energy_red}.
\end{theorem}
\begin{proof}
As a sum of convex functions, \eref{eq:model_energy_red}
is convex. Therefore, the function
$V(\bm{v},\bm{\tilde{W}}) := E(\bm{v},\bm{\tilde{W}}) - 
E(\bm{v}^*,\bm{\tilde{W}})$ (where $\bm{v}^*$ 
is the equilibrium point) is a Lyapunov function with 
$V(\bm{v}^*,\bm{\tilde{W}}) = 0$ 
and $V(\bm{v},\bm{\tilde{W}}) > 0$ for all $\bm{v} \neq \bm{v}^*$.
Furthermore, we have
\begin{equation}
\label{eq:lyapunov_n_one}
\partial_t V(\bm{v},\bm{\tilde{W}}) =
-\sum\limits_{i=1}^{N} \big( \partial_{v_i} E(\bm{v},\bm{\tilde{W}}) \big)^2 
\leq 0 
\enspace.
\end{equation}
According to Gershgorin's theorem \cite{Ge31}, one can show that the Hessian 
matrix of \eref{eq:model_energy_red} is positive definite for $\varPsi = 
\varPsi_{a,1}$ 
from which it follows that $E(\bm{v},\bm{\tilde{W}})$ has a strict (global) 
minimum. This implies that the inequality in \eref{eq:lyapunov_n_one} becomes 
strict except in case of $\bm{v} = \bm{v}^*$, and guarantees asymptotic 
Lyapunov stability \cite{LY92} of $\bm{v}^*$. Thus, we have convergence to 
$\bm{v}^*$ for $t \to \infty$.\qed
\end{proof}
\begin{remark}
Theorem \ref{thm:convergence_n_one} can be extended to the case of $n=2$ and -- 
in a weaker formulation -- to arbitrary $n \in \mathbb{N}$. The proofs for 
both cases are based on a straightforward application of the Gershgorin circle 
theorem. 
For details we refer to the supplementary material.
\begin{enumerate}[{(}a{)}]
\item Given that $\varPsi = \varPsi_{a,n=2}$, let us assume that one of the 
following two conditions
\begin{itemize}
\item $v_i \neq \frac12$, or
\item there exists $j \in J_2^i$ for which $v_j \neq 1 - v_i$ and 
$\tilde{w}_{i,j} > 0$,
\end{itemize}
is fulfilled for every $i \in [1,N]$ and $t \geq 0$. Then the Hessian matrix 
of \eref{eq:model_energy_red} is 
positive definite and convergence to the strict global minimum of 
$E(\bm{v},\bm{\tilde{W}})$ follows.
\item For all penaliser functions $\varPsi = \varPsi_{a,n}$, one can show that 
the Hessian matrix of \eref{eq:model_energy_red} is positive 
se\-mi-de\-fi\-nite. This means that our method converges to a global minimum 
of $E(\bm{v},\bm{\tilde{W}})$. 
However, this minimum does not have to be unique.
\end{enumerate}
\end{remark}
In general, the steady-state solution of 
\eref{eq:model_gradient_descent_redundant} depends on the definition of the 
penaliser function $\varPsi$. Based on 
\eref{eq:model_gradient_descent_redundant_alternative}, and assuming that 
$\varPsi = \varPsi_{a,n}$, a minimiser of $E(\bm{v},\bm{\tilde{W}})$ 
necessarily fulfils the equation
\begin{equation}
\label{eq:steady_state_condition}
0 =  
\sum\limits_{j \in J_2^i}
\tilde{w}_{i,j} \cdot \big( (v_j^*-v_i^*-1)^{2n-1} - 
(v_j^*+v_i^*-1)^{2n-1} \big)
- \sum\limits_{j \in J_3^i}
\tilde{w}_{i,j} \cdot (2v_i^*-1)^{2n-1},
\end{equation}
where $i = 1,\ldots,N$.
%
% ------------------------------------------------------------------------------
%
\subsection{Global Model}
\label{sec:global_model}
If all positions $v_i$ interact with each other during the evolution, i.e. 
$\tilde{w}_{i,j} > 0$ for $1 \leq i,j \leq N$, we speak of our model as 
acting \emph{globally}.
Below, we prove the existence of weight matrices 
$\bm{\tilde{W}}$ for which distinct positions $v_i$ and $v_j$ (with $i \neq j$) 
can never become equal (assuming that the positions $v_i$, $i = 1, \ldots, N$, 
are distinct for $t = 0$). This implies that the initial 
rank-order of $v_i$ is preserved throughout the evolution.
\begin{theorem}[Distinctness of $v_i$ and $v_j$]
\label{thm:nonequality}
Among $N$ initially distinct positions $v_i \in (0,1)$ evolving according to 
\eref{eq:model_gradient_descent_redundant}, no two ever become equal if 
%$\tilde{w}_{j,i} + \tilde{w}_{i,j} > 0$ and 
$\tilde{w}_{j,k} = \tilde{w}_{i,k} > 0$ for $1 \leq i,j,k \leq N, \, i 
\neq j$.
\end{theorem}
\begin{proof}
Given $N$ distinct positions $v_i \in (0,1)$, equation
\eref{eq:model_gradient_descent_redundant} can be written as
\begin{equation}
\partial_tv_i =
\sum \limits_{\scriptstyle k = 1\atop \scriptstyle k \neq i}^N
\tilde{w}_{i,k} \cdot \varPhi(v_k-v_i)
- \sum \limits_{k = 1}^N
\tilde{w}_{i,k} \cdot \varPhi(v_k+v_i) ,
\end{equation}
for $i = 1,\ldots,N$.
We use this equation to derive the difference
\begin{equation}
\label{eq:thm_nonequality:00}
\begin{aligned}
\partial_t \left( v_j - v_i \right) = 
& \, (\tilde{w}_{j,i} + \tilde{w}_{i,j}) \cdot\varPhi(v_i - v_j)\\
& + \sum \limits_{\mathclap{\scriptstyle k = 1\atop \scriptstyle k \neq i,j}}^N
\bigg( \tilde{w}_{j,k} \cdot \varPhi(v_k - v_j) \, -
\tilde{w}_{i,k} \cdot \varPhi(v_k - v_i) \bigg)\\
& - \sum \limits_{k = 1}^N
\bigg( \tilde{w}_{j,k} \cdot \varPhi(v_k + v_j) \, -
\tilde{w}_{i,k} \cdot \varPhi(v_k + v_i) \bigg) ,
%\end{split}
\end{aligned}
\end{equation}
where $1 \leq i,j \leq N$. 
Assume w.l.o.g.\ that $v_j > v_i$ and consider \eref{eq:thm_nonequality:00}
in the limit $v_j-v_i\to0$. Then we have
\begin{equation}
\label{eq:thm_nonequality:01}
\lim\limits_{v_j-v_i\to0} (\tilde{w}_{j,i} + \tilde{w}_{i,j}) \cdot 
\varPhi(v_i - v_j) > 0 ,
\end{equation}
if $\tilde{w}_{j,i} + \tilde{w}_{i,j} > 0$,
which every global model fulfils by the assumption 
that $\tilde{w}_{i,j} > 0$ for $1 \leq i,j \leq N$.
This follows from the fact that $\varPhi(s) > 0$ for $s \in 
(-1,0)$. Furthermore, we have
\begin{align}
\nonumber
\lim\limits_{v_j-v_i\to0}
& \hphantom{-} \sum \limits_{\scriptstyle k = 1\atop \scriptstyle k \neq i,j}^N
\bigg( \tilde{w}_{j,k} \cdot \varPhi(v_k - v_j) -
\tilde{w}_{i,k} \cdot \varPhi(v_k - v_i) \bigg)\\
\nonumber
& - \sum \limits_{k = 1}^N \bigg( \tilde{w}_{j,k} \cdot \varPhi(v_k + v_j) 
- \tilde{w}_{i,k} \cdot \varPhi(v_k + v_i) \bigg)\\
\nonumber
= & \hphantom{-} 
\sum \limits_{\scriptstyle k = 1\atop \scriptstyle k \neq i,j}^N
(\tilde{w}_{j,k} - \tilde{w}_{i,k}) \cdot \varPhi(v_k - v_i) %\\
%\nonumber
%&
- \sum \limits_{k = 1}^N (\tilde{w}_{j,k} - \tilde{w}_{i,k}) \cdot 
\varPhi(v_k + v_i)\\
\label{eq:thm_nonequality:02}
= & \,\, 0 
\quad \text{if $\tilde{w}_{j,k} = \tilde{w}_{i,k}$ for $1 \leq k \leq 
N$} .
\end{align}
In conclusion, we can guarantee for global models with distinct particle 
positions that
\begin{equation}
\label{eq:thm_nonequality:03}
\lim\limits_{v_j-v_i\to0} \partial_t \left( v_j - v_i \right) > 0 ,
\end{equation}
if $\tilde{w}_{j,k} = \tilde{w}_{i,k}$ where $1 \leq i,j,k \leq N$ and $i \neq 
j$.
According to \eref{eq:thm_nonequality:03},
$v_j$ will always start moving away from $v_i$ (and 
vice versa) when the difference between both gets sufficiently small. Since the 
initial positions are distinct, it follows that $v_i \neq v_j$ for $i \neq j$ 
for all times $t$.\qed
\end{proof}
%
% ------------------------------------------------------------------------------
%

A special case occurs if all entries of the weight matrix $\bm{\tilde{W}}$ 
are set to $1$ -- i.e. 
$\bm{\tilde{W}} = \bm{1}\bm{1}\transpose$ with $\bm{1} := ( 1, \ldots, 
1)\transpose \in \mathbb{R}^N$.
For this scenario, we obtain an analytic steady-state solution which is 
independent of the penaliser $\varPsi$:
%
% ------------------------------------------------------------------------------
%
\begin{theorem}[Analytic Steady-State Solution for $\bm{\tilde{W}} = 
\mathbf{1}\mathbf{1}\transpose$]
\label{thm:solution_ones}
Under the assumption that $(v_i)$ is in increasing order, 
$\bm{\tilde{W}}=\mathbf{1}\mathbf{1}\transpose$, and that $\varPsi(s^2)$ is 
twice continuously differentiable in 
$(0,2)$ the unique minimiser of \eref{eq:model_energy} is given by 
$\boldsymbol{v}^*=(v_1^*,\ldots,v_{2N}^*)\transpose$,
$v_i^*=(i-\nicefrac12)/N$, $i=1,\ldots,2N$.
\end{theorem}
\begin{proof}
With $\bm{\tilde{W}}=\mathbf{1}\mathbf{1}\transpose$,  
\Eref{eq:model_energy} can be 
rewritten without the redundant entries of $\bm{v}$ as
\begin{equation}
\label{eq:model_energy_redundant}
%\begin{split}
E(\bm{v}) = 
\sum \limits_{i = 1}^{N - 1} \sum \limits_{j = i + 1}^N
\varPsi ( (v_j - v_i)^2 ) + \frac12 \cdot \sum \limits_{i = 1}^N \varPsi( 4 
v_i^2 )
+ \sum \limits_{i = 1}^{N - 1} \sum \limits_{j = i + 1}^N 
\varPsi ( (v_j + v_i)^2 ) .
%\end{split}
\end{equation}
%
%from which
From this, one can verify by straightforward, albeit lengthy calculations
that $\bm{\nabla}E(\bm{v}^*)=0$.
Moreover, one finds that the Hessian of $E$ at $\bm{v}^*$ is
\begin{equation}
\label{eq:hesse_energy}
\mathrm{D}^2E(\bm{v}^*) =
\sum \limits_{k=1}^{N} \bm{A}_k \varPhi' \left(\frac{k}{N}\right) .
\end{equation}
Here, $\bm{A}_k$ are sparse symmetric $N\times N$-matrices given by
\begin{align}
\label{eq:hesse_energy_coeff}
\bm{A}_k &
= 2 \bm{I} - \bm{T}_k - \bm{T}_{-k} + \bm{H}_{k+1} + \bm{H}_{2N-k+1} ,\\
\bm{A}_N & = \bm{I} + \bm{H}_{N+1} ,
\end{align}
for $k=1,\ldots,N-1$,
where the unit matrix $\bm{I}$, the single-diagonal Toeplitz matrices 
$\bm{T}_k$,  
and the single-an\-ti\-dia\-gonal Hankel matrices $\bm{H}_k$ are defined as
\begin{align}
\bm{I} & = \bigl(\delta_{i,j}\bigr)_{i,j=1}^N \enspace ,\\
\bm{T}_k & = \bigl(\delta_{j-i,k}\bigr)_{i,j=1}^N \enspace ,\\
\bm{H}_k & = \bigl(\delta_{i+j,k}\bigr)_{i,j=1}^N \enspace .
\end{align}
Here, $\delta_{i,j}$ denotes the Kronecker symbol, $\delta_{i,j}=1$ if
$i=j$, and $\delta_{i,j}=0$ otherwise. All $\bm{A}_k$, $k=1,\ldots,N$ are weakly
diagonally dominant with positive diagonal, thus positive semidefinite
by Gershgorin's Theorem. Moreover, the tridiagonal matrix $\bm{A}_1$ 
is of full rank, thus even positive definite. By strict convexity of 
$\varPsi(s^2)$, all $\varPhi'(\nicefrac{k}{N})$ are positive, thus 
$\mathrm{D}^2E(\bm{v}^*)$ is positive definite.
\par
As a consequence, the steady state of the gradient descent 
\eref{eq:model_gradient_descent_redundant} for any initial data $\bm{f}$ (with 
arbitrary rank-order) can -- under the condition that $\bm{\tilde{W}} = 
\bm{1}\bm{1}\transpose$ -- be computed directly by sorting the $f_i$:
Let $\sigma$ be the permutation of $\{1,\ldots,N\}$ for which 
$(f_{\sigma^{-1}(i)})_{i=1,\ldots,N}$ is increasing (this is what a sorting
algorithm computes), the steady state is given by
$v_i^*=(\sigma(i)-\nicefrac12)/N$ for $i=1,\ldots,N$ (cf. 
\Fref{fig:swarm-simple}).\qed
\end{proof}
\begin{figure*}[t]
    \centering
    \begin{tikzpicture}[x=1.2pt,y=1.2pt]
    \definecolor{fillColor}{RGB}{255,255,255}
    \path[use as bounding box,fill=fillColor,fill opacity=0.00] (0,0) rectangle 
    (325.21, 57.82);
    \begin{scope}
    \path[clip] (  0.00,  0.00) rectangle (325.21, 57.82);
    \definecolor{drawColor}{RGB}{0,0,0}
    
    \path[draw=drawColor,line width= 0.4pt,line join=round,line cap=round] (  
    0.00, 37.83) -- (325.21, 37.83);
    \definecolor{fillColor}{RGB}{255,255,255}
    
    \path[draw=drawColor,line width= 0.4pt,line join=round,line 
    cap=round,fill=fillColor] ( 47.51, 37.83) --
    ( 47.50, 38.21) --
    ( 47.46, 38.59) --
    ( 47.40, 38.96) --
    ( 47.32, 39.33) --
    ( 47.21, 39.69) --
    ( 47.09, 40.05) --
    ( 46.94, 40.39) --
    ( 46.76, 40.73) --
    ( 46.57, 41.06) --
    ( 46.36, 41.37) --
    ( 46.13, 41.67) --
    ( 45.88, 41.95) --
    ( 45.61, 42.22) --
    ( 45.33, 42.47) --
    ( 45.03, 42.70) --
    ( 44.71, 42.92) --
    ( 44.39, 43.11) --
    ( 44.05, 43.28) --
    ( 43.70, 43.43) --
    ( 43.35, 43.56) --
    ( 42.98, 43.66) --
    ( 42.61, 43.75) --
    ( 42.24, 43.81) --
    ( 41.86, 43.84) --
    ( 41.49, 43.85) --
    ( 41.11, 43.84) --
    ( 40.73, 43.81) --
    ( 40.36, 43.75) --
    ( 39.99, 43.66) --
    ( 39.63, 43.56) --
    ( 39.27, 43.43) --
    ( 38.92, 43.28) --
    ( 38.58, 43.11) --
    ( 38.26, 42.92) --
    ( 37.95, 42.70) --
    ( 37.65, 42.47) --
    ( 37.36, 42.22) --
    ( 37.10, 41.95) --
    ( 36.85, 41.67) --
    ( 36.61, 41.37) --
    ( 36.40, 41.06) --
    ( 36.21, 40.73) --
    ( 36.04, 40.39) --
    ( 35.89, 40.05) --
    ( 35.76, 39.69) --
    ( 35.65, 39.33) --
    ( 35.57, 38.96) --
    ( 35.51, 38.59) --
    ( 35.48, 38.21) --
    ( 35.46, 37.83) --
    ( 35.48, 37.45) --
    ( 35.51, 37.08) --
    ( 35.57, 36.70) --
    ( 35.65, 36.33) --
    ( 35.76, 35.97) --
    ( 35.89, 35.61) --
    ( 36.04, 35.27) --
    ( 36.21, 34.93) --
    ( 36.40, 34.60) --
    ( 36.61, 34.29) --
    ( 36.85, 33.99) --
    ( 37.10, 33.71) --
    ( 37.36, 33.44) --
    ( 37.65, 33.19) --
    ( 37.95, 32.96) --
    ( 38.26, 32.75) --
    ( 38.58, 32.55) --
    ( 38.92, 32.38) --
    ( 39.27, 32.23) --
    ( 39.63, 32.10) --
    ( 39.99, 32.00) --
    ( 40.36, 31.91) --
    ( 40.73, 31.86) --
    ( 41.11, 31.82) --
    ( 41.49, 31.81) --
    ( 41.86, 31.82) --
    ( 42.24, 31.86) --
    ( 42.61, 31.91) --
    ( 42.98, 32.00) --
    ( 43.35, 32.10) --
    ( 43.70, 32.23) --
    ( 44.05, 32.38) --
    ( 44.39, 32.55) --
    ( 44.71, 32.75) --
    ( 45.03, 32.96) --
    ( 45.33, 33.19) --
    ( 45.61, 33.44) --
    ( 45.88, 33.71) --
    ( 46.13, 33.99) --
    ( 46.36, 34.29) --
    ( 46.57, 34.60) --
    ( 46.76, 34.93) --
    ( 46.94, 35.27) --
    ( 47.09, 35.61) --
    ( 47.21, 35.97) --
    ( 47.32, 36.33) --
    ( 47.40, 36.70) --
    ( 47.46, 37.08) --
    ( 47.50, 37.45) --
    cycle;
    
    \node[text=drawColor,anchor=base,inner sep=0pt, outer sep=0pt, scale=  
    1.00] at ( 41.49, 34.62) {1};
    
    \path[draw=drawColor,line width= 0.4pt,line join=round,line 
    cap=round,fill=fillColor] ( 84.66, 37.83) --
    ( 84.64, 38.21) --
    ( 84.61, 38.59) --
    ( 84.55, 38.96) --
    ( 84.47, 39.33) --
    ( 84.36, 39.69) --
    ( 84.23, 40.05) --
    ( 84.08, 40.39) --
    ( 83.91, 40.73) --
    ( 83.72, 41.06) --
    ( 83.51, 41.37) --
    ( 83.27, 41.67) --
    ( 83.02, 41.95) --
    ( 82.76, 42.22) --
    ( 82.47, 42.47) --
    ( 82.17, 42.70) --
    ( 81.86, 42.92) --
    ( 81.53, 43.11) --
    ( 81.20, 43.28) --
    ( 80.85, 43.43) --
    ( 80.49, 43.56) --
    ( 80.13, 43.66) --
    ( 79.76, 43.75) --
    ( 79.39, 43.81) --
    ( 79.01, 43.84) --
    ( 78.63, 43.85) --
    ( 78.26, 43.84) --
    ( 77.88, 43.81) --
    ( 77.51, 43.75) --
    ( 77.14, 43.66) --
    ( 76.77, 43.56) --
    ( 76.42, 43.43) --
    ( 76.07, 43.28) --
    ( 75.73, 43.11) --
    ( 75.41, 42.92) --
    ( 75.09, 42.70) --
    ( 74.79, 42.47) --
    ( 74.51, 42.22) --
    ( 74.24, 41.95) --
    ( 73.99, 41.67) --
    ( 73.76, 41.37) --
    ( 73.55, 41.06) --
    ( 73.36, 40.73) --
    ( 73.18, 40.39) --
    ( 73.03, 40.05) --
    ( 72.91, 39.69) --
    ( 72.80, 39.33) --
    ( 72.72, 38.96) --
    ( 72.66, 38.59) --
    ( 72.62, 38.21) --
    ( 72.61, 37.83) --
    ( 72.62, 37.45) --
    ( 72.66, 37.08) --
    ( 72.72, 36.70) --
    ( 72.80, 36.33) --
    ( 72.91, 35.97) --
    ( 73.03, 35.61) --
    ( 73.18, 35.27) --
    ( 73.36, 34.93) --
    ( 73.55, 34.60) --
    ( 73.76, 34.29) --
    ( 73.99, 33.99) --
    ( 74.24, 33.71) --
    ( 74.51, 33.44) --
    ( 74.79, 33.19) --
    ( 75.09, 32.96) --
    ( 75.41, 32.75) --
    ( 75.73, 32.55) --
    ( 76.07, 32.38) --
    ( 76.42, 32.23) --
    ( 76.77, 32.10) --
    ( 77.14, 32.00) --
    ( 77.51, 31.91) --
    ( 77.88, 31.86) --
    ( 78.26, 31.82) --
    ( 78.63, 31.81) --
    ( 79.01, 31.82) --
    ( 79.39, 31.86) --
    ( 79.76, 31.91) --
    ( 80.13, 32.00) --
    ( 80.49, 32.10) --
    ( 80.85, 32.23) --
    ( 81.20, 32.38) --
    ( 81.53, 32.55) --
    ( 81.86, 32.75) --
    ( 82.17, 32.96) --
    ( 82.47, 33.19) --
    ( 82.76, 33.44) --
    ( 83.02, 33.71) --
    ( 83.27, 33.99) --
    ( 83.51, 34.29) --
    ( 83.72, 34.60) --
    ( 83.91, 34.93) --
    ( 84.08, 35.27) --
    ( 84.23, 35.61) --
    ( 84.36, 35.97) --
    ( 84.47, 36.33) --
    ( 84.55, 36.70) --
    ( 84.61, 37.08) --
    ( 84.64, 37.45) --
    cycle;
    
    \node[text=drawColor,anchor=base,inner sep=0pt, outer sep=0pt, scale=  
    1.00] at ( 78.63, 34.62) {2};
    
    \path[draw=drawColor,line width= 0.4pt,line join=round,line 
    cap=round,fill=fillColor] (123.50, 37.83) --
    (123.49, 38.21) --
    (123.45, 38.59) --
    (123.39, 38.96) --
    (123.31, 39.33) --
    (123.20, 39.69) --
    (123.08, 40.05) --
    (122.93, 40.39) --
    (122.75, 40.73) --
    (122.56, 41.06) --
    (122.35, 41.37) --
    (122.12, 41.67) --
    (121.87, 41.95) --
    (121.60, 42.22) --
    (121.32, 42.47) --
    (121.02, 42.70) --
    (120.70, 42.92) --
    (120.38, 43.11) --
    (120.04, 43.28) --
    (119.69, 43.43) --
    (119.34, 43.56) --
    (118.97, 43.66) --
    (118.61, 43.75) --
    (118.23, 43.81) --
    (117.86, 43.84) --
    (117.48, 43.85) --
    (117.10, 43.84) --
    (116.72, 43.81) --
    (116.35, 43.75) --
    (115.98, 43.66) --
    (115.62, 43.56) --
    (115.26, 43.43) --
    (114.91, 43.28) --
    (114.58, 43.11) --
    (114.25, 42.92) --
    (113.94, 42.70) --
    (113.64, 42.47) --
    (113.35, 42.22) --
    (113.09, 41.95) --
    (112.84, 41.67) --
    (112.60, 41.37) --
    (112.39, 41.06) --
    (112.20, 40.73) --
    (112.03, 40.39) --
    (111.88, 40.05) --
    (111.75, 39.69) --
    (111.64, 39.33) --
    (111.56, 38.96) --
    (111.50, 38.59) --
    (111.47, 38.21) --
    (111.45, 37.83) --
    (111.47, 37.45) --
    (111.50, 37.08) --
    (111.56, 36.70) --
    (111.64, 36.33) --
    (111.75, 35.97) --
    (111.88, 35.61) --
    (112.03, 35.27) --
    (112.20, 34.93) --
    (112.39, 34.60) --
    (112.60, 34.29) --
    (112.84, 33.99) --
    (113.09, 33.71) --
    (113.35, 33.44) --
    (113.64, 33.19) --
    (113.94, 32.96) --
    (114.25, 32.75) --
    (114.58, 32.55) --
    (114.91, 32.38) --
    (115.26, 32.23) --
    (115.62, 32.10) --
    (115.98, 32.00) --
    (116.35, 31.91) --
    (116.72, 31.86) --
    (117.10, 31.82) --
    (117.48, 31.81) --
    (117.86, 31.82) --
    (118.23, 31.86) --
    (118.61, 31.91) --
    (118.97, 32.00) --
    (119.34, 32.10) --
    (119.69, 32.23) --
    (120.04, 32.38) --
    (120.38, 32.55) --
    (120.70, 32.75) --
    (121.02, 32.96) --
    (121.32, 33.19) --
    (121.60, 33.44) --
    (121.87, 33.71) --
    (122.12, 33.99) --
    (122.35, 34.29) --
    (122.56, 34.60) --
    (122.75, 34.93) --
    (122.93, 35.27) --
    (123.08, 35.61) --
    (123.20, 35.97) --
    (123.31, 36.33) --
    (123.39, 36.70) --
    (123.45, 37.08) --
    (123.49, 37.45) --
    cycle;
    
    \node[text=drawColor,anchor=base,inner sep=0pt, outer sep=0pt, scale=  
    1.00] at (117.48, 34.62) {3};
    
    \path[draw=drawColor,line width= 0.4pt,line join=round,line 
    cap=round,fill=fillColor] (190.53, 37.83) --
    (190.52, 38.21) --
    (190.48, 38.59) --
    (190.43, 38.96) --
    (190.34, 39.33) --
    (190.24, 39.69) --
    (190.11, 40.05) --
    (189.96, 40.39) --
    (189.79, 40.73) --
    (189.59, 41.06) --
    (189.38, 41.37) --
    (189.15, 41.67) --
    (188.90, 41.95) --
    (188.63, 42.22) --
    (188.35, 42.47) --
    (188.05, 42.70) --
    (187.74, 42.92) --
    (187.41, 43.11) --
    (187.07, 43.28) --
    (186.73, 43.43) --
    (186.37, 43.56) --
    (186.01, 43.66) --
    (185.64, 43.75) --
    (185.26, 43.81) --
    (184.89, 43.84) --
    (184.51, 43.85) --
    (184.13, 43.84) --
    (183.75, 43.81) --
    (183.38, 43.75) --
    (183.01, 43.66) --
    (182.65, 43.56) --
    (182.29, 43.43) --
    (181.95, 43.28) --
    (181.61, 43.11) --
    (181.28, 42.92) --
    (180.97, 42.70) --
    (180.67, 42.47) --
    (180.39, 42.22) --
    (180.12, 41.95) --
    (179.87, 41.67) --
    (179.64, 41.37) --
    (179.42, 41.06) --
    (179.23, 40.73) --
    (179.06, 40.39) --
    (178.91, 40.05) --
    (178.78, 39.69) --
    (178.68, 39.33) --
    (178.59, 38.96) --
    (178.53, 38.59) --
    (178.50, 38.21) --
    (178.49, 37.83) --
    (178.50, 37.45) --
    (178.53, 37.08) --
    (178.59, 36.70) --
    (178.68, 36.33) --
    (178.78, 35.97) --
    (178.91, 35.61) --
    (179.06, 35.27) --
    (179.23, 34.93) --
    (179.42, 34.60) --
    (179.64, 34.29) --
    (179.87, 33.99) --
    (180.12, 33.71) --
    (180.39, 33.44) --
    (180.67, 33.19) --
    (180.97, 32.96) --
    (181.28, 32.75) --
    (181.61, 32.55) --
    (181.95, 32.38) --
    (182.29, 32.23) --
    (182.65, 32.10) --
    (183.01, 32.00) --
    (183.38, 31.91) --
    (183.75, 31.86) --
    (184.13, 31.82) --
    (184.51, 31.81) --
    (184.89, 31.82) --
    (185.26, 31.86) --
    (185.64, 31.91) --
    (186.01, 32.00) --
    (186.37, 32.10) --
    (186.73, 32.23) --
    (187.07, 32.38) --
    (187.41, 32.55) --
    (187.74, 32.75) --
    (188.05, 32.96) --
    (188.35, 33.19) --
    (188.63, 33.44) --
    (188.90, 33.71) --
    (189.15, 33.99) --
    (189.38, 34.29) --
    (189.59, 34.60) --
    (189.79, 34.93) --
    (189.96, 35.27) --
    (190.11, 35.61) --
    (190.24, 35.97) --
    (190.34, 36.33) --
    (190.43, 36.70) --
    (190.48, 37.08) --
    (190.52, 37.45) --
    cycle;
    
    \node[text=drawColor,anchor=base,inner sep=0pt, outer sep=0pt, scale=  
    1.00] at (184.51, 34.62) {4};
    
    \path[draw=drawColor,line width= 0.4pt,line join=round,line 
    cap=round,fill=fillColor] (209.53, 37.83) --
    (209.51, 38.21) --
    (209.48, 38.59) --
    (209.42, 38.96) --
    (209.34, 39.33) --
    (209.23, 39.69) --
    (209.10, 40.05) --
    (208.95, 40.39) --
    (208.78, 40.73) --
    (208.59, 41.06) --
    (208.38, 41.37) --
    (208.14, 41.67) --
    (207.89, 41.95) --
    (207.63, 42.22) --
    (207.34, 42.47) --
    (207.04, 42.70) --
    (206.73, 42.92) --
    (206.40, 43.11) --
    (206.07, 43.28) --
    (205.72, 43.43) --
    (205.36, 43.56) --
    (205.00, 43.66) --
    (204.63, 43.75) --
    (204.26, 43.81) --
    (203.88, 43.84) --
    (203.50, 43.85) --
    (203.13, 43.84) --
    (202.75, 43.81) --
    (202.37, 43.75) --
    (202.01, 43.66) --
    (201.64, 43.56) --
    (201.29, 43.43) --
    (200.94, 43.28) --
    (200.60, 43.11) --
    (200.28, 42.92) --
    (199.96, 42.70) --
    (199.66, 42.47) --
    (199.38, 42.22) --
    (199.11, 41.95) --
    (198.86, 41.67) --
    (198.63, 41.37) --
    (198.42, 41.06) --
    (198.23, 40.73) --
    (198.05, 40.39) --
    (197.90, 40.05) --
    (197.78, 39.69) --
    (197.67, 39.33) --
    (197.59, 38.96) --
    (197.53, 38.59) --
    (197.49, 38.21) --
    (197.48, 37.83) --
    (197.49, 37.45) --
    (197.53, 37.08) --
    (197.59, 36.70) --
    (197.67, 36.33) --
    (197.78, 35.97) --
    (197.90, 35.61) --
    (198.05, 35.27) --
    (198.23, 34.93) --
    (198.42, 34.60) --
    (198.63, 34.29) --
    (198.86, 33.99) --
    (199.11, 33.71) --
    (199.38, 33.44) --
    (199.66, 33.19) --
    (199.96, 32.96) --
    (200.28, 32.75) --
    (200.60, 32.55) --
    (200.94, 32.38) --
    (201.29, 32.23) --
    (201.64, 32.10) --
    (202.01, 32.00) --
    (202.37, 31.91) --
    (202.75, 31.86) --
    (203.13, 31.82) --
    (203.50, 31.81) --
    (203.88, 31.82) --
    (204.26, 31.86) --
    (204.63, 31.91) --
    (205.00, 32.00) --
    (205.36, 32.10) --
    (205.72, 32.23) --
    (206.07, 32.38) --
    (206.40, 32.55) --
    (206.73, 32.75) --
    (207.04, 32.96) --
    (207.34, 33.19) --
    (207.63, 33.44) --
    (207.89, 33.71) --
    (208.14, 33.99) --
    (208.38, 34.29) --
    (208.59, 34.60) --
    (208.78, 34.93) --
    (208.95, 35.27) --
    (209.10, 35.61) --
    (209.23, 35.97) --
    (209.34, 36.33) --
    (209.42, 36.70) --
    (209.48, 37.08) --
    (209.51, 37.45) --
    cycle;
    
    \node[text=drawColor,anchor=base,inner sep=0pt, outer sep=0pt, scale=  
    1.00] at (203.50, 34.62) {5};
    
    \path[draw=drawColor,line width= 0.4pt,line join=round,line 
    cap=round,fill=fillColor] (229.29, 37.83) --
    (229.27, 38.21) --
    (229.24, 38.59) --
    (229.18, 38.96) --
    (229.10, 39.33) --
    (228.99, 39.69) --
    (228.86, 40.05) --
    (228.71, 40.39) --
    (228.54, 40.73) --
    (228.35, 41.06) --
    (228.14, 41.37) --
    (227.90, 41.67) --
    (227.65, 41.95) --
    (227.39, 42.22) --
    (227.10, 42.47) --
    (226.80, 42.70) --
    (226.49, 42.92) --
    (226.17, 43.11) --
    (225.83, 43.28) --
    (225.48, 43.43) --
    (225.13, 43.56) --
    (224.76, 43.66) --
    (224.39, 43.75) --
    (224.02, 43.81) --
    (223.64, 43.84) --
    (223.26, 43.85) --
    (222.89, 43.84) --
    (222.51, 43.81) --
    (222.14, 43.75) --
    (221.77, 43.66) --
    (221.40, 43.56) --
    (221.05, 43.43) --
    (220.70, 43.28) --
    (220.36, 43.11) --
    (220.04, 42.92) --
    (219.72, 42.70) --
    (219.43, 42.47) --
    (219.14, 42.22) --
    (218.87, 41.95) --
    (218.62, 41.67) --
    (218.39, 41.37) --
    (218.18, 41.06) --
    (217.99, 40.73) --
    (217.81, 40.39) --
    (217.66, 40.05) --
    (217.54, 39.69) --
    (217.43, 39.33) --
    (217.35, 38.96) --
    (217.29, 38.59) --
    (217.25, 38.21) --
    (217.24, 37.83) --
    (217.25, 37.45) --
    (217.29, 37.08) --
    (217.35, 36.70) --
    (217.43, 36.33) --
    (217.54, 35.97) --
    (217.66, 35.61) --
    (217.81, 35.27) --
    (217.99, 34.93) --
    (218.18, 34.60) --
    (218.39, 34.29) --
    (218.62, 33.99) --
    (218.87, 33.71) --
    (219.14, 33.44) --
    (219.43, 33.19) --
    (219.72, 32.96) --
    (220.04, 32.75) --
    (220.36, 32.55) --
    (220.70, 32.38) --
    (221.05, 32.23) --
    (221.40, 32.10) --
    (221.77, 32.00) --
    (222.14, 31.91) --
    (222.51, 31.86) --
    (222.89, 31.82) --
    (223.26, 31.81) --
    (223.64, 31.82) --
    (224.02, 31.86) --
    (224.39, 31.91) --
    (224.76, 32.00) --
    (225.13, 32.10) --
    (225.48, 32.23) --
    (225.83, 32.38) --
    (226.17, 32.55) --
    (226.49, 32.75) --
    (226.80, 32.96) --
    (227.10, 33.19) --
    (227.39, 33.44) --
    (227.65, 33.71) --
    (227.90, 33.99) --
    (228.14, 34.29) --
    (228.35, 34.60) --
    (228.54, 34.93) --
    (228.71, 35.27) --
    (228.86, 35.61) --
    (228.99, 35.97) --
    (229.10, 36.33) --
    (229.18, 36.70) --
    (229.24, 37.08) --
    (229.27, 37.45) --
    cycle;
    
    \node[text=drawColor,anchor=base,inner sep=0pt, outer sep=0pt, scale=  
    1.00] at (223.26, 34.62) {6};
    
    \path[draw=drawColor,line width= 0.4pt,line join=round,line 
    cap=round,fill=fillColor] (295.57, 37.83) --
    (295.56, 38.21) --
    (295.53, 38.59) --
    (295.47, 38.96) --
    (295.38, 39.33) --
    (295.28, 39.69) --
    (295.15, 40.05) --
    (295.00, 40.39) --
    (294.83, 40.73) --
    (294.64, 41.06) --
    (294.42, 41.37) --
    (294.19, 41.67) --
    (293.94, 41.95) --
    (293.67, 42.22) --
    (293.39, 42.47) --
    (293.09, 42.70) --
    (292.78, 42.92) --
    (292.45, 43.11) --
    (292.11, 43.28) --
    (291.77, 43.43) --
    (291.41, 43.56) --
    (291.05, 43.66) --
    (290.68, 43.75) --
    (290.31, 43.81) --
    (289.93, 43.84) --
    (289.55, 43.85) --
    (289.17, 43.84) --
    (288.80, 43.81) --
    (288.42, 43.75) --
    (288.05, 43.66) --
    (287.69, 43.56) --
    (287.33, 43.43) --
    (286.99, 43.28) --
    (286.65, 43.11) --
    (286.32, 42.92) --
    (286.01, 42.70) --
    (285.71, 42.47) --
    (285.43, 42.22) --
    (285.16, 41.95) --
    (284.91, 41.67) --
    (284.68, 41.37) --
    (284.47, 41.06) --
    (284.27, 40.73) --
    (284.10, 40.39) --
    (283.95, 40.05) --
    (283.82, 39.69) --
    (283.72, 39.33) --
    (283.63, 38.96) --
    (283.58, 38.59) --
    (283.54, 38.21) --
    (283.53, 37.83) --
    (283.54, 37.45) --
    (283.58, 37.08) --
    (283.63, 36.70) --
    (283.72, 36.33) --
    (283.82, 35.97) --
    (283.95, 35.61) --
    (284.10, 35.27) --
    (284.27, 34.93) --
    (284.47, 34.60) --
    (284.68, 34.29) --
    (284.91, 33.99) --
    (285.16, 33.71) --
    (285.43, 33.44) --
    (285.71, 33.19) --
    (286.01, 32.96) --
    (286.32, 32.75) --
    (286.65, 32.55) --
    (286.99, 32.38) --
    (287.33, 32.23) --
    (287.69, 32.10) --
    (288.05, 32.00) --
    (288.42, 31.91) --
    (288.80, 31.86) --
    (289.17, 31.82) --
    (289.55, 31.81) --
    (289.93, 31.82) --
    (290.31, 31.86) --
    (290.68, 31.91) --
    (291.05, 32.00) --
    (291.41, 32.10) --
    (291.77, 32.23) --
    (292.11, 32.38) --
    (292.45, 32.55) --
    (292.78, 32.75) --
    (293.09, 32.96) --
    (293.39, 33.19) --
    (293.67, 33.44) --
    (293.94, 33.71) --
    (294.19, 33.99) --
    (294.42, 34.29) --
    (294.64, 34.60) --
    (294.83, 34.93) --
    (295.00, 35.27) --
    (295.15, 35.61) --
    (295.28, 35.97) --
    (295.38, 36.33) --
    (295.47, 36.70) --
    (295.53, 37.08) --
    (295.56, 37.45) --
    cycle;
    
    \node[text=drawColor,anchor=base,inner sep=0pt, outer sep=0pt, scale=  
    1.00] at (289.55, 34.62) {7};
    
    \path[draw=drawColor,line width= 0.4pt,dash pattern=on 4pt off 4pt ,line 
    join=round,line cap=round] ( 12.04, 25.93) --
    ( 12.04, 49.73);
    
    \path[draw=drawColor,line width= 0.4pt,dash pattern=on 4pt off 4pt ,line 
    join=round,line cap=round] (313.17, 25.93) --
    (313.17, 49.73);
    
    \node[text=drawColor,anchor=base,inner sep=0pt, outer sep=0pt, scale=  
    1.00] at 
    ( 12.04, 16.78) {0};
    
    \node[text=drawColor,anchor=base,inner sep=0pt, outer sep=0pt, scale=  
    1.00] at 
    (313.17, 16.78) {1};
    \end{scope}
    \end{tikzpicture}\\
    \vspace{-10mm}
    {\bf initial state}\\
    \begin{tikzpicture}[x=1.2pt,y=1.2pt]
    \definecolor{fillColor}{RGB}{255,255,255}
    \path[use as bounding box,fill=fillColor,fill opacity=0.00] (0,0) rectangle 
    (325.21, 57.82);
    \begin{scope}
    \path[clip] (  0.00,  0.00) rectangle (325.21, 57.82);
    \definecolor{drawColor}{RGB}{0,0,0}
    
    \path[draw=drawColor,line width= 0.4pt,line join=round,line cap=round] (  
    0.00, 37.83) -- (325.21, 37.83);
    \definecolor{fillColor}{RGB}{255,255,255}
    
    \path[draw=drawColor,line width= 0.4pt,line join=round,line 
    cap=round,fill=fillColor] ( 39.58, 37.83) --
    ( 39.56, 38.21) --
    ( 39.53, 38.59) --
    ( 39.47, 38.96) --
    ( 39.39, 39.33) --
    ( 39.28, 39.69) --
    ( 39.15, 40.05) --
    ( 39.00, 40.39) --
    ( 38.83, 40.73) --
    ( 38.64, 41.06) --
    ( 38.43, 41.37) --
    ( 38.19, 41.67) --
    ( 37.94, 41.95) --
    ( 37.68, 42.22) --
    ( 37.39, 42.47) --
    ( 37.09, 42.70) --
    ( 36.78, 42.92) --
    ( 36.46, 43.11) --
    ( 36.12, 43.28) --
    ( 35.77, 43.43) --
    ( 35.41, 43.56) --
    ( 35.05, 43.66) --
    ( 34.68, 43.75) --
    ( 34.31, 43.81) --
    ( 33.93, 43.84) --
    ( 33.55, 43.85) --
    ( 33.18, 43.84) --
    ( 32.80, 43.81) --
    ( 32.43, 43.75) --
    ( 32.06, 43.66) --
    ( 31.69, 43.56) --
    ( 31.34, 43.43) --
    ( 30.99, 43.28) --
    ( 30.65, 43.11) --
    ( 30.33, 42.92) --
    ( 30.01, 42.70) --
    ( 29.72, 42.47) --
    ( 29.43, 42.22) --
    ( 29.16, 41.95) --
    ( 28.91, 41.67) --
    ( 28.68, 41.37) --
    ( 28.47, 41.06) --
    ( 28.28, 40.73) --
    ( 28.10, 40.39) --
    ( 27.95, 40.05) --
    ( 27.83, 39.69) --
    ( 27.72, 39.33) --
    ( 27.64, 38.96) --
    ( 27.58, 38.59) --
    ( 27.54, 38.21) --
    ( 27.53, 37.83) --
    ( 27.54, 37.45) --
    ( 27.58, 37.08) --
    ( 27.64, 36.70) --
    ( 27.72, 36.33) --
    ( 27.83, 35.97) --
    ( 27.95, 35.61) --
    ( 28.10, 35.27) --
    ( 28.28, 34.93) --
    ( 28.47, 34.60) --
    ( 28.68, 34.29) --
    ( 28.91, 33.99) --
    ( 29.16, 33.71) --
    ( 29.43, 33.44) --
    ( 29.72, 33.19) --
    ( 30.01, 32.96) --
    ( 30.33, 32.75) --
    ( 30.65, 32.55) --
    ( 30.99, 32.38) --
    ( 31.34, 32.23) --
    ( 31.69, 32.10) --
    ( 32.06, 32.00) --
    ( 32.43, 31.91) --
    ( 32.80, 31.86) --
    ( 33.18, 31.82) --
    ( 33.55, 31.81) --
    ( 33.93, 31.82) --
    ( 34.31, 31.86) --
    ( 34.68, 31.91) --
    ( 35.05, 32.00) --
    ( 35.41, 32.10) --
    ( 35.77, 32.23) --
    ( 36.12, 32.38) --
    ( 36.46, 32.55) --
    ( 36.78, 32.75) --
    ( 37.09, 32.96) --
    ( 37.39, 33.19) --
    ( 37.68, 33.44) --
    ( 37.94, 33.71) --
    ( 38.19, 33.99) --
    ( 38.43, 34.29) --
    ( 38.64, 34.60) --
    ( 38.83, 34.93) --
    ( 39.00, 35.27) --
    ( 39.15, 35.61) --
    ( 39.28, 35.97) --
    ( 39.39, 36.33) --
    ( 39.47, 36.70) --
    ( 39.53, 37.08) --
    ( 39.56, 37.45) --
    cycle;
    
    \node[text=drawColor,anchor=base,inner sep=0pt, outer sep=0pt, scale=  
    1.00] at ( 33.55, 34.62) {1};
    
    \path[draw=drawColor,line width= 0.4pt,line join=round,line 
    cap=round,fill=fillColor] ( 82.59, 37.83) --
    ( 82.58, 38.21) --
    ( 82.55, 38.59) --
    ( 82.49, 38.96) --
    ( 82.41, 39.33) --
    ( 82.30, 39.69) --
    ( 82.17, 40.05) --
    ( 82.02, 40.39) --
    ( 81.85, 40.73) --
    ( 81.66, 41.06) --
    ( 81.44, 41.37) --
    ( 81.21, 41.67) --
    ( 80.96, 41.95) --
    ( 80.69, 42.22) --
    ( 80.41, 42.47) --
    ( 80.11, 42.70) --
    ( 79.80, 42.92) --
    ( 79.47, 43.11) --
    ( 79.14, 43.28) --
    ( 78.79, 43.43) --
    ( 78.43, 43.56) --
    ( 78.07, 43.66) --
    ( 77.70, 43.75) --
    ( 77.33, 43.81) --
    ( 76.95, 43.84) --
    ( 76.57, 43.85) --
    ( 76.19, 43.84) --
    ( 75.82, 43.81) --
    ( 75.44, 43.75) --
    ( 75.07, 43.66) --
    ( 74.71, 43.56) --
    ( 74.35, 43.43) --
    ( 74.01, 43.28) --
    ( 73.67, 43.11) --
    ( 73.34, 42.92) --
    ( 73.03, 42.70) --
    ( 72.73, 42.47) --
    ( 72.45, 42.22) --
    ( 72.18, 41.95) --
    ( 71.93, 41.67) --
    ( 71.70, 41.37) --
    ( 71.49, 41.06) --
    ( 71.29, 40.73) --
    ( 71.12, 40.39) --
    ( 70.97, 40.05) --
    ( 70.84, 39.69) --
    ( 70.74, 39.33) --
    ( 70.66, 38.96) --
    ( 70.60, 38.59) --
    ( 70.56, 38.21) --
    ( 70.55, 37.83) --
    ( 70.56, 37.45) --
    ( 70.60, 37.08) --
    ( 70.66, 36.70) --
    ( 70.74, 36.33) --
    ( 70.84, 35.97) --
    ( 70.97, 35.61) --
    ( 71.12, 35.27) --
    ( 71.29, 34.93) --
    ( 71.49, 34.60) --
    ( 71.70, 34.29) --
    ( 71.93, 33.99) --
    ( 72.18, 33.71) --
    ( 72.45, 33.44) --
    ( 72.73, 33.19) --
    ( 73.03, 32.96) --
    ( 73.34, 32.75) --
    ( 73.67, 32.55) --
    ( 74.01, 32.38) --
    ( 74.35, 32.23) --
    ( 74.71, 32.10) --
    ( 75.07, 32.00) --
    ( 75.44, 31.91) --
    ( 75.82, 31.86) --
    ( 76.19, 31.82) --
    ( 76.57, 31.81) --
    ( 76.95, 31.82) --
    ( 77.33, 31.86) --
    ( 77.70, 31.91) --
    ( 78.07, 32.00) --
    ( 78.43, 32.10) --
    ( 78.79, 32.23) --
    ( 79.14, 32.38) --
    ( 79.47, 32.55) --
    ( 79.80, 32.75) --
    ( 80.11, 32.96) --
    ( 80.41, 33.19) --
    ( 80.69, 33.44) --
    ( 80.96, 33.71) --
    ( 81.21, 33.99) --
    ( 81.44, 34.29) --
    ( 81.66, 34.60) --
    ( 81.85, 34.93) --
    ( 82.02, 35.27) --
    ( 82.17, 35.61) --
    ( 82.30, 35.97) --
    ( 82.41, 36.33) --
    ( 82.49, 36.70) --
    ( 82.55, 37.08) --
    ( 82.58, 37.45) --
    cycle;
    
    \node[text=drawColor,anchor=base,inner sep=0pt, outer sep=0pt, scale=  
    1.00] at ( 76.57, 34.62) {2};
    
    \path[draw=drawColor,line width= 0.4pt,line join=round,line 
    cap=round,fill=fillColor] (125.61, 37.83) --
    (125.60, 38.21) --
    (125.56, 38.59) --
    (125.51, 38.96) --
    (125.42, 39.33) --
    (125.32, 39.69) --
    (125.19, 40.05) --
    (125.04, 40.39) --
    (124.87, 40.73) --
    (124.67, 41.06) --
    (124.46, 41.37) --
    (124.23, 41.67) --
    (123.98, 41.95) --
    (123.71, 42.22) --
    (123.43, 42.47) --
    (123.13, 42.70) --
    (122.82, 42.92) --
    (122.49, 43.11) --
    (122.15, 43.28) --
    (121.81, 43.43) --
    (121.45, 43.56) --
    (121.09, 43.66) --
    (120.72, 43.75) --
    (120.34, 43.81) --
    (119.97, 43.84) --
    (119.59, 43.85) --
    (119.21, 43.84) --
    (118.83, 43.81) --
    (118.46, 43.75) --
    (118.09, 43.66) --
    (117.73, 43.56) --
    (117.37, 43.43) --
    (117.03, 43.28) --
    (116.69, 43.11) --
    (116.36, 42.92) --
    (116.05, 42.70) --
    (115.75, 42.47) --
    (115.47, 42.22) --
    (115.20, 41.95) --
    (114.95, 41.67) --
    (114.72, 41.37) --
    (114.50, 41.06) --
    (114.31, 40.73) --
    (114.14, 40.39) --
    (113.99, 40.05) --
    (113.86, 39.69) --
    (113.76, 39.33) --
    (113.67, 38.96) --
    (113.61, 38.59) --
    (113.58, 38.21) --
    (113.57, 37.83) --
    (113.58, 37.45) --
    (113.61, 37.08) --
    (113.67, 36.70) --
    (113.76, 36.33) --
    (113.86, 35.97) --
    (113.99, 35.61) --
    (114.14, 35.27) --
    (114.31, 34.93) --
    (114.50, 34.60) --
    (114.72, 34.29) --
    (114.95, 33.99) --
    (115.20, 33.71) --
    (115.47, 33.44) --
    (115.75, 33.19) --
    (116.05, 32.96) --
    (116.36, 32.75) --
    (116.69, 32.55) --
    (117.03, 32.38) --
    (117.37, 32.23) --
    (117.73, 32.10) --
    (118.09, 32.00) --
    (118.46, 31.91) --
    (118.83, 31.86) --
    (119.21, 31.82) --
    (119.59, 31.81) --
    (119.97, 31.82) --
    (120.34, 31.86) --
    (120.72, 31.91) --
    (121.09, 32.00) --
    (121.45, 32.10) --
    (121.81, 32.23) --
    (122.15, 32.38) --
    (122.49, 32.55) --
    (122.82, 32.75) --
    (123.13, 32.96) --
    (123.43, 33.19) --
    (123.71, 33.44) --
    (123.98, 33.71) --
    (124.23, 33.99) --
    (124.46, 34.29) --
    (124.67, 34.60) --
    (124.87, 34.93) --
    (125.04, 35.27) --
    (125.19, 35.61) --
    (125.32, 35.97) --
    (125.42, 36.33) --
    (125.51, 36.70) --
    (125.56, 37.08) --
    (125.60, 37.45) --
    cycle;
    
    \node[text=drawColor,anchor=base,inner sep=0pt, outer sep=0pt, scale=  
    1.00] at (119.59, 34.62) {3};
    
    \path[draw=drawColor,line width= 0.4pt,line join=round,line 
    cap=round,fill=fillColor] (168.63, 37.83) --
    (168.62, 38.21) --
    (168.58, 38.59) --
    (168.52, 38.96) --
    (168.44, 39.33) --
    (168.34, 39.69) --
    (168.21, 40.05) --
    (168.06, 40.39) --
    (167.89, 40.73) --
    (167.69, 41.06) --
    (167.48, 41.37) --
    (167.25, 41.67) --
    (167.00, 41.95) --
    (166.73, 42.22) --
    (166.45, 42.47) --
    (166.15, 42.70) --
    (165.83, 42.92) --
    (165.51, 43.11) --
    (165.17, 43.28) --
    (164.82, 43.43) --
    (164.47, 43.56) --
    (164.11, 43.66) --
    (163.74, 43.75) --
    (163.36, 43.81) --
    (162.99, 43.84) --
    (162.61, 43.85) --
    (162.23, 43.84) --
    (161.85, 43.81) --
    (161.48, 43.75) --
    (161.11, 43.66) --
    (160.75, 43.56) --
    (160.39, 43.43) --
    (160.04, 43.28) --
    (159.71, 43.11) --
    (159.38, 42.92) --
    (159.07, 42.70) --
    (158.77, 42.47) --
    (158.48, 42.22) --
    (158.22, 41.95) --
    (157.97, 41.67) --
    (157.74, 41.37) --
    (157.52, 41.06) --
    (157.33, 40.73) --
    (157.16, 40.39) --
    (157.01, 40.05) --
    (156.88, 39.69) --
    (156.77, 39.33) --
    (156.69, 38.96) --
    (156.63, 38.59) --
    (156.60, 38.21) --
    (156.58, 37.83) --
    (156.60, 37.45) --
    (156.63, 37.08) --
    (156.69, 36.70) --
    (156.77, 36.33) --
    (156.88, 35.97) --
    (157.01, 35.61) --
    (157.16, 35.27) --
    (157.33, 34.93) --
    (157.52, 34.60) --
    (157.74, 34.29) --
    (157.97, 33.99) --
    (158.22, 33.71) --
    (158.48, 33.44) --
    (158.77, 33.19) --
    (159.07, 32.96) --
    (159.38, 32.75) --
    (159.71, 32.55) --
    (160.04, 32.38) --
    (160.39, 32.23) --
    (160.75, 32.10) --
    (161.11, 32.00) --
    (161.48, 31.91) --
    (161.85, 31.86) --
    (162.23, 31.82) --
    (162.61, 31.81) --
    (162.99, 31.82) --
    (163.36, 31.86) --
    (163.74, 31.91) --
    (164.11, 32.00) --
    (164.47, 32.10) --
    (164.82, 32.23) --
    (165.17, 32.38) --
    (165.51, 32.55) --
    (165.83, 32.75) --
    (166.15, 32.96) --
    (166.45, 33.19) --
    (166.73, 33.44) --
    (167.00, 33.71) --
    (167.25, 33.99) --
    (167.48, 34.29) --
    (167.69, 34.60) --
    (167.89, 34.93) --
    (168.06, 35.27) --
    (168.21, 35.61) --
    (168.34, 35.97) --
    (168.44, 36.33) --
    (168.52, 36.70) --
    (168.58, 37.08) --
    (168.62, 37.45) --
    cycle;
    
    \node[text=drawColor,anchor=base,inner sep=0pt, outer sep=0pt, scale=  
    1.00] at (162.61, 34.62) {4};
    
    \path[draw=drawColor,line width= 0.4pt,line join=round,line 
    cap=round,fill=fillColor] (211.65, 37.83) --
    (211.64, 38.21) --
    (211.60, 38.59) --
    (211.54, 38.96) --
    (211.46, 39.33) --
    (211.35, 39.69) --
    (211.22, 40.05) --
    (211.07, 40.39) --
    (210.90, 40.73) --
    (210.71, 41.06) --
    (210.50, 41.37) --
    (210.27, 41.67) --
    (210.02, 41.95) --
    (209.75, 42.22) --
    (209.46, 42.47) --
    (209.17, 42.70) --
    (208.85, 42.92) --
    (208.53, 43.11) --
    (208.19, 43.28) --
    (207.84, 43.43) --
    (207.49, 43.56) --
    (207.12, 43.66) --
    (206.75, 43.75) --
    (206.38, 43.81) --
    (206.00, 43.84) --
    (205.63, 43.85) --
    (205.25, 43.84) --
    (204.87, 43.81) --
    (204.50, 43.75) --
    (204.13, 43.66) --
    (203.76, 43.56) --
    (203.41, 43.43) --
    (203.06, 43.28) --
    (202.72, 43.11) --
    (202.40, 42.92) --
    (202.09, 42.70) --
    (201.79, 42.47) --
    (201.50, 42.22) --
    (201.24, 41.95) --
    (200.98, 41.67) --
    (200.75, 41.37) --
    (200.54, 41.06) --
    (200.35, 40.73) --
    (200.18, 40.39) --
    (200.03, 40.05) --
    (199.90, 39.69) --
    (199.79, 39.33) --
    (199.71, 38.96) --
    (199.65, 38.59) --
    (199.61, 38.21) --
    (199.60, 37.83) --
    (199.61, 37.45) --
    (199.65, 37.08) --
    (199.71, 36.70) --
    (199.79, 36.33) --
    (199.90, 35.97) --
    (200.03, 35.61) --
    (200.18, 35.27) --
    (200.35, 34.93) --
    (200.54, 34.60) --
    (200.75, 34.29) --
    (200.98, 33.99) --
    (201.24, 33.71) --
    (201.50, 33.44) --
    (201.79, 33.19) --
    (202.09, 32.96) --
    (202.40, 32.75) --
    (202.72, 32.55) --
    (203.06, 32.38) --
    (203.41, 32.23) --
    (203.76, 32.10) --
    (204.13, 32.00) --
    (204.50, 31.91) --
    (204.87, 31.86) --
    (205.25, 31.82) --
    (205.63, 31.81) --
    (206.00, 31.82) --
    (206.38, 31.86) --
    (206.75, 31.91) --
    (207.12, 32.00) --
    (207.49, 32.10) --
    (207.84, 32.23) --
    (208.19, 32.38) --
    (208.53, 32.55) --
    (208.85, 32.75) --
    (209.17, 32.96) --
    (209.46, 33.19) --
    (209.75, 33.44) --
    (210.02, 33.71) --
    (210.27, 33.99) --
    (210.50, 34.29) --
    (210.71, 34.60) --
    (210.90, 34.93) --
    (211.07, 35.27) --
    (211.22, 35.61) --
    (211.35, 35.97) --
    (211.46, 36.33) --
    (211.54, 36.70) --
    (211.60, 37.08) --
    (211.64, 37.45) --
    cycle;
    
    \node[text=drawColor,anchor=base,inner sep=0pt, outer sep=0pt, scale=  
    1.00] at (205.63, 34.62) {5};
    
    \path[draw=drawColor,line width= 0.4pt,line join=round,line 
    cap=round,fill=fillColor] (254.67, 37.83) --
    (254.65, 38.21) --
    (254.62, 38.59) --
    (254.56, 38.96) --
    (254.48, 39.33) --
    (254.37, 39.69) --
    (254.24, 40.05) --
    (254.09, 40.39) --
    (253.92, 40.73) --
    (253.73, 41.06) --
    (253.52, 41.37) --
    (253.28, 41.67) --
    (253.03, 41.95) --
    (252.77, 42.22) --
    (252.48, 42.47) --
    (252.18, 42.70) --
    (251.87, 42.92) --
    (251.54, 43.11) --
    (251.21, 43.28) --
    (250.86, 43.43) --
    (250.50, 43.56) --
    (250.14, 43.66) --
    (249.77, 43.75) --
    (249.40, 43.81) --
    (249.02, 43.84) --
    (248.64, 43.85) --
    (248.27, 43.84) --
    (247.89, 43.81) --
    (247.51, 43.75) --
    (247.15, 43.66) --
    (246.78, 43.56) --
    (246.43, 43.43) --
    (246.08, 43.28) --
    (245.74, 43.11) --
    (245.42, 42.92) --
    (245.10, 42.70) --
    (244.80, 42.47) --
    (244.52, 42.22) --
    (244.25, 41.95) --
    (244.00, 41.67) --
    (243.77, 41.37) --
    (243.56, 41.06) --
    (243.37, 40.73) --
    (243.19, 40.39) --
    (243.04, 40.05) --
    (242.92, 39.69) --
    (242.81, 39.33) --
    (242.73, 38.96) --
    (242.67, 38.59) --
    (242.63, 38.21) --
    (242.62, 37.83) --
    (242.63, 37.45) --
    (242.67, 37.08) --
    (242.73, 36.70) --
    (242.81, 36.33) --
    (242.92, 35.97) --
    (243.04, 35.61) --
    (243.19, 35.27) --
    (243.37, 34.93) --
    (243.56, 34.60) --
    (243.77, 34.29) --
    (244.00, 33.99) --
    (244.25, 33.71) --
    (244.52, 33.44) --
    (244.80, 33.19) --
    (245.10, 32.96) --
    (245.42, 32.75) --
    (245.74, 32.55) --
    (246.08, 32.38) --
    (246.43, 32.23) --
    (246.78, 32.10) --
    (247.15, 32.00) --
    (247.51, 31.91) --
    (247.89, 31.86) --
    (248.27, 31.82) --
    (248.64, 31.81) --
    (249.02, 31.82) --
    (249.40, 31.86) --
    (249.77, 31.91) --
    (250.14, 32.00) --
    (250.50, 32.10) --
    (250.86, 32.23) --
    (251.21, 32.38) --
    (251.54, 32.55) --
    (251.87, 32.75) --
    (252.18, 32.96) --
    (252.48, 33.19) --
    (252.77, 33.44) --
    (253.03, 33.71) --
    (253.28, 33.99) --
    (253.52, 34.29) --
    (253.73, 34.60) --
    (253.92, 34.93) --
    (254.09, 35.27) --
    (254.24, 35.61) --
    (254.37, 35.97) --
    (254.48, 36.33) --
    (254.56, 36.70) --
    (254.62, 37.08) --
    (254.65, 37.45) --
    cycle;
    
    \node[text=drawColor,anchor=base,inner sep=0pt, outer sep=0pt, scale=  
    1.00] at (248.64, 34.62) {6};
    
    \path[draw=drawColor,line width= 0.4pt,line join=round,line 
    cap=round,fill=fillColor] (297.68, 37.83) --
    (297.67, 38.21) --
    (297.64, 38.59) --
    (297.58, 38.96) --
    (297.49, 39.33) --
    (297.39, 39.69) --
    (297.26, 40.05) --
    (297.11, 40.39) --
    (296.94, 40.73) --
    (296.75, 41.06) --
    (296.53, 41.37) --
    (296.30, 41.67) --
    (296.05, 41.95) --
    (295.78, 42.22) --
    (295.50, 42.47) --
    (295.20, 42.70) --
    (294.89, 42.92) --
    (294.56, 43.11) --
    (294.23, 43.28) --
    (293.88, 43.43) --
    (293.52, 43.56) --
    (293.16, 43.66) --
    (292.79, 43.75) --
    (292.42, 43.81) --
    (292.04, 43.84) --
    (291.66, 43.85) --
    (291.28, 43.84) --
    (290.91, 43.81) --
    (290.53, 43.75) --
    (290.16, 43.66) --
    (289.80, 43.56) --
    (289.44, 43.43) --
    (289.10, 43.28) --
    (288.76, 43.11) --
    (288.43, 42.92) --
    (288.12, 42.70) --
    (287.82, 42.47) --
    (287.54, 42.22) --
    (287.27, 41.95) --
    (287.02, 41.67) --
    (286.79, 41.37) --
    (286.58, 41.06) --
    (286.38, 40.73) --
    (286.21, 40.39) --
    (286.06, 40.05) --
    (285.93, 39.69) --
    (285.83, 39.33) --
    (285.75, 38.96) --
    (285.69, 38.59) --
    (285.65, 38.21) --
    (285.64, 37.83) --
    (285.65, 37.45) --
    (285.69, 37.08) --
    (285.75, 36.70) --
    (285.83, 36.33) --
    (285.93, 35.97) --
    (286.06, 35.61) --
    (286.21, 35.27) --
    (286.38, 34.93) --
    (286.58, 34.60) --
    (286.79, 34.29) --
    (287.02, 33.99) --
    (287.27, 33.71) --
    (287.54, 33.44) --
    (287.82, 33.19) --
    (288.12, 32.96) --
    (288.43, 32.75) --
    (288.76, 32.55) --
    (289.10, 32.38) --
    (289.44, 32.23) --
    (289.80, 32.10) --
    (290.16, 32.00) --
    (290.53, 31.91) --
    (290.91, 31.86) --
    (291.28, 31.82) --
    (291.66, 31.81) --
    (292.04, 31.82) --
    (292.42, 31.86) --
    (292.79, 31.91) --
    (293.16, 32.00) --
    (293.52, 32.10) --
    (293.88, 32.23) --
    (294.23, 32.38) --
    (294.56, 32.55) --
    (294.89, 32.75) --
    (295.20, 32.96) --
    (295.50, 33.19) --
    (295.78, 33.44) --
    (296.05, 33.71) --
    (296.30, 33.99) --
    (296.53, 34.29) --
    (296.75, 34.60) --
    (296.94, 34.93) --
    (297.11, 35.27) --
    (297.26, 35.61) --
    (297.39, 35.97) --
    (297.49, 36.33) --
    (297.58, 36.70) --
    (297.64, 37.08) --
    (297.67, 37.45) --
    cycle;
    
    \node[text=drawColor,anchor=base,inner sep=0pt, outer sep=0pt, scale=  
    1.00] at (291.66, 34.62) {7};
    
    \path[draw=drawColor,line width= 0.4pt,dash pattern=on 4pt off 4pt ,line 
    join=round,line cap=round] ( 12.04, 25.93) --
    ( 12.04, 49.73);
    
    \path[draw=drawColor,line width= 0.4pt,dash pattern=on 4pt off 4pt ,line 
    join=round,line cap=round] (313.17, 25.93) --
    (313.17, 49.73);
    
    \node[text=drawColor,anchor=base,inner sep=0pt, outer sep=0pt, scale=  
    1.00] at ( 12.04, 16.78) {0};
    
    \node[text=drawColor,anchor=base,inner sep=0pt, outer sep=0pt, scale=  
    1.00] at (313.17, 16.78) {1};
    \end{scope}
    \end{tikzpicture}\\
    \vspace{-10mm}
    {\bf steady state}
    \caption{Application of the global model to a system of 7 particles with 
    weight matrix $\bm{\tilde{W}} = \bm{1}\bm{1}\transpose$.}
    \label{fig:swarm-simple}
\end{figure*}
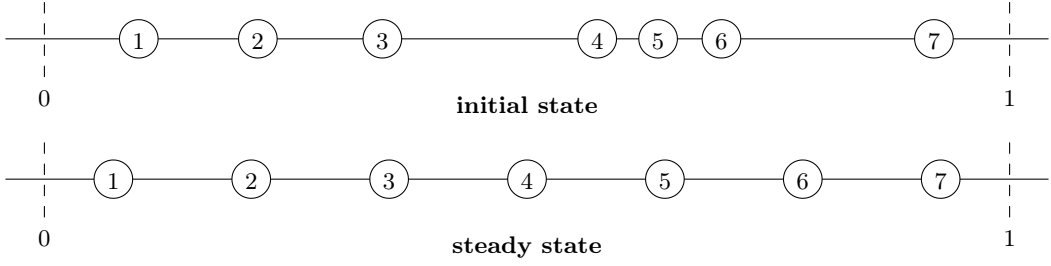
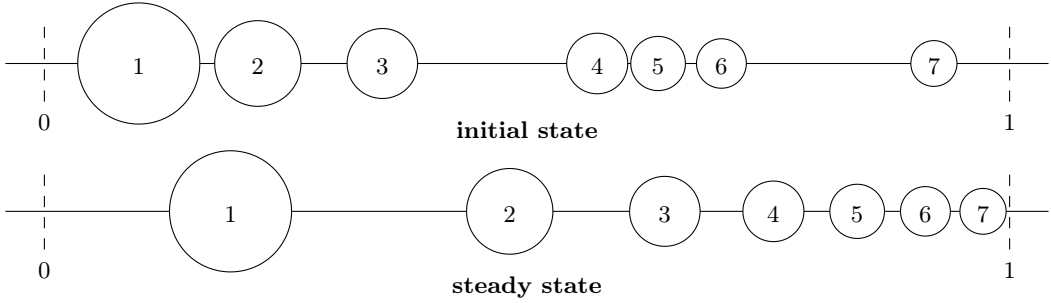
\begin{figure*}[t]
    \centering
    \begin{tikzpicture}[x=1.2pt,y=1.2pt]
    \definecolor{fillColor}{RGB}{255,255,255}
    \path[use as bounding box,fill=fillColor,fill opacity=0.00] (0,0) rectangle 
    (325.21, 57.82);
    \begin{scope}
    \path[clip] (  0.00,  0.00) rectangle (325.21, 57.82);
    \definecolor{drawColor}{RGB}{0,0,0}
    
    \path[draw=drawColor,line width= 0.4pt,line join=round,line cap=round] (  
    0.00, 37.83) -- (325.21, 37.83);
    \definecolor{fillColor}{RGB}{255,255,255}
    
    \path[draw=drawColor,line width= 0.4pt,line join=round,line 
    cap=round,fill=fillColor] ( 60.53, 37.83) --
    ( 60.49, 39.03) --
    ( 60.38, 40.22) --
    ( 60.19, 41.40) --
    ( 59.93, 42.57) --
    ( 59.60, 43.72) --
    ( 59.19, 44.84) --
    ( 58.72, 45.94) --
    ( 58.18, 47.01) --
    ( 57.57, 48.03) --
    ( 56.89, 49.02) --
    ( 56.16, 49.97) --
    ( 55.37, 50.87) --
    ( 54.52, 51.71) --
    ( 53.63, 52.50) --
    ( 52.68, 53.24) --
    ( 51.69, 53.91) --
    ( 50.66, 54.52) --
    ( 49.60, 55.06) --
    ( 48.50, 55.54) --
    ( 47.37, 55.94) --
    ( 46.22, 56.28) --
    ( 45.05, 56.54) --
    ( 43.87, 56.72) --
    ( 42.68, 56.84) --
    ( 41.49, 56.88) --
    ( 40.29, 56.84) --
    ( 39.10, 56.72) --
    ( 37.92, 56.54) --
    ( 36.75, 56.28) --
    ( 35.60, 55.94) --
    ( 34.48, 55.54) --
    ( 33.38, 55.06) --
    ( 32.31, 54.52) --
    ( 31.28, 53.91) --
    ( 30.29, 53.24) --
    ( 29.35, 52.50) --
    ( 28.45, 51.71) --
    ( 27.60, 50.87) --
    ( 26.81, 49.97) --
    ( 26.08, 49.02) --
    ( 25.41, 48.03) --
    ( 24.80, 47.01) --
    ( 24.25, 45.94) --
    ( 23.78, 44.84) --
    ( 23.37, 43.72) --
    ( 23.04, 42.57) --
    ( 22.78, 41.40) --
    ( 22.59, 40.22) --
    ( 22.48, 39.03) --
    ( 22.44, 37.83) --
    ( 22.48, 36.63) --
    ( 22.59, 35.44) --
    ( 22.78, 34.26) --
    ( 23.04, 33.09) --
    ( 23.37, 31.95) --
    ( 23.78, 30.82) --
    ( 24.25, 29.72) --
    ( 24.80, 28.66) --
    ( 25.41, 27.63) --
    ( 26.08, 26.64) --
    ( 26.81, 25.69) --
    ( 27.60, 24.79) --
    ( 28.45, 23.95) --
    ( 29.35, 23.16) --
    ( 30.29, 22.42) --
    ( 31.28, 21.75) --
    ( 32.31, 21.14) --
    ( 33.38, 20.60) --
    ( 34.48, 20.12) --
    ( 35.60, 19.72) --
    ( 36.75, 19.38) --
    ( 37.92, 19.12) --
    ( 39.10, 18.94) --
    ( 40.29, 18.82) --
    ( 41.49, 18.79) --
    ( 42.68, 18.82) --
    ( 43.87, 18.94) --
    ( 45.05, 19.12) --
    ( 46.22, 19.38) --
    ( 47.37, 19.72) --
    ( 48.50, 20.12) --
    ( 49.60, 20.60) --
    ( 50.66, 21.14) --
    ( 51.69, 21.75) --
    ( 52.68, 22.42) --
    ( 53.63, 23.16) --
    ( 54.52, 23.95) --
    ( 55.37, 24.79) --
    ( 56.16, 25.69) --
    ( 56.89, 26.64) --
    ( 57.57, 27.63) --
    ( 58.18, 28.66) --
    ( 58.72, 29.72) --
    ( 59.19, 30.82) --
    ( 59.60, 31.95) --
    ( 59.93, 33.09) --
    ( 60.19, 34.26) --
    ( 60.38, 35.44) --
    ( 60.49, 36.63) --
    cycle;
    
    \node[text=drawColor,anchor=base,inner sep=0pt, outer sep=0pt, scale=  
    1.00] at ( 41.49, 34.62) {1};
    
    \path[draw=drawColor,line width= 0.4pt,line join=round,line 
    cap=round,fill=fillColor] ( 92.10, 37.83) --
    ( 92.07, 38.68) --
    ( 91.99, 39.52) --
    ( 91.86, 40.35) --
    ( 91.68, 41.18) --
    ( 91.44, 41.99) --
    ( 91.15, 42.79) --
    ( 90.82, 43.56) --
    ( 90.43, 44.32) --
    ( 90.00, 45.05) --
    ( 89.53, 45.75) --
    ( 89.01, 46.41) --
    ( 88.45, 47.05) --
    ( 87.85, 47.65) --
    ( 87.22, 48.21) --
    ( 86.55, 48.73) --
    ( 85.85, 49.20) --
    ( 85.12, 49.63) --
    ( 84.37, 50.02) --
    ( 83.59, 50.35) --
    ( 82.80, 50.64) --
    ( 81.98, 50.87) --
    ( 81.16, 51.06) --
    ( 80.32, 51.19) --
    ( 79.48, 51.27) --
    ( 78.63, 51.30) --
    ( 77.79, 51.27) --
    ( 76.95, 51.19) --
    ( 76.11, 51.06) --
    ( 75.28, 50.87) --
    ( 74.47, 50.64) --
    ( 73.68, 50.35) --
    ( 72.90, 50.02) --
    ( 72.15, 49.63) --
    ( 71.42, 49.20) --
    ( 70.72, 48.73) --
    ( 70.05, 48.21) --
    ( 69.41, 47.65) --
    ( 68.82, 47.05) --
    ( 68.26, 46.41) --
    ( 67.74, 45.75) --
    ( 67.26, 45.05) --
    ( 66.83, 44.32) --
    ( 66.45, 43.56) --
    ( 66.11, 42.79) --
    ( 65.83, 41.99) --
    ( 65.59, 41.18) --
    ( 65.41, 40.35) --
    ( 65.27, 39.52) --
    ( 65.19, 38.68) --
    ( 65.17, 37.83) --
    ( 65.19, 36.98) --
    ( 65.27, 36.14) --
    ( 65.41, 35.31) --
    ( 65.59, 34.48) --
    ( 65.83, 33.67) --
    ( 66.11, 32.87) --
    ( 66.45, 32.10) --
    ( 66.83, 31.34) --
    ( 67.26, 30.61) --
    ( 67.74, 29.91) --
    ( 68.26, 29.25) --
    ( 68.82, 28.61) --
    ( 69.41, 28.01) --
    ( 70.05, 27.45) --
    ( 70.72, 26.94) --
    ( 71.42, 26.46) --
    ( 72.15, 26.03) --
    ( 72.90, 25.65) --
    ( 73.68, 25.31) --
    ( 74.47, 25.02) --
    ( 75.28, 24.79) --
    ( 76.11, 24.60) --
    ( 76.95, 24.47) --
    ( 77.79, 24.39) --
    ( 78.63, 24.36) --
    ( 79.48, 24.39) --
    ( 80.32, 24.47) --
    ( 81.16, 24.60) --
    ( 81.98, 24.79) --
    ( 82.80, 25.02) --
    ( 83.59, 25.31) --
    ( 84.37, 25.65) --
    ( 85.12, 26.03) --
    ( 85.85, 26.46) --
    ( 86.55, 26.94) --
    ( 87.22, 27.45) --
    ( 87.85, 28.01) --
    ( 88.45, 28.61) --
    ( 89.01, 29.25) --
    ( 89.53, 29.91) --
    ( 90.00, 30.61) --
    ( 90.43, 31.34) --
    ( 90.82, 32.10) --
    ( 91.15, 32.87) --
    ( 91.44, 33.67) --
    ( 91.68, 34.48) --
    ( 91.86, 35.31) --
    ( 91.99, 36.14) --
    ( 92.07, 36.98) --
    cycle;
    
    \node[text=drawColor,anchor=base,inner sep=0pt, outer sep=0pt, scale=  
    1.00] at ( 78.63, 34.62) {2};
    
    \path[draw=drawColor,line width= 0.4pt,line join=round,line 
    cap=round,fill=fillColor] (128.47, 37.83) --
    (128.45, 38.52) --
    (128.39, 39.21) --
    (128.28, 39.89) --
    (128.13, 40.56) --
    (127.93, 41.23) --
    (127.70, 41.88) --
    (127.43, 42.51) --
    (127.11, 43.13) --
    (126.76, 43.72) --
    (126.37, 44.29) --
    (125.95, 44.84) --
    (125.49, 45.36) --
    (125.00, 45.85) --
    (124.49, 46.30) --
    (123.94, 46.73) --
    (123.37, 47.11) --
    (122.77, 47.47) --
    (122.16, 47.78) --
    (121.52, 48.05) --
    (120.87, 48.29) --
    (120.21, 48.48) --
    (119.54, 48.63) --
    (118.86, 48.74) --
    (118.17, 48.80) --
    (117.48, 48.83) --
    (116.79, 48.80) --
    (116.10, 48.74) --
    (115.42, 48.63) --
    (114.74, 48.48) --
    (114.08, 48.29) --
    (113.43, 48.05) --
    (112.80, 47.78) --
    (112.18, 47.47) --
    (111.59, 47.11) --
    (111.01, 46.73) --
    (110.47, 46.30) --
    (109.95, 45.85) --
    (109.46, 45.36) --
    (109.00, 44.84) --
    (108.58, 44.29) --
    (108.19, 43.72) --
    (107.84, 43.13) --
    (107.53, 42.51) --
    (107.25, 41.88) --
    (107.02, 41.23) --
    (106.83, 40.56) --
    (106.68, 39.89) --
    (106.57, 39.21) --
    (106.50, 38.52) --
    (106.48, 37.83) --
    (106.50, 37.14) --
    (106.57, 36.45) --
    (106.68, 35.77) --
    (106.83, 35.10) --
    (107.02, 34.43) --
    (107.25, 33.78) --
    (107.53, 33.15) --
    (107.84, 32.53) --
    (108.19, 31.94) --
    (108.58, 31.37) --
    (109.00, 30.82) --
    (109.46, 30.30) --
    (109.95, 29.81) --
    (110.47, 29.36) --
    (111.01, 28.93) --
    (111.59, 28.55) --
    (112.18, 28.19) --
    (112.80, 27.88) --
    (113.43, 27.61) --
    (114.08, 27.37) --
    (114.74, 27.18) --
    (115.42, 27.03) --
    (116.10, 26.92) --
    (116.79, 26.86) --
    (117.48, 26.83) --
    (118.17, 26.86) --
    (118.86, 26.92) --
    (119.54, 27.03) --
    (120.21, 27.18) --
    (120.87, 27.37) --
    (121.52, 27.61) --
    (122.16, 27.88) --
    (122.77, 28.19) --
    (123.37, 28.55) --
    (123.94, 28.93) --
    (124.49, 29.36) --
    (125.00, 29.81) --
    (125.49, 30.30) --
    (125.95, 30.82) --
    (126.37, 31.37) --
    (126.76, 31.94) --
    (127.11, 32.53) --
    (127.43, 33.15) --
    (127.70, 33.78) --
    (127.93, 34.43) --
    (128.13, 35.10) --
    (128.28, 35.77) --
    (128.39, 36.45) --
    (128.45, 37.14) --
    cycle;
    
    \node[text=drawColor,anchor=base,inner sep=0pt, outer sep=0pt, scale=  
    1.00] at (117.48, 34.62) {3};
    
    \path[draw=drawColor,line width= 0.4pt,line join=round,line 
    cap=round,fill=fillColor] (194.03, 37.83) --
    (194.01, 38.43) --
    (193.96, 39.02) --
    (193.86, 39.61) --
    (193.73, 40.20) --
    (193.57, 40.77) --
    (193.36, 41.34) --
    (193.13, 41.88) --
    (192.85, 42.42) --
    (192.55, 42.93) --
    (192.21, 43.43) --
    (191.85, 43.90) --
    (191.45, 44.35) --
    (191.03, 44.77) --
    (190.58, 45.17) --
    (190.11, 45.53) --
    (189.61, 45.87) --
    (189.10, 46.17) --
    (188.56, 46.45) --
    (188.01, 46.68) --
    (187.45, 46.89) --
    (186.88, 47.05) --
    (186.29, 47.18) --
    (185.70, 47.28) --
    (185.11, 47.33) --
    (184.51, 47.35) --
    (183.91, 47.33) --
    (183.32, 47.28) --
    (182.73, 47.18) --
    (182.14, 47.05) --
    (181.57, 46.89) --
    (181.00, 46.68) --
    (180.46, 46.45) --
    (179.92, 46.17) --
    (179.41, 45.87) --
    (178.91, 45.53) --
    (178.44, 45.17) --
    (177.99, 44.77) --
    (177.57, 44.35) --
    (177.17, 43.90) --
    (176.81, 43.43) --
    (176.47, 42.93) --
    (176.16, 42.42) --
    (175.89, 41.88) --
    (175.66, 41.34) --
    (175.45, 40.77) --
    (175.29, 40.20) --
    (175.16, 39.61) --
    (175.06, 39.02) --
    (175.01, 38.43) --
    (174.99, 37.83) --
    (175.01, 37.23) --
    (175.06, 36.64) --
    (175.16, 36.05) --
    (175.29, 35.46) --
    (175.45, 34.89) --
    (175.66, 34.32) --
    (175.89, 33.78) --
    (176.16, 33.24) --
    (176.47, 32.73) --
    (176.81, 32.23) --
    (177.17, 31.76) --
    (177.57, 31.31) --
    (177.99, 30.89) --
    (178.44, 30.49) --
    (178.91, 30.13) --
    (179.41, 29.79) --
    (179.92, 29.49) --
    (180.46, 29.21) --
    (181.00, 28.98) --
    (181.57, 28.77) --
    (182.14, 28.61) --
    (182.73, 28.48) --
    (183.32, 28.38) --
    (183.91, 28.33) --
    (184.51, 28.31) --
    (185.11, 28.33) --
    (185.70, 28.38) --
    (186.29, 28.48) --
    (186.88, 28.61) --
    (187.45, 28.77) --
    (188.01, 28.98) --
    (188.56, 29.21) --
    (189.10, 29.49) --
    (189.61, 29.79) --
    (190.11, 30.13) --
    (190.58, 30.49) --
    (191.03, 30.89) --
    (191.45, 31.31) --
    (191.85, 31.76) --
    (192.21, 32.23) --
    (192.55, 32.73) --
    (192.85, 33.24) --
    (193.13, 33.78) --
    (193.36, 34.32) --
    (193.57, 34.89) --
    (193.73, 35.46) --
    (193.86, 36.05) --
    (193.96, 36.64) --
    (194.01, 37.23) --
    cycle;
    
    \node[text=drawColor,anchor=base,inner sep=0pt, outer sep=0pt, scale=  
    1.00] at (184.51, 34.62) {4};
    
    \path[draw=drawColor,line width= 0.4pt,line join=round,line 
    cap=round,fill=fillColor] (212.02, 37.83) --
    (212.00, 38.37) --
    (211.95, 38.90) --
    (211.87, 39.43) --
    (211.75, 39.95) --
    (211.60, 40.46) --
    (211.42, 40.97) --
    (211.21, 41.46) --
    (210.97, 41.93) --
    (210.69, 42.39) --
    (210.39, 42.84) --
    (210.07, 43.26) --
    (209.71, 43.66) --
    (209.33, 44.04) --
    (208.93, 44.39) --
    (208.51, 44.72) --
    (208.07, 45.02) --
    (207.61, 45.29) --
    (207.13, 45.54) --
    (206.64, 45.75) --
    (206.14, 45.93) --
    (205.62, 46.08) --
    (205.10, 46.20) --
    (204.57, 46.28) --
    (204.04, 46.33) --
    (203.50, 46.35) --
    (202.97, 46.33) --
    (202.44, 46.28) --
    (201.91, 46.20) --
    (201.39, 46.08) --
    (200.87, 45.93) --
    (200.37, 45.75) --
    (199.88, 45.54) --
    (199.40, 45.29) --
    (198.94, 45.02) --
    (198.50, 44.72) --
    (198.07, 44.39) --
    (197.67, 44.04) --
    (197.29, 43.66) --
    (196.94, 43.26) --
    (196.61, 42.84) --
    (196.31, 42.39) --
    (196.04, 41.93) --
    (195.80, 41.46) --
    (195.58, 40.97) --
    (195.40, 40.46) --
    (195.25, 39.95) --
    (195.14, 39.43) --
    (195.05, 38.90) --
    (195.00, 38.37) --
    (194.99, 37.83) --
    (195.00, 37.30) --
    (195.05, 36.76) --
    (195.14, 36.23) --
    (195.25, 35.71) --
    (195.40, 35.20) --
    (195.58, 34.69) --
    (195.80, 34.20) --
    (196.04, 33.73) --
    (196.31, 33.27) --
    (196.61, 32.82) --
    (196.94, 32.40) --
    (197.29, 32.00) --
    (197.67, 31.62) --
    (198.07, 31.27) --
    (198.50, 30.94) --
    (198.94, 30.64) --
    (199.40, 30.37) --
    (199.88, 30.12) --
    (200.37, 29.91) --
    (200.87, 29.73) --
    (201.39, 29.58) --
    (201.91, 29.46) --
    (202.44, 29.38) --
    (202.97, 29.33) --
    (203.50, 29.31) --
    (204.04, 29.33) --
    (204.57, 29.38) --
    (205.10, 29.46) --
    (205.62, 29.58) --
    (206.14, 29.73) --
    (206.64, 29.91) --
    (207.13, 30.12) --
    (207.61, 30.37) --
    (208.07, 30.64) --
    (208.51, 30.94) --
    (208.93, 31.27) --
    (209.33, 31.62) --
    (209.71, 32.00) --
    (210.07, 32.40) --
    (210.39, 32.82) --
    (210.69, 33.27) --
    (210.97, 33.73) --
    (211.21, 34.20) --
    (211.42, 34.69) --
    (211.60, 35.20) --
    (211.75, 35.71) --
    (211.87, 36.23) --
    (211.95, 36.76) --
    (212.00, 37.30) --
    cycle;
    
    \node[text=drawColor,anchor=base,inner sep=0pt, outer sep=0pt, scale=  
    1.00] at (203.50, 34.62) {5};
    
    \path[draw=drawColor,line width= 0.4pt,line join=round,line 
    cap=round,fill=fillColor] (231.04, 37.83) --
    (231.02, 38.32) --
    (230.98, 38.80) --
    (230.90, 39.29) --
    (230.79, 39.76) --
    (230.66, 40.23) --
    (230.49, 40.69) --
    (230.30, 41.14) --
    (230.08, 41.58) --
    (229.83, 42.00) --
    (229.55, 42.40) --
    (229.25, 42.79) --
    (228.93, 43.15) --
    (228.59, 43.50) --
    (228.22, 43.82) --
    (227.83, 44.12) --
    (227.43, 44.39) --
    (227.01, 44.64) --
    (226.57, 44.87) --
    (226.13, 45.06) --
    (225.67, 45.22) --
    (225.20, 45.36) --
    (224.72, 45.47) --
    (224.24, 45.54) --
    (223.75, 45.59) --
    (223.26, 45.61) --
    (222.78, 45.59) --
    (222.29, 45.54) --
    (221.81, 45.47) --
    (221.33, 45.36) --
    (220.86, 45.22) --
    (220.40, 45.06) --
    (219.95, 44.87) --
    (219.52, 44.64) --
    (219.10, 44.39) --
    (218.69, 44.12) --
    (218.31, 43.82) --
    (217.94, 43.50) --
    (217.60, 43.15) --
    (217.27, 42.79) --
    (216.97, 42.40) --
    (216.70, 42.00) --
    (216.45, 41.58) --
    (216.23, 41.14) --
    (216.04, 40.69) --
    (215.87, 40.23) --
    (215.73, 39.76) --
    (215.63, 39.29) --
    (215.55, 38.80) --
    (215.50, 38.32) --
    (215.49, 37.83) --
    (215.50, 37.34) --
    (215.55, 36.86) --
    (215.63, 36.37) --
    (215.73, 35.90) --
    (215.87, 35.43) --
    (216.04, 34.97) --
    (216.23, 34.52) --
    (216.45, 34.08) --
    (216.70, 33.66) --
    (216.97, 33.26) --
    (217.27, 32.87) --
    (217.60, 32.51) --
    (217.94, 32.16) --
    (218.31, 31.84) --
    (218.69, 31.54) --
    (219.10, 31.27) --
    (219.52, 31.02) --
    (219.95, 30.80) --
    (220.40, 30.60) --
    (220.86, 30.44) --
    (221.33, 30.30) --
    (221.81, 30.19) --
    (222.29, 30.12) --
    (222.78, 30.07) --
    (223.26, 30.06) --
    (223.75, 30.07) --
    (224.24, 30.12) --
    (224.72, 30.19) --
    (225.20, 30.30) --
    (225.67, 30.44) --
    (226.13, 30.60) --
    (226.57, 30.80) --
    (227.01, 31.02) --
    (227.43, 31.27) --
    (227.83, 31.54) --
    (228.22, 31.84) --
    (228.59, 32.16) --
    (228.93, 32.51) --
    (229.25, 32.87) --
    (229.55, 33.26) --
    (229.83, 33.66) --
    (230.08, 34.08) --
    (230.30, 34.52) --
    (230.49, 34.97) --
    (230.66, 35.43) --
    (230.79, 35.90) --
    (230.90, 36.37) --
    (230.98, 36.86) --
    (231.02, 37.34) --
    cycle;
    
    \node[text=drawColor,anchor=base,inner sep=0pt, outer sep=0pt, scale=  
    1.00] at (223.26, 34.62) {6};
    
    \path[draw=drawColor,line width= 0.4pt,line join=round,line 
    cap=round,fill=fillColor] (296.75, 37.83) --
    (296.73, 38.28) --
    (296.69, 38.73) --
    (296.62, 39.18) --
    (296.52, 39.62) --
    (296.40, 40.05) --
    (296.24, 40.48) --
    (296.06, 40.90) --
    (295.86, 41.30) --
    (295.63, 41.69) --
    (295.37, 42.06) --
    (295.10, 42.42) --
    (294.80, 42.76) --
    (294.48, 43.08) --
    (294.14, 43.38) --
    (293.78, 43.65) --
    (293.41, 43.91) --
    (293.02, 44.14) --
    (292.62, 44.34) --
    (292.20, 44.52) --
    (291.77, 44.68) --
    (291.34, 44.80) --
    (290.90, 44.90) --
    (290.45, 44.97) --
    (290.00, 45.01) --
    (289.55, 45.03) --
    (289.10, 45.01) --
    (288.65, 44.97) --
    (288.20, 44.90) --
    (287.76, 44.80) --
    (287.33, 44.68) --
    (286.90, 44.52) --
    (286.49, 44.34) --
    (286.08, 44.14) --
    (285.69, 43.91) --
    (285.32, 43.65) --
    (284.96, 43.38) --
    (284.62, 43.08) --
    (284.30, 42.76) --
    (284.00, 42.42) --
    (283.73, 42.06) --
    (283.47, 41.69) --
    (283.24, 41.30) --
    (283.04, 40.90) --
    (282.86, 40.48) --
    (282.70, 40.05) --
    (282.58, 39.62) --
    (282.48, 39.18) --
    (282.41, 38.73) --
    (282.37, 38.28) --
    (282.35, 37.83) --
    (282.37, 37.38) --
    (282.41, 36.93) --
    (282.48, 36.48) --
    (282.58, 36.04) --
    (282.70, 35.61) --
    (282.86, 35.18) --
    (283.04, 34.77) --
    (283.24, 34.36) --
    (283.47, 33.97) --
    (283.73, 33.60) --
    (284.00, 33.24) --
    (284.30, 32.90) --
    (284.62, 32.58) --
    (284.96, 32.28) --
    (285.32, 32.01) --
    (285.69, 31.75) --
    (286.08, 31.52) --
    (286.49, 31.32) --
    (286.90, 31.14) --
    (287.33, 30.98) --
    (287.76, 30.86) --
    (288.20, 30.76) --
    (288.65, 30.69) --
    (289.10, 30.65) --
    (289.55, 30.63) --
    (290.00, 30.65) --
    (290.45, 30.69) --
    (290.90, 30.76) --
    (291.34, 30.86) --
    (291.77, 30.98) --
    (292.20, 31.14) --
    (292.62, 31.32) --
    (293.02, 31.52) --
    (293.41, 31.75) --
    (293.78, 32.01) --
    (294.14, 32.28) --
    (294.48, 32.58) --
    (294.80, 32.90) --
    (295.10, 33.24) --
    (295.37, 33.60) --
    (295.63, 33.97) --
    (295.86, 34.36) --
    (296.06, 34.77) --
    (296.24, 35.18) --
    (296.40, 35.61) --
    (296.52, 36.04) --
    (296.62, 36.48) --
    (296.69, 36.93) --
    (296.73, 37.38) --
    cycle;
    
    \node[text=drawColor,anchor=base,inner sep=0pt, outer sep=0pt, scale=  
    1.00] at (289.55, 34.62) {7};
    
    \path[draw=drawColor,line width= 0.4pt,dash pattern=on 4pt off 4pt ,line 
    join=round,line cap=round] ( 12.04, 25.93) --
    ( 12.04, 49.73);
    
    \path[draw=drawColor,line width= 0.4pt,dash pattern=on 4pt off 4pt ,line 
    join=round,line cap=round] (313.17, 25.93) --
    (313.17, 49.73);
    
    \node[text=drawColor,anchor=base,inner sep=0pt, outer sep=0pt, scale=  
    1.00] at ( 12.04, 16.78) {0};
    
    \node[text=drawColor,anchor=base,inner sep=0pt, outer sep=0pt, scale=  
    1.00] at (313.17, 16.78) {1};
    \end{scope}
    \end{tikzpicture}\\
    \vspace{-10mm}
    {\bf initial state}\\[1mm]
    
    \begin{tikzpicture}[x=1.2pt,y=1.2pt]
    \definecolor{fillColor}{RGB}{255,255,255}
    \path[use as bounding box,fill=fillColor,fill opacity=0.00] (0,0) rectangle 
    (325.21, 57.82);
    \begin{scope}
    \path[clip] (  0.00,  0.00) rectangle (325.21, 57.82);
    \definecolor{drawColor}{RGB}{0,0,0}
    
    \path[draw=drawColor,line width= 0.4pt,line join=round,line cap=round] (  
    0.00, 37.83) -- (325.21, 37.83);
    \definecolor{fillColor}{RGB}{255,255,255}
    
    \path[draw=drawColor,line width= 0.4pt,line join=round,line 
    cap=round,fill=fillColor] ( 89.16, 37.83) --
    ( 89.12, 39.03) --
    ( 89.01, 40.22) --
    ( 88.82, 41.40) --
    ( 88.56, 42.57) --
    ( 88.23, 43.72) --
    ( 87.82, 44.84) --
    ( 87.35, 45.94) --
    ( 86.80, 47.01) --
    ( 86.19, 48.03) --
    ( 85.52, 49.02) --
    ( 84.79, 49.97) --
    ( 84.00, 50.87) --
    ( 83.15, 51.71) --
    ( 82.25, 52.50) --
    ( 81.31, 53.24) --
    ( 80.32, 53.91) --
    ( 79.29, 54.52) --
    ( 78.22, 55.06) --
    ( 77.12, 55.54) --
    ( 76.00, 55.94) --
    ( 74.85, 56.28) --
    ( 73.68, 56.54) --
    ( 72.50, 56.72) --
    ( 71.31, 56.84) --
    ( 70.11, 56.88) --
    ( 68.92, 56.84) --
    ( 67.73, 56.72) --
    ( 66.54, 56.54) --
    ( 65.38, 56.28) --
    ( 64.23, 55.94) --
    ( 63.10, 55.54) --
    ( 62.00, 55.06) --
    ( 60.94, 54.52) --
    ( 59.91, 53.91) --
    ( 58.92, 53.24) --
    ( 57.97, 52.50) --
    ( 57.08, 51.71) --
    ( 56.23, 50.87) --
    ( 55.44, 49.97) --
    ( 54.71, 49.02) --
    ( 54.03, 48.03) --
    ( 53.42, 47.01) --
    ( 52.88, 45.94) --
    ( 52.41, 44.84) --
    ( 52.00, 43.72) --
    ( 51.67, 42.57) --
    ( 51.41, 41.40) --
    ( 51.22, 40.22) --
    ( 51.11, 39.03) --
    ( 51.07, 37.83) --
    ( 51.11, 36.63) --
    ( 51.22, 35.44) --
    ( 51.41, 34.26) --
    ( 51.67, 33.09) --
    ( 52.00, 31.95) --
    ( 52.41, 30.82) --
    ( 52.88, 29.72) --
    ( 53.42, 28.66) --
    ( 54.03, 27.63) --
    ( 54.71, 26.64) --
    ( 55.44, 25.69) --
    ( 56.23, 24.79) --
    ( 57.08, 23.95) --
    ( 57.97, 23.16) --
    ( 58.92, 22.42) --
    ( 59.91, 21.75) --
    ( 60.94, 21.14) --
    ( 62.00, 20.60) --
    ( 63.10, 20.12) --
    ( 64.23, 19.72) --
    ( 65.38, 19.38) --
    ( 66.54, 19.12) --
    ( 67.73, 18.94) --
    ( 68.92, 18.82) --
    ( 70.11, 18.79) --
    ( 71.31, 18.82) --
    ( 72.50, 18.94) --
    ( 73.68, 19.12) --
    ( 74.85, 19.38) --
    ( 76.00, 19.72) --
    ( 77.12, 20.12) --
    ( 78.22, 20.60) --
    ( 79.29, 21.14) --
    ( 80.32, 21.75) --
    ( 81.31, 22.42) --
    ( 82.25, 23.16) --
    ( 83.15, 23.95) --
    ( 84.00, 24.79) --
    ( 84.79, 25.69) --
    ( 85.52, 26.64) --
    ( 86.19, 27.63) --
    ( 86.80, 28.66) --
    ( 87.35, 29.72) --
    ( 87.82, 30.82) --
    ( 88.23, 31.95) --
    ( 88.56, 33.09) --
    ( 88.82, 34.26) --
    ( 89.01, 35.44) --
    ( 89.12, 36.63) --
    cycle;
    
    \node[text=drawColor,anchor=base,inner sep=0pt, outer sep=0pt, scale=  
    1.00] at ( 70.11, 34.62) {1};
    
    \path[draw=drawColor,line width= 0.4pt,line join=round,line 
    cap=round,fill=fillColor] (170.68, 37.83) --
    (170.66, 38.68) --
    (170.58, 39.52) --
    (170.44, 40.35) --
    (170.26, 41.18) --
    (170.02, 41.99) --
    (169.74, 42.79) --
    (169.40, 43.56) --
    (169.02, 44.32) --
    (168.59, 45.05) --
    (168.11, 45.75) --
    (167.59, 46.41) --
    (167.03, 47.05) --
    (166.43, 47.65) --
    (165.80, 48.21) --
    (165.13, 48.73) --
    (164.43, 49.20) --
    (163.70, 49.63) --
    (162.95, 50.02) --
    (162.17, 50.35) --
    (161.38, 50.64) --
    (160.56, 50.87) --
    (159.74, 51.06) --
    (158.90, 51.19) --
    (158.06, 51.27) --
    (157.22, 51.30) --
    (156.37, 51.27) --
    (155.53, 51.19) --
    (154.69, 51.06) --
    (153.87, 50.87) --
    (153.05, 50.64) --
    (152.26, 50.35) --
    (151.48, 50.02) --
    (150.73, 49.63) --
    (150.00, 49.20) --
    (149.30, 48.73) --
    (148.63, 48.21) --
    (148.00, 47.65) --
    (147.40, 47.05) --
    (146.84, 46.41) --
    (146.32, 45.75) --
    (145.85, 45.05) --
    (145.41, 44.32) --
    (145.03, 43.56) --
    (144.69, 42.79) --
    (144.41, 41.99) --
    (144.17, 41.18) --
    (143.99, 40.35) --
    (143.85, 39.52) --
    (143.78, 38.68) --
    (143.75, 37.83) --
    (143.78, 36.98) --
    (143.85, 36.14) --
    (143.99, 35.31) --
    (144.17, 34.48) --
    (144.41, 33.67) --
    (144.69, 32.87) --
    (145.03, 32.10) --
    (145.41, 31.34) --
    (145.85, 30.61) --
    (146.32, 29.91) --
    (146.84, 29.25) --
    (147.40, 28.61) --
    (148.00, 28.01) --
    (148.63, 27.45) --
    (149.30, 26.94) --
    (150.00, 26.46) --
    (150.73, 26.03) --
    (151.48, 25.65) --
    (152.26, 25.31) --
    (153.05, 25.02) --
    (153.87, 24.79) --
    (154.69, 24.60) --
    (155.53, 24.47) --
    (156.37, 24.39) --
    (157.22, 24.36) --
    (158.06, 24.39) --
    (158.90, 24.47) --
    (159.74, 24.60) --
    (160.56, 24.79) --
    (161.38, 25.02) --
    (162.17, 25.31) --
    (162.95, 25.65) --
    (163.70, 26.03) --
    (164.43, 26.46) --
    (165.13, 26.94) --
    (165.80, 27.45) --
    (166.43, 28.01) --
    (167.03, 28.61) --
    (167.59, 29.25) --
    (168.11, 29.91) --
    (168.59, 30.61) --
    (169.02, 31.34) --
    (169.40, 32.10) --
    (169.74, 32.87) --
    (170.02, 33.67) --
    (170.26, 34.48) --
    (170.44, 35.31) --
    (170.58, 36.14) --
    (170.66, 36.98) --
    cycle;
    
    \node[text=drawColor,anchor=base,inner sep=0pt, outer sep=0pt, scale=  
    1.00] at (157.22, 34.62) {2};
    
    \path[draw=drawColor,line width= 0.4pt,line join=round,line 
    cap=round,fill=fillColor] (216.60, 37.83) --
    (216.58, 38.52) --
    (216.51, 39.21) --
    (216.41, 39.89) --
    (216.26, 40.56) --
    (216.06, 41.23) --
    (215.83, 41.88) --
    (215.55, 42.51) --
    (215.24, 43.13) --
    (214.89, 43.72) --
    (214.50, 44.29) --
    (214.08, 44.84) --
    (213.62, 45.36) --
    (213.13, 45.85) --
    (212.61, 46.30) --
    (212.07, 46.73) --
    (211.50, 47.11) --
    (210.90, 47.47) --
    (210.29, 47.78) --
    (209.65, 48.05) --
    (209.00, 48.29) --
    (208.34, 48.48) --
    (207.67, 48.63) --
    (206.98, 48.74) --
    (206.30, 48.80) --
    (205.61, 48.83) --
    (204.92, 48.80) --
    (204.23, 48.74) --
    (203.55, 48.63) --
    (202.87, 48.48) --
    (202.21, 48.29) --
    (201.56, 48.05) --
    (200.92, 47.78) --
    (200.31, 47.47) --
    (199.71, 47.11) --
    (199.14, 46.73) --
    (198.60, 46.30) --
    (198.08, 45.85) --
    (197.59, 45.36) --
    (197.13, 44.84) --
    (196.71, 44.29) --
    (196.32, 43.72) --
    (195.97, 43.13) --
    (195.66, 42.51) --
    (195.38, 41.88) --
    (195.15, 41.23) --
    (194.96, 40.56) --
    (194.80, 39.89) --
    (194.70, 39.21) --
    (194.63, 38.52) --
    (194.61, 37.83) --
    (194.63, 37.14) --
    (194.70, 36.45) --
    (194.80, 35.77) --
    (194.96, 35.10) --
    (195.15, 34.43) --
    (195.38, 33.78) --
    (195.66, 33.15) --
    (195.97, 32.53) --
    (196.32, 31.94) --
    (196.71, 31.37) --
    (197.13, 30.82) --
    (197.59, 30.30) --
    (198.08, 29.81) --
    (198.60, 29.36) --
    (199.14, 28.93) --
    (199.71, 28.55) --
    (200.31, 28.19) --
    (200.92, 27.88) --
    (201.56, 27.61) --
    (202.21, 27.37) --
    (202.87, 27.18) --
    (203.55, 27.03) --
    (204.23, 26.92) --
    (204.92, 26.86) --
    (205.61, 26.83) --
    (206.30, 26.86) --
    (206.98, 26.92) --
    (207.67, 27.03) --
    (208.34, 27.18) --
    (209.00, 27.37) --
    (209.65, 27.61) --
    (210.29, 27.88) --
    (210.90, 28.19) --
    (211.50, 28.55) --
    (212.07, 28.93) --
    (212.61, 29.36) --
    (213.13, 29.81) --
    (213.62, 30.30) --
    (214.08, 30.82) --
    (214.50, 31.37) --
    (214.89, 31.94) --
    (215.24, 32.53) --
    (215.55, 33.15) --
    (215.83, 33.78) --
    (216.06, 34.43) --
    (216.26, 35.10) --
    (216.41, 35.77) --
    (216.51, 36.45) --
    (216.58, 37.14) --
    cycle;
    
    \node[text=drawColor,anchor=base,inner sep=0pt, outer sep=0pt, scale=  
    1.00] at (205.61, 34.62) {3};
    
    \path[draw=drawColor,line width= 0.4pt,line join=round,line 
    cap=round,fill=fillColor] (249.00, 37.83) --
    (248.98, 38.43) --
    (248.93, 39.02) --
    (248.83, 39.61) --
    (248.70, 40.20) --
    (248.54, 40.77) --
    (248.33, 41.34) --
    (248.09, 41.88) --
    (247.82, 42.42) --
    (247.52, 42.93) --
    (247.18, 43.43) --
    (246.82, 43.90) --
    (246.42, 44.35) --
    (246.00, 44.77) --
    (245.55, 45.17) --
    (245.08, 45.53) --
    (244.58, 45.87) --
    (244.07, 46.17) --
    (243.53, 46.45) --
    (242.98, 46.68) --
    (242.42, 46.89) --
    (241.85, 47.05) --
    (241.26, 47.18) --
    (240.67, 47.28) --
    (240.08, 47.33) --
    (239.48, 47.35) --
    (238.88, 47.33) --
    (238.29, 47.28) --
    (237.69, 47.18) --
    (237.11, 47.05) --
    (236.54, 46.89) --
    (235.97, 46.68) --
    (235.42, 46.45) --
    (234.89, 46.17) --
    (234.38, 45.87) --
    (233.88, 45.53) --
    (233.41, 45.17) --
    (232.96, 44.77) --
    (232.54, 44.35) --
    (232.14, 43.90) --
    (231.77, 43.43) --
    (231.44, 42.93) --
    (231.13, 42.42) --
    (230.86, 41.88) --
    (230.63, 41.34) --
    (230.42, 40.77) --
    (230.26, 40.20) --
    (230.12, 39.61) --
    (230.03, 39.02) --
    (229.98, 38.43) --
    (229.96, 37.83) --
    (229.98, 37.23) --
    (230.03, 36.64) --
    (230.12, 36.05) --
    (230.26, 35.46) --
    (230.42, 34.89) --
    (230.63, 34.32) --
    (230.86, 33.78) --
    (231.13, 33.24) --
    (231.44, 32.73) --
    (231.77, 32.23) --
    (232.14, 31.76) --
    (232.54, 31.31) --
    (232.96, 30.89) --
    (233.41, 30.49) --
    (233.88, 30.13) --
    (234.38, 29.79) --
    (234.89, 29.49) --
    (235.42, 29.21) --
    (235.97, 28.98) --
    (236.54, 28.77) --
    (237.11, 28.61) --
    (237.69, 28.48) --
    (238.29, 28.38) --
    (238.88, 28.33) --
    (239.48, 28.31) --
    (240.08, 28.33) --
    (240.67, 28.38) --
    (241.26, 28.48) --
    (241.85, 28.61) --
    (242.42, 28.77) --
    (242.98, 28.98) --
    (243.53, 29.21) --
    (244.07, 29.49) --
    (244.58, 29.79) --
    (245.08, 30.13) --
    (245.55, 30.49) --
    (246.00, 30.89) --
    (246.42, 31.31) --
    (246.82, 31.76) --
    (247.18, 32.23) --
    (247.52, 32.73) --
    (247.82, 33.24) --
    (248.09, 33.78) --
    (248.33, 34.32) --
    (248.54, 34.89) --
    (248.70, 35.46) --
    (248.83, 36.05) --
    (248.93, 36.64) --
    (248.98, 37.23) --
    cycle;
    
    \node[text=drawColor,anchor=base,inner sep=0pt, outer sep=0pt, scale=  
    1.00] at (239.48, 34.62) {4};
    
    \path[draw=drawColor,line width= 0.4pt,line join=round,line 
    cap=round,fill=fillColor] (274.13, 37.83) --
    (274.11, 38.37) --
    (274.06, 38.90) --
    (273.98, 39.43) --
    (273.86, 39.95) --
    (273.71, 40.46) --
    (273.53, 40.97) --
    (273.32, 41.46) --
    (273.07, 41.93) --
    (272.80, 42.39) --
    (272.50, 42.84) --
    (272.17, 43.26) --
    (271.82, 43.66) --
    (271.44, 44.04) --
    (271.04, 44.39) --
    (270.62, 44.72) --
    (270.17, 45.02) --
    (269.71, 45.29) --
    (269.24, 45.54) --
    (268.74, 45.75) --
    (268.24, 45.93) --
    (267.73, 46.08) --
    (267.21, 46.20) --
    (266.68, 46.28) --
    (266.14, 46.33) --
    (265.61, 46.35) --
    (265.07, 46.33) --
    (264.54, 46.28) --
    (264.01, 46.20) --
    (263.49, 46.08) --
    (262.98, 45.93) --
    (262.47, 45.75) --
    (261.98, 45.54) --
    (261.51, 45.29) --
    (261.05, 45.02) --
    (260.60, 44.72) --
    (260.18, 44.39) --
    (259.78, 44.04) --
    (259.40, 43.66) --
    (259.05, 43.26) --
    (258.72, 42.84) --
    (258.42, 42.39) --
    (258.15, 41.93) --
    (257.90, 41.46) --
    (257.69, 40.97) --
    (257.51, 40.46) --
    (257.36, 39.95) --
    (257.24, 39.43) --
    (257.16, 38.90) --
    (257.11, 38.37) --
    (257.09, 37.83) --
    (257.11, 37.30) --
    (257.16, 36.76) --
    (257.24, 36.23) --
    (257.36, 35.71) --
    (257.51, 35.20) --
    (257.69, 34.69) --
    (257.90, 34.20) --
    (258.15, 33.73) --
    (258.42, 33.27) --
    (258.72, 32.82) --
    (259.05, 32.40) --
    (259.40, 32.00) --
    (259.78, 31.62) --
    (260.18, 31.27) --
    (260.60, 30.94) --
    (261.05, 30.64) --
    (261.51, 30.37) --
    (261.98, 30.12) --
    (262.47, 29.91) --
    (262.98, 29.73) --
    (263.49, 29.58) --
    (264.01, 29.46) --
    (264.54, 29.38) --
    (265.07, 29.33) --
    (265.61, 29.31) --
    (266.14, 29.33) --
    (266.68, 29.38) --
    (267.21, 29.46) --
    (267.73, 29.58) --
    (268.24, 29.73) --
    (268.74, 29.91) --
    (269.24, 30.12) --
    (269.71, 30.37) --
    (270.17, 30.64) --
    (270.62, 30.94) --
    (271.04, 31.27) --
    (271.44, 31.62) --
    (271.82, 32.00) --
    (272.17, 32.40) --
    (272.50, 32.82) --
    (272.80, 33.27) --
    (273.07, 33.73) --
    (273.32, 34.20) --
    (273.53, 34.69) --
    (273.71, 35.20) --
    (273.86, 35.71) --
    (273.98, 36.23) --
    (274.06, 36.76) --
    (274.11, 37.30) --
    cycle;
    
    \node[text=drawColor,anchor=base,inner sep=0pt, outer sep=0pt, scale=  
    1.00] at (265.61, 34.62) {5};
    
    \path[draw=drawColor,line width= 0.4pt,line join=round,line 
    cap=round,fill=fillColor] (294.68, 37.83) --
    (294.66, 38.32) --
    (294.61, 38.80) --
    (294.54, 39.29) --
    (294.43, 39.76) --
    (294.30, 40.23) --
    (294.13, 40.69) --
    (293.94, 41.14) --
    (293.71, 41.58) --
    (293.47, 42.00) --
    (293.19, 42.40) --
    (292.89, 42.79) --
    (292.57, 43.15) --
    (292.22, 43.50) --
    (291.86, 43.82) --
    (291.47, 44.12) --
    (291.07, 44.39) --
    (290.65, 44.64) --
    (290.21, 44.87) --
    (289.76, 45.06) --
    (289.30, 45.22) --
    (288.83, 45.36) --
    (288.36, 45.47) --
    (287.88, 45.54) --
    (287.39, 45.59) --
    (286.90, 45.61) --
    (286.41, 45.59) --
    (285.93, 45.54) --
    (285.44, 45.47) --
    (284.97, 45.36) --
    (284.50, 45.22) --
    (284.04, 45.06) --
    (283.59, 44.87) --
    (283.16, 44.64) --
    (282.73, 44.39) --
    (282.33, 44.12) --
    (281.95, 43.82) --
    (281.58, 43.50) --
    (281.23, 43.15) --
    (280.91, 42.79) --
    (280.61, 42.40) --
    (280.34, 42.00) --
    (280.09, 41.58) --
    (279.87, 41.14) --
    (279.67, 40.69) --
    (279.51, 40.23) --
    (279.37, 39.76) --
    (279.26, 39.29) --
    (279.19, 38.80) --
    (279.14, 38.32) --
    (279.13, 37.83) --
    (279.14, 37.34) --
    (279.19, 36.86) --
    (279.26, 36.37) --
    (279.37, 35.90) --
    (279.51, 35.43) --
    (279.67, 34.97) --
    (279.87, 34.52) --
    (280.09, 34.08) --
    (280.34, 33.66) --
    (280.61, 33.26) --
    (280.91, 32.87) --
    (281.23, 32.51) --
    (281.58, 32.16) --
    (281.95, 31.84) --
    (282.33, 31.54) --
    (282.73, 31.27) --
    (283.16, 31.02) --
    (283.59, 30.80) --
    (284.04, 30.60) --
    (284.50, 30.44) --
    (284.97, 30.30) --
    (285.44, 30.19) --
    (285.93, 30.12) --
    (286.41, 30.07) --
    (286.90, 30.06) --
    (287.39, 30.07) --
    (287.88, 30.12) --
    (288.36, 30.19) --
    (288.83, 30.30) --
    (289.30, 30.44) --
    (289.76, 30.60) --
    (290.21, 30.80) --
    (290.65, 31.02) --
    (291.07, 31.27) --
    (291.47, 31.54) --
    (291.86, 31.84) --
    (292.22, 32.16) --
    (292.57, 32.51) --
    (292.89, 32.87) --
    (293.19, 33.26) --
    (293.47, 33.66) --
    (293.71, 34.08) --
    (293.94, 34.52) --
    (294.13, 34.97) --
    (294.30, 35.43) --
    (294.43, 35.90) --
    (294.54, 36.37) --
    (294.61, 36.86) --
    (294.66, 37.34) --
    cycle;
    
    \node[text=drawColor,anchor=base,inner sep=0pt, outer sep=0pt, scale=  
    1.00] at (286.90, 34.62) {6};
    
    \path[draw=drawColor,line width= 0.4pt,line join=round,line 
    cap=round,fill=fillColor] (312.07, 37.83) --
    (312.06, 38.28) --
    (312.02, 38.73) --
    (311.95, 39.18) --
    (311.85, 39.62) --
    (311.72, 40.05) --
    (311.57, 40.48) --
    (311.39, 40.90) --
    (311.18, 41.30) --
    (310.95, 41.69) --
    (310.70, 42.06) --
    (310.42, 42.42) --
    (310.12, 42.76) --
    (309.80, 43.08) --
    (309.46, 43.38) --
    (309.11, 43.65) --
    (308.73, 43.91) --
    (308.34, 44.14) --
    (307.94, 44.34) --
    (307.52, 44.52) --
    (307.10, 44.68) --
    (306.66, 44.80) --
    (306.22, 44.90) --
    (305.78, 44.97) --
    (305.33, 45.01) --
    (304.87, 45.03) --
    (304.42, 45.01) --
    (303.97, 44.97) --
    (303.53, 44.90) --
    (303.08, 44.80) --
    (302.65, 44.68) --
    (302.22, 44.52) --
    (301.81, 44.34) --
    (301.41, 44.14) --
    (301.02, 43.91) --
    (300.64, 43.65) --
    (300.29, 43.38) --
    (299.95, 43.08) --
    (299.63, 42.76) --
    (299.33, 42.42) --
    (299.05, 42.06) --
    (298.80, 41.69) --
    (298.57, 41.30) --
    (298.36, 40.90) --
    (298.18, 40.48) --
    (298.03, 40.05) --
    (297.90, 39.62) --
    (297.80, 39.18) --
    (297.73, 38.73) --
    (297.69, 38.28) --
    (297.68, 37.83) --
    (297.69, 37.38) --
    (297.73, 36.93) --
    (297.80, 36.48) --
    (297.90, 36.04) --
    (298.03, 35.61) --
    (298.18, 35.18) --
    (298.36, 34.77) --
    (298.57, 34.36) --
    (298.80, 33.97) --
    (299.05, 33.60) --
    (299.33, 33.24) --
    (299.63, 32.90) --
    (299.95, 32.58) --
    (300.29, 32.28) --
    (300.64, 32.01) --
    (301.02, 31.75) --
    (301.41, 31.52) --
    (301.81, 31.32) --
    (302.22, 31.14) --
    (302.65, 30.98) --
    (303.08, 30.86) --
    (303.53, 30.76) --
    (303.97, 30.69) --
    (304.42, 30.65) --
    (304.87, 30.63) --
    (305.33, 30.65) --
    (305.78, 30.69) --
    (306.22, 30.76) --
    (306.66, 30.86) --
    (307.10, 30.98) --
    (307.52, 31.14) --
    (307.94, 31.32) --
    (308.34, 31.52) --
    (308.73, 31.75) --
    (309.11, 32.01) --
    (309.46, 32.28) --
    (309.80, 32.58) --
    (310.12, 32.90) --
    (310.42, 33.24) --
    (310.70, 33.60) --
    (310.95, 33.97) --
    (311.18, 34.36) --
    (311.39, 34.77) --
    (311.57, 35.18) --
    (311.72, 35.61) --
    (311.85, 36.04) --
    (311.95, 36.48) --
    (312.02, 36.93) --
    (312.06, 37.38) --
    cycle;
    
    \node[text=drawColor,anchor=base,inner sep=0pt, outer sep=0pt, scale=  
    1.00] at (304.87, 34.62) {7};
    
    \path[draw=drawColor,line width= 0.4pt,dash pattern=on 4pt off 4pt ,line 
    join=round,line cap=round] ( 12.04, 25.93) --
    ( 12.04, 49.73);
    
    \path[draw=drawColor,line width= 0.4pt,dash pattern=on 4pt off 4pt ,line 
    join=round,line cap=round] (313.17, 25.93) --
    (313.17, 49.73);
    
    \node[text=drawColor,anchor=base,inner sep=0pt, outer sep=0pt, scale=  
    1.00] at ( 12.04, 16.78) {0};
    
    \node[text=drawColor,anchor=base,inner sep=0pt, outer sep=0pt, scale=  
    1.00] at (313.17, 16.78) {1};
    \end{scope}
    \end{tikzpicture}\\
    \vspace{-9mm}
    {\bf steady state}
    
    \caption{Application of the global model to a system of 7 particles with 
    $\tilde{w}_{i,k} = \nicefrac{1}{k}$ for $1 \leq i,k \leq N$.}
    \label{fig:swarm-weights}
\end{figure*}
%
% ------------------------------------------------------------------------------
%
Additionally, we present an analytic expression for the steady state of the 
global model given a penaliser function $\varPsi = \varPsi_{a,n}$ (cf.  
\Tref{tab:class_psi}) with $n=1$.
\begin{theorem}[Analytic Steady-State Solution for $\varPsi = \varPsi_{a,n=1}$]
Given $N$ distinct positions $v_i$ in increasing order and a penaliser function 
$\varPsi = \varPsi_{a,n=1}$, the unique minimiser of \eref{eq:model_energy} is 
given by
\begin{equation}
\label{eq:steady_state_linear}
v_i^* = 
\frac{\sum \limits_{j = 1}^i \tilde{w}_{i,j} - \frac12 
\tilde{w}_{i,i}}{\sum \limits_{j = 1}^N \tilde{w}_{i,j}}
\enspace, \qquad i = 1,\ldots,N .
\end{equation}
\end{theorem}
\begin{proof}
The presented minimiser follows directly from \eref{eq:steady_state_condition}. 
\Fref{fig:swarm-weights} provides an illustration of the steady state. \qed
\end{proof}
%
% ------------------------------------------------------------------------------
%
Finally, and in case all entries of the weight matrix $\bm{\tilde{W}}$ are set 
to $1$, we show that the global model converges -- independently of 
$\varPsi$ -- to a unique steady state:
\begin{theorem}[Convergence for $\bm{\tilde{W}} = 
\mathbf{1}\mathbf{1}\transpose$]
\label{thm:convergence_matrix_ones}
Given that $\bm{\tilde{W}} = \mathbf{1}\mathbf{1}\transpose$, any initial 
configuration $\bm{v} \in (0,1)^N$ with distinct entries converges to a unique 
steady state $\bm{v}^*$ for $t \to \infty$. This is the global minimiser of the 
energy given 
in \eref{eq:model_energy_red}.
\end{theorem}
\begin{proof}
Using the same reasoning as in the proof for Theorem 
\ref{thm:convergence_n_one} we know that inequality 
\eref{eq:lyapunov_n_one} holds.
Due to the positive definiteness of \eref{eq:hesse_energy} it follows
that $E(\bm{v},\bm{\tilde{W}})$ has a strict (global) minimum which implies that
the inequality in 
\eref{eq:lyapunov_n_one}
becomes strict except in case 
of $\bm{v} = \bm{v}^*$. This guarantees asymptotic Lyapunov stability of 
$\bm{v}^*$ and thus convergence to $\bm{v}^*$ for $t \to \infty$.
\end{proof}
%
% ------------------------------------------------------------------------------
%
\subsection{Relation to Variational Signal and Image Filtering}
\label{sec-theo-sig}
Let us now interpret $v_1,\ldots,v_N$ as samples of a smooth 1D signal
$u:\Omega\to[0,1]$ over an interval $\Omega$ of the real axis,
taken at sampling positions $x_i=x_0+i\,h$ with grid mesh size $h>0$.
We consider the model \eref{eq:model_energy} with  
$w_{i,j}:=\gamma(|x_j-x_i|)$, where $\gamma: 
\mathbb{R}_0^+ \to [0,1]$ is a non-increasing weighting function with compact 
support $[0,\varrho)$.
\begin{theorem}[Space-Continuous Energy]\label{thm-sce}
\Eref{eq:model_energy}
can be considered as a discretisation of
\begin{equation}
\label{energywithbarrier}
E[u] = \frac12 \int\limits_\Omega \bigl(W (u_x^2) + B (u)\bigr)\,\mathrm{d}x
\end{equation}
with penaliser $W(u_x^2)\approx C\,\varPsi(u_x^2)$ and barrier function
$B(u)\approx D\,\varPsi(4u^2)$, where $C$ and $D$ are positive constants.
\end{theorem}
\begin{remark} \hfill
\begin{enumerate}[{(}a{)}]
\item The penaliser $W$ is decreasing and convex in $u_x$. The barrier function 
$B$ is convex and it enforces the interval 
constraint on $u$ by favouring values $u$ away from the interval boundaries.
The discrete penaliser $\varPsi$ generates both the penaliser $W$ for 
derivatives and the barrier function $B$.
\item Note that by construction of $W$ the diffusivity\\
$g(u_x^2) := W'(u_x^2)\sim\varPsi'(u_x^2)$
has a singularity at $0$ with $-\infty$ as limit.
\item The cut-off of $\gamma$ at radius $\varrho$ implies the locality of the 
functional \eref{energywithbarrier} that can thereby be linked to a diffusion 
equation of type \eref{eq:fd}. Without a cut-off, a nonlocal diffusion equation 
would arise instead.
\end{enumerate}
\end{remark}

\begin{proof}[of Theorem~\ref{thm-sce}]
We notice first that $v_j-v_i$ and $v_i+v_j$ for $1\le i,j\le N$
are first-order approximations of $(j-i)\,h\,u_x(x_i)$ and
$2 u(x_i)$, respectively.

\paragraph{Derivation of the Penaliser $W$.}
Assume first for simplicity that $\varPsi(s^2)=-\kappa s$, $\kappa>0$ is linear 
in $s$ on $[0,1]$ (thus not strictly convex).
Then we have for a part of the inner
sums of \eref{eq:model_energy}
corresponding to a fixed $i$:
\begin{equation}
\label{Psi-from-Phi}
\begin{split}
\frac12 \biggl(
\sum \limits_{j=1}^N
&
\gamma(|x_j-x_i|) \cdot \varPsi\bigl((v_j-v_i)^2\bigr) +
\sum \limits_{\mathclap{j=N+1}}^{2N}
\gamma(|x_j-x_{2N+1-i}|) \cdot \varPsi\bigl((v_j-v_{2N+1-i})^2\bigr)
\biggr)\\
& = \sum \limits_{j=1}^N
\gamma(|x_j - x_i|) \cdot \varPsi\bigl((|v_j-v_i|)^2\bigr)\\
& \approx
-\kappa \, h \, u_x(x_i) \sum\limits_{j=1}^N
\gamma(|j-i| \, h) \cdot |j-i|\\
& = h \, \varPsi\bigl(u_x(x_i)^2\bigr)
\sum\limits_{k=1-i}^{N-i} |k| \, \gamma(|k|\,h)\\
& \approx h \varPsi (u_x (x_i)^2) \cdot
\frac{2}{h^2} \int_0^\varrho z \gamma (z) \mathrm{d}z\\
& =: h C \varPsi (u_x (x_i)^2) , \hspace{0.2\textwidth}
\end{split}
\end{equation}
where in the last step the sum over $k=1-i,\ldots,N-i$ has been replaced
with a sum over 
$k = -\lfloor \varrho/h \rfloor, \dots, \lfloor \varrho/h \rfloor$, 
thus
introducing a cutoff error for those locations $x_i$ that are within the 
distance $\varrho$ from the interval ends.
%\begin{sloppypar}
Summation over $i=1,\ldots,N$ approximates \linebreak
$\int_\Omega C \varPsi(u_x^2) \mathrm{d}x$
from which we can read off 
$W(u_x^2)\approx C \varPsi(u_x^2)$.
%\end{sloppypar}

For $\varPsi(s^2)$ that are non-linear in $s$,
$\varPsi(u_x(x_i)^2)$ in \eref{Psi-from-Phi} is changed into
a weighted sum of $\varPsi\bigl((k u_x(x_i))^2\bigr)$ for
$k=1,\ldots,N-1$,
which still amounts to a decreasing function $W(u_x^2)$ that is convex
in $u_x$. Qualitatively, $W'$ then behaves the same way as before.

\paragraph{Derivation of the Barrier Function $B$.}
Collecting the summands of \eref{eq:model_energy}
that were not used in \eref{Psi-from-Phi}, we have, again for fixed $i$,
\begin{equation}
\begin{split}
\frac12 \biggl(
\sum \limits_{{j = N+1}}^{2N}
&
\gamma\bigl(|x_j-x_i|\bigr) \cdot
\varPsi\bigl((v_j-v_i)^2\bigr)
+ \sum\limits_{j = 1}^{N} 
\gamma\bigl(|x_j-x_{2N+1-i}|\bigr) \cdot
\varPsi\bigl((v_j-v_{2N+1-i})^2\bigr)
\biggr)\\
& = \sum\limits_{j = 1}^{N} 
\gamma\bigl(|x_j-x_i|\bigr) \cdot
\varPsi \bigl((v_i+v_j)^2\bigr)\\
& \approx \left( \frac{2}{h} \int_0^\varrho \gamma(z) \mathrm{d} z + 1 \right)
\cdot \varPsi (4 u(x_i)^2)\\
& =: h D \cdot \varPsi (4 u(x_i)^2) , \hspace{0.2\textwidth}
\end{split}
\end{equation}
and thus after summation over $i$ analogous to the previous step
$\int_\Omega B(u)\,\mathrm{d}x$ with
$B(u) \approx D \varPsi (4 u^2)$.\qed
\end{proof}

Similar derivations can be made for patches of 2D images.
A point worth noticing is that the barrier function
$B$ is bounded.
This differs from usual continuous models where such barrier functions tend to
infinity at the interval boundaries. However,
for each given sampling grid and patch size the barrier function is just
strong enough to prevent $W$ from pushing the values out of the interval.
%
% ------------------------------------------------------------------------------
% ------------------------------------------------------------------------------
%
\section{Explicit Time Discretisation}
\label{sec:numerics}
Up to this point we have established a theory for the time-continuous evolution 
of particle positions. In order to be able to employ our model in 
simulations and applications we need to discretise 
\eref{eq:model_gradient_descent_redundant} in time.
Subsequently, we provide a simple yet powerful discretisation which
preserves all important properties of the time-continuous model.
An approximation of the time derivative in 
\eref{eq:model_gradient_descent_redundant} by forward differences yields the 
explicit scheme
\begin{equation}
\label{eq:explicit_redundant}
%\begin{split}
v_i^{k+1} = 
\,\, v_i^k
+ \tau \cdot 
\sum\limits_{\ell \in J_2^i}
\tilde{w}_{i,\ell} \cdot \varPhi(v_\ell^k-v_i^k)
- \tau \cdot \sum\limits_{\ell = 1}^N
\tilde{w}_{i,\ell} \cdot \varPhi(v_\ell^k+v_i^k) ,
%\end{split}
\end{equation}
for $i = 1,\ldots,N$,
where $\tau$ denotes the time step size and an upper index $k$ refers to the 
time $k\tau$.
In the following, we derive necessary conditions for which the explicit scheme 
preserves the position range $(0,1)$ and the position ordering. Furthermore, we 
show convergence of \eref{eq:explicit_redundant} in dependence of $\tau$.
\begin{theorem}[Avoidance of Range Interval Boundaries of the Explicit Scheme]
\label{thm:range_preserv}
Let $L_\varPhi$ be the Lipschitz constant of $\varPhi$ restricted to the 
interval $(0,2)$. Moreover, let $0 < v_i^k < 1$, for every $1 \leq i \leq 
N$, and assume that the time step size $\tau$ of the explicit scheme 
\eref{eq:explicit_redundant} satisfies
\begin{equation}
\label{eq:thm:explicit_range_000}
0 < \tau < \frac{1}{2 \cdot L_\varPhi \cdot \max \limits_{1 \leq i \leq N} \sum 
\limits_{\ell = 1}^N \tilde{w}_{i,\ell}} \enspace.
\end{equation}
Then it follows that $0 < v_i^{k+1} < 1$ for every $1 \leq i \leq N$.
\end{theorem}
\begin{proof}
In accordance with \eref{eq:model_gradient_descent_redundant_alternative} the 
explicit scheme \eref{eq:explicit_redundant} can be 
written as
\begin{equation}
\label{eq:explicit_redundant_alternative}
v_i^{k+1} = 
\,\, v_i^k
+ \tau \cdot 
\sum\limits_{\ell \in J_2^i}
\tilde{w}_{i,\ell} \cdot \Big( \varPhi(v_\ell^k-v_i^k) -  
\varPhi(v_\ell^k+v_i^k) \Big)
- \tau \cdot \sum\limits_{\ell \in J_3^i}
\tilde{w}_{i,\ell} \cdot \varPhi(2v_i^k) ,
\end{equation}
where $i = 1,\ldots,N$.
Now assume that $0 < v_i^k, v_j^k < 1$ and let us examine the contribution of 
the two summation terms in \eref{eq:explicit_redundant_alternative}. We need to 
distinguish the following five cases:\\
1. If $v_i^k = v_j^k \leq \frac12$ then $2v_i^k \in (0,1]$. Thus,
\begin{equation}
\label{eq:thm:explicit_range_00}
0 \leq -\varPhi(2v_i^k) .
\end{equation}
2. If $\frac12 < v_i^k = v_j^k$ then $2v_i^k \in (1,2)$. Thus, using 
$\varPhi(1) = 0$,
\begin{equation}
\label{eq:thm:explicit_range_01}
%\begin{split}
|\varPhi(2v_i^k)| = 
|\varPhi(2v_i^k) - \varPhi(1)|
\leq |2v_i^k - 1| \cdot L_\varPhi
<  2v_i^k \cdot L_\varPhi .
%\end{split}
\end{equation}
3. If $v_i^k < v_j^k$ then $v_j^k - v_i^k, \, v_j^k + v_i^k \in (0,2)$. Thus,
\begin{equation}
\label{eq:thm:explicit_range_02}
|\varPhi(v_j^k+v_i^k)-\varPhi(v_j^k-v_i^k)| \leq L_\varPhi \cdot 2v_i^k .
\end{equation}
4. If $v_j^k < v_i^k \leq \frac12$ then $v_j^k - v_i^k \in (-1,0)$ and $v_j^k + 
v_i^k \in (0,1)$. Thus,
\begin{align}
\label{eq:thm:explicit_range_03}
0 & \leq  \varPhi(v_j^k-v_i^k) - \varPhi(v_j^k+v_i^k) ,\\
\label{eq:thm:explicit_range_04}
0 & \leq  -\varPhi(2v_i^k) .
\end{align}
5. Finally, if $v_j^k < v_i^k$ and $\frac12 < v_i^k$, using the periodicity of 
$\varPhi$ we get
\begin{equation}
\label{eq:thm:explicit_range_05}
%\begin{split}
|\varPhi(v_j^k-v_i^k)-\varPhi(v_j^k+v_i^k)| %\hspace{0.1\textwidth}
%
%\begin{split}
= |\varPhi(v_j^k+v_i^k)-\varPhi(2+v_j^k-v_i^k)|
\leq 2v_i^k \cdot L_\varPhi .
%\end{split}
%\end{split}
\end{equation}
Combining \eref{eq:explicit_redundant} with 
\eref{eq:thm:explicit_range_000} and
\eref{eq:thm:explicit_range_00}--\eref{eq:thm:explicit_range_05}
we obtain that
\begin{align}
\nonumber
v_i^{k+1} - v_i^k
= &
-\tau \cdot 
\sum\limits_{\ell \in J_2^i}
\tilde{w}_{i,\ell} \cdot \Big( \varPhi(v_\ell^k+v_i^k) - 
\varPhi(v_\ell^k-v_i^k) \Big)%\\
%
%\nonumber
%&
- \tau \cdot \sum\limits_{\ell \in J_3^i}
\tilde{w}_{i,\ell} \cdot \varPhi(2v_i^k)\\
\nonumber
\geq &
-\tau \cdot L_\varPhi \cdot 2v_i^k \cdot 
\sum \limits_{\ell = 1}^N \tilde{w}_{i,\ell}\\
\label{eq:thm:explicit_range_06}
> & -v_i^k ,
\end{align}
from which it directly follows that $v_i^{k+1} > 0$, as claimed.\\
The proof for $v_i^{k+1} < 1$ is straightforward. Assume w.l.o.g. that 
$\tilde{v}_i^k := 1-v_i^k$. For the reasons given above, we obtain 
$\tilde{v}_i^{k+1} > 0$. Consequently, $1 - v_i^{k+1} > 0$ and $v_i^{k+1} < 1$ 
follows.\qed
\end{proof}
\begin{theorem}[Rank-Order Preservation of the Explicit Scheme]
Let $L_\varPhi$ be the Lipschitz constant of $\varPhi$ restricted to the 
interval 
$(0,2)$. Furthermore, let $v_i^0$, for $i = 1,\ldots,N$, denote the initially 
distinct positions in $(0,1)$ and -- in accordance with Theorem 
\ref{thm:nonequality} -- let the weight matrix $\bm{\tilde{W}}$ have constant 
columns, i.e. $\tilde{w}_{j,\ell} = \tilde{w}_{i,\ell}$ for $1 \leq i,j,\ell 
\leq N$. 
Moreover, let $0 < v_i^k < v_j^k < 1$ and assume that the time step size 
$\tau$ used in the explicit scheme 
\eref{eq:explicit_redundant} satisfies
\begin{equation}
\label{eq:thm:explicit_order_000}
0 < \tau < \frac{1}{2 \cdot L_\varPhi \cdot \max \limits_{1 \leq i \leq N} \sum 
\limits_{\ell = 1}^N \tilde{w}_{i,\ell}} \, .
\end{equation}
Then we have $v_i^{k+1} < v_j^{k+1}$.
\end{theorem}
\begin{proof}
For distinct positions, \eref{eq:explicit_redundant} reads
\begin{equation}
%\begin{split}
v_i^{k+1} = 
\,\, v_i^k +
\tau \cdot
\sum \limits_{\scriptstyle \ell = 1\atop \scriptstyle \ell \neq i}^N
\tilde{w}_{i,\ell} \cdot \varPhi(v_\ell^k-v_i^k)
- \tau \cdot
\sum \limits_{\ell = 1}^N
\tilde{w}_{i,\ell} \cdot \varPhi(v_\ell^k+v_i^k)
%\end{split}
\end{equation}
for $i = 1,\ldots,N$. Considering this explicit discretisation for $\partial_t 
v_i$ and $\partial_t v_j$ we obtain for $i,j \in \{1,2,\ldots,N\}$:
\begin{equation}
\label{eq:thm:explicit_order_00}
\begin{split}
v_j^{k+1}-v_i^{k+1}
= & \,\, v_j^k - v_i^k
+ \tau \cdot (\tilde{w}_{j,i} + \tilde{w}_{i,j}) \cdot \varPhi(v_i^k - v_j^k)\\
& \hspace{.5ex}
%\begin{split}
+ \, \tau \cdot
\sum \limits_{\mathclap{\scriptstyle \ell = 1\atop \scriptstyle \ell \neq 
i,j}}^N
\Big( \tilde{w}_{j,\ell} \cdot \varPhi(v_\ell^k - v_j^k) -
\tilde{w}_{i,\ell} \cdot \varPhi(v_\ell^k - v_i^k) \Big)\\
%\end{split}\\
%
& \hspace{.5ex}
%\begin{split}
- \, \tau \cdot \sum \limits_{\ell = 1}^N
\Big( \tilde{w}_{j,\ell} \cdot \varPhi(v_\ell^k + v_j^k) -
\tilde{w}_{i,\ell} \cdot \varPhi(v_\ell^k + v_i^k) \Big) .
%\end{split}
\end{split}
\end{equation}
Now remember that $v_i^k < v_j^k$ by assumption and that -- as a consequence -- 
\begin{equation}
\label{eq:thm:explicit_order_01}
\tau \cdot (\tilde{w}_{j,i} + 
\tilde{w}_{i,j}) \cdot \varPhi(v_i^k - v_j^k) > 0 .
\end{equation}
Using the fact that $\tilde{w}_{j,k} = \tilde{w}_{i,k}$ for $1 \leq 
i,j,k \leq N$ and that $\varPhi$ is Lipschitz in $(0,2)$, we also know that
\begin{align}
\nonumber
T_1 & :=
\tau \cdot
\sum \limits_{\scriptstyle \ell = 1\atop \scriptstyle \ell \neq i,j}^N
\Big| \tilde{w}_{j,\ell} \cdot \varPhi(v_\ell^k - v_j^k) -
\tilde{w}_{i,\ell} \cdot \varPhi(v_\ell^k - v_i^k) \Big|\\
\nonumber
& = \tau \cdot 
\sum \limits_{\scriptstyle \ell = 1\atop \scriptstyle \ell \neq i,j}^N
\tilde{w}_{j,\ell} \cdot
\Big| \varPhi(v_\ell^k - v_j^k) - \varPhi(v_\ell^k - v_i^k) \Big|\\
& \leq \tau \cdot L_\varPhi \cdot |v_i^k - v_j^k| \cdot
\sum \limits_{\scriptstyle \ell = 1\atop \scriptstyle \ell \neq i,j}^N
\tilde{w}_{j,\ell} \enspace,\\
%
%\end{align}
%\begin{align}
\nonumber
T_2 & :=
\tau \cdot
\sum \limits_{\ell = 1}^N
\Big| \tilde{w}_{j,\ell} \cdot \varPhi(v_\ell^k + v_j^k) -
\tilde{w}_{i,\ell} \cdot \varPhi(v_\ell^k + v_i^k) \Big|\\
\nonumber
& = \tau \cdot
\sum \limits_{\ell = 1}^N
\tilde{w}_{j,\ell} \cdot 
\Big| \varPhi(v_\ell^k + v_j^k) - \varPhi(v_\ell^k + v_i^k) \Big|\\
& \leq \tau \cdot L_\varPhi \cdot |v_j^k-v_i^k| \cdot
\sum \limits_{\ell = 1}^N \tilde{w}_{j,\ell} \, .
\end{align}
Let the time step size $\tau$ fulfil \eref{eq:thm:explicit_order_000}. Then we 
can write
\begin{equation}
T_1 + T_2 <
2 \cdot L_\varPhi \cdot 2 \cdot |v_j^k-v_i^k| \cdot
\sum \limits_{\ell = 1}^N \tilde{w}_{j,\ell} <
v_j^k-v_i^k .
\end{equation}
In combination with $T_1, T_2 \geq 0$, it follows that 
\begin{equation}
\label{eq:thm:explicit_order_02}
T_2 - T_1 \geq 
-T_2 - T_1 >
-(v_j^k - v_i^k) ,
\end{equation}
and we immediately know that $v_j^k - v_i^k - T_1 + T_2 > 0$. Together
with \eref{eq:thm:explicit_order_00} and \eref{eq:thm:explicit_order_01} we get 
$0 < v_j^{k+1} - v_i^{k+1}$, as claimed.\qed
\end{proof}
%
% ------------------------------------------------------------------------------
%
\begin{theorem}[Convergence of the Explicit Scheme]
\label{thm:conv_expl}
Let \eref{eq:model_energy_red} be a twice
continuously differentiable convex function.
Then the explicit scheme \eref{eq:explicit_redundant} converges for time step 
sizes
\begin{equation}
\label{eq:thm:conv_explicit_000}
0 < \tau \leq \frac{1}{2 \cdot L_{\varPhi} \cdot 
\max\limits_{1 \leq i \leq N} \sum \limits_{j=1}^N \tilde{w}_{i,j}}
< \frac{2}{L} ,
\end{equation}
where $L_\varPhi$ denotes the Lipschitz constant of $\varPhi$ restricted to the 
interval $(0,2)$ and $L$ refers to the Lipschitz constant of the gradient of
\eref{eq:model_energy_red}.
\end{theorem}
\begin{proof}
Convergence of the gradient method to the global minimum of 
$E(\bm{v},\tilde{\bm{W}})$ is well-known for continuously differentiable convex 
functions with Lipschitz continuous gradient and
time step sizes $0 < \tau < 2 / L$ 
(see e.g. \cite[Theorem 2.1.14]{Ne04}).
A valid Lipschitz constant is given by $L_{\mathrm{max}}$ as defined in 
\eref{eq:lipschitz_max}.
Consequently, the time step sizes $\tau$ need to fulfil 
\eref{eq:thm:conv_explicit_000} in order to ensure convergence of 
\eref{eq:explicit_redundant}.
%The smaller or equal relation results from the strictness in 
%\eref{eq:lipschitz_max}.\qed
The smaller or equal relation results from \eref{eq:lipschitz_max}.
%which implies that $L_{\mathrm{max}} > L$.
The latter defines $L_{\mathrm{max}} > L$ such that 
$\tau = 2/L_{\mathrm{max}}$ represents a valid time step size.\qed
\end{proof}
%
% ------------------------------------------------------------------------------
%
\begin{remark}[Optimal Time Step Size]
The optimal time step size, 
i.e. the value of $\tau$ which leads to most rapid descent, is given by $\tau = 
1/L$ (see e.g. \cite[Section 2.1.5]{Ne04}). Thus, 
we suggest to use $\tau = 1/L_{\mathrm{max}}$.
\end{remark}
%
% ------------------------------------------------------------------------------
% ------------------------------------------------------------------------------
%
\section{Application to Image Enhancement}
\label{sec:application}
Now that we have presented a stable and convergent numerical scheme, we apply 
\eref{eq:explicit_redundant} to enhance the contrast of digital greyscale and 
colour images.
Throughout all experiments we use $\varPsi = \varPsi_{1,1}$ (cf. 
\Tref{tab:class_psi} and \Fref{fig:psi_dpsi_phi}).
%
% ------------------------------------------------------------------------------
%
\subsection{Greyscale Images}
The application of the proposed model to greyscale images follows the ideas 
presented in \cite{BW16}. We define a 
digital greyscale image as a mapping $f: \{1,\ldots,n_x\} \times 
\{1,\ldots,n_y\} \to [0,1]$. Note that all grey values are mapped to the 
interval $(0,1)$ to ensure the validity of our model before processing.
The grid position of the $i$-th image pixel is 
given by the vector $\bm{x}_i$ whereas $v_i$ denotes the corresponding grey 
value. Subsequently, we will see that a well-considered choice of the weighting 
matrix $\bm{\tilde{W}}$ allows either to enhance the \emph{global} or the 
\emph{local} contrast of a given image.
%
% ------------------------------------------------------------------------------
%
\subsubsection{Global Contrast Enhancement}
\label{sec:app:grey:global}
For global contrast enhancement we make use of the global model as discussed in
\Sref{sec:global_model}. Only the $N$ different occurring grey values $v_i$ -- 
and not their positions in the image -- are considered. We let every entry 
$\tilde{w}_{i,j}$ of the weighting 
matrix denote the frequency of grey value $v_j$ in the image. Assuming an 8-bit 
greyscale image this leads to a weighting matrix of size $256 \times 256$ which 
is independent of the image size. As illustrated in \Fref{fig:app:grey_global}, 
global contrast enhancement can be achieved in two ways:
\begin{figure}[t]
\newdimen\imgwidtha
\imgwidtha=0.32\textwidth
\centering
\includegraphics[width=\imgwidtha]{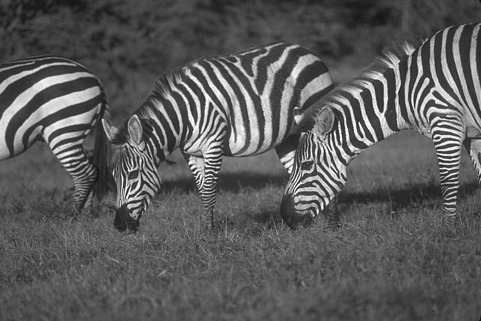} \hfill
\includegraphics[width=\imgwidtha]{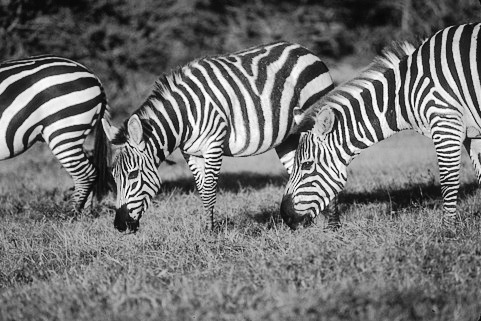} \hfill
\includegraphics[width=\imgwidtha]{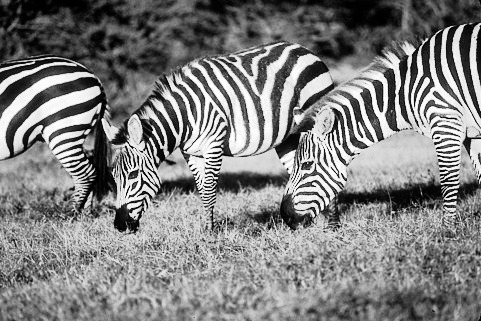}\\[0.02\textwidth]
\includegraphics[width=\imgwidtha]{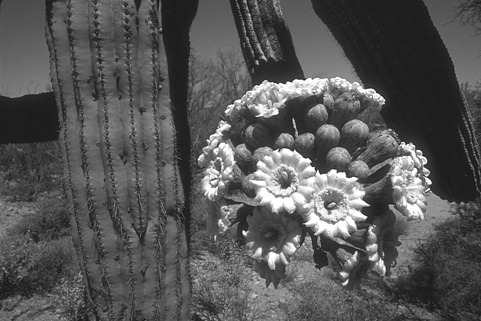} \hfill
\includegraphics[width=\imgwidtha]{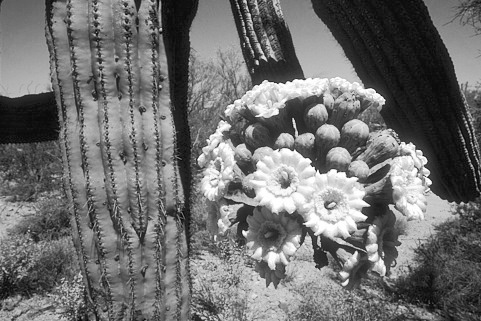} \hfill
\includegraphics[width=\imgwidtha]{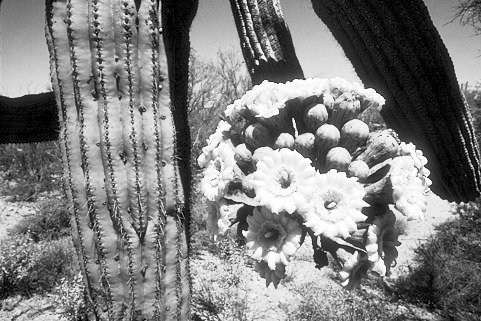}\\[0.02\textwidth]
\includegraphics[width=\imgwidtha]{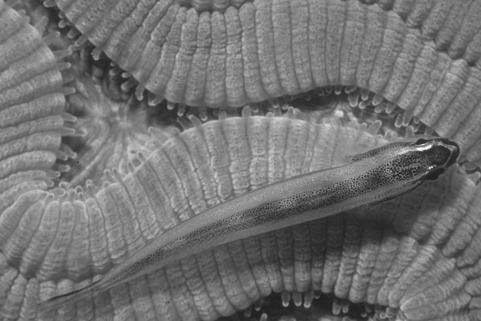} \hfill
\includegraphics[width=\imgwidtha]{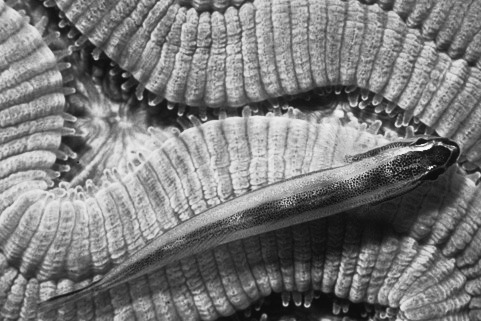} \hfill
\includegraphics[width=\imgwidtha]{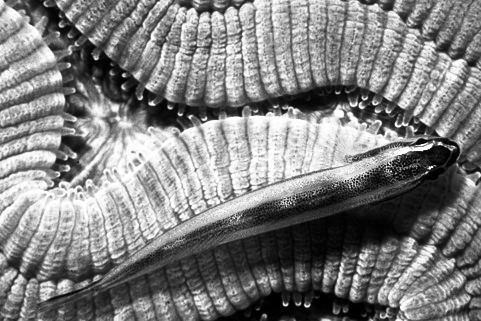}\\
\makebox[\imgwidtha]{Original image}
\hfill
\makebox[\imgwidtha]{$t = 2 \cdot 10^{-6}$}
\hfill
\makebox[\imgwidtha]{Steady state \eref{eq:steady_state_linear}}
\caption{Global contrast enhancement using $\varPhi = \varPhi_{1,1}$ and 
greyscale versions of images from the BSDS500 \cite{AMFM11}.}
\label{fig:app:grey_global}
\end{figure}
As a first option one can use the explicit scheme \eref{eq:explicit_redundant} 
to describe the evolution of all grey values $v_i$ up to some time $t$
(see column two
of \Fref{fig:app:grey_global}).
The amount of contrast enhancement grows with increasing values of $t$. In our 
experiments an image size of
$481 \times 321$
pixels and the application of the 
flux function $\varPhi_{1,1}$ with $L_\varPhi=1$ imply an upper bound of
$1/(2 \cdot 481 \cdot 321)$ for $\tau$. Thus, we can achieve the time $t = 2 
\cdot 10^{-6}$ in \Fref{fig:app:grey_global} in a single iteration.
If one is only interested in an enhanced version of the original image with 
maximum global contrast there is an alternative, namely the derived steady 
state solution for linear flux functions \eref{eq:steady_state_linear}. The 
results are shown in the last column of \Fref{fig:app:grey_global}. 
This figure 
also confirms that the solution of the explicit scheme 
\eref{eq:explicit_redundant} converges to the steady-state solution 
\eref{eq:steady_state_linear} for $t \to \infty$. From 
\eref{eq:steady_state_linear} it is clear that this steady state is equivalent 
to histogram equalisation.
In summary, this means that the application of our global model to greyscale 
images offers
an evolution equation
histogram equalisation which allows to 
control the amount of contrast enhancement in an intuitive way through the time 
parameter $t$.
%
% ------------------------------------------------------------------------------
%
\subsubsection{Local Contrast Enhancement}
\label{sec:app:local_grey}
In order to achieve local contrast enhancement 
we use our model to describe the evolution of grey 
values $v_i$ at all $n_x \cdot n_y$ image grid positions. The change of every 
grey value $v_i$ depends on all grey values within a disk-shaped 
neighbourhood of radius $\varrho$ around its grid position $\bm{x}_i$.
We assume that
\begin{equation}
\label{eq:local_model_weights}
\tilde{w}_{i,j} := \gamma(|\bm{x}_j-\bm{x}_i|),
\qquad \forall i,j \in \{1,2,\ldots,N\} ,
\end{equation}
where we weight the spatial distance $|\bm{x}_j-\bm{x}_i|$ by a function 
$\gamma: \mathbb{R}_0^+ \to [0,1]$ with compact support $[0,\varrho)$ 
which fulfils
\begin{equation}
\label{eq:local_model_gamma}
%\begin{split}
\begin{aligned}
& \gamma(x) \in (0,1], & \text{if } x < \varrho,\\
& \gamma(x) = 0, & \text{if } x \geq \varrho.
\end{aligned}
%\end{split}
\end{equation}
The choice of $\gamma$ is application dependent. However, it usually makes 
sense to define $\gamma(x)$ as a non-increasing function in $x$. Possible 
choices are e.g.
\begin{align}
\label{eq:local_model_gamma_01}
\gamma_1(x) & =
\begin{cases}
1 , & \text{if } x < \varrho ,\\
0 , & \text{else} ,
\end{cases}\\
\label{eq:local_model_gamma_02}
\gamma_2(x) & =
\begin{cases}
%1 - 6x^2 + 6 x^3 , & \text{if } 0 \leq x < \frac\varrho2 ,\\
%2 \cdot (1 - x)^3 , & \text{if } \frac\varrho2 \leq x < \varrho ,\\
1 - 6 \frac{x^2}{\varrho^2} + 6 \frac{x^3}{\varrho^3} , & \text{if } 0 \leq x < 
\frac\varrho2 ,\\
2 \cdot (1 - \frac{x}{\varrho})^3 , & \text{if } \frac\varrho2 \leq x < \varrho 
,\\
0 , & \text{else} ,
\end{cases}
\end{align}
which are both sketched in \Fref{fig:app:local_kernel}.
\begin{figure}
    \centering
    \begin{tikzpicture}
    \draw[->] (-0.1,0) -- (2.75,0) node[right] {$x$};
    \draw[dashed] (2.5,0) node[below] {$\varrho$} -- (2.5,2.75);
    \draw[->] (0,-0.1) -- (0,2.75) node[above] {$\gamma_1(x)$} 
    node[left,pos=0.9] 
    {$1$};
    \draw[scale=2.5,domain=0:1,smooth,variable=\x,blue] plot ({\x},{1});
    \end{tikzpicture}
%     \hspace{1.5cm}
    \begin{tikzpicture}
    \draw[->] (-0.1,0) -- (2.75,0) node[right] {$x$};
    \draw[dashed] (2.5,0) node[below] {$\varrho$} -- (2.5,2.75);
    \draw[->] (0,-0.1) -- (0,2.75) node[above] {$\gamma_2(x)$} 
    node[left,pos=0.9] 
    {$1$};
    \draw[dashed] (0,2.5) -- (2.75,2.5);
    \draw[scale=2.5,domain=0:0.5,smooth,variable=\x,blue] plot 
    ({\x},{1-6*\x*\x+6*\x*\x*\x});
    \draw[scale=2.5,domain=0.5:1,smooth,variable=\x,blue] plot 
    ({\x},{2*(1-\x)*(1-\x)*(1-\x)});
    \end{tikzpicture}\\
    \caption{Box function $\gamma_1$ and scaled cubic B-spline $\gamma_2$.}
    \label{fig:app:local_kernel}
\end{figure}
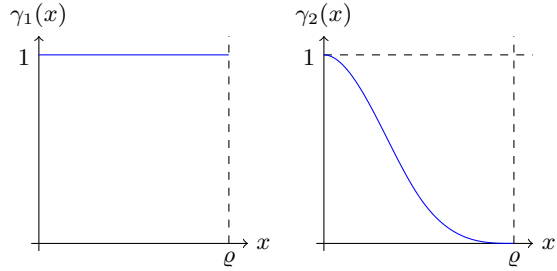
When applying this local model to images we make use of mirroring boundary 
conditions in order to avoid artefacts at the image boundaries. 
\Fref{fig:app:grey_local} provides an example for local contrast enhancement of 
digital greyscale images.
\begin{figure}[t]
\centering
\imgwidtha=0.32\textwidth
\includegraphics[width=\imgwidtha]{16068-grey} \hfill
\includegraphics[width=\imgwidtha]{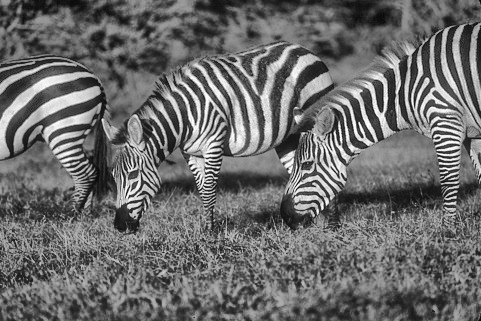} \hfill
\includegraphics[width=\imgwidtha]{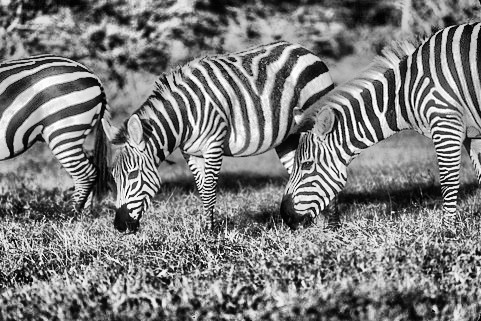}\\[0.02\textwidth]
\includegraphics[width=\imgwidtha]{19021-grey} \hfill
\includegraphics[width=\imgwidtha]{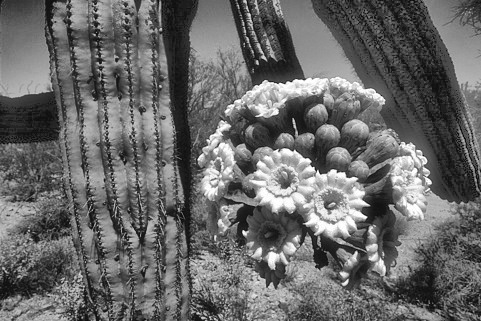} \hfill
\includegraphics[width=\imgwidtha]{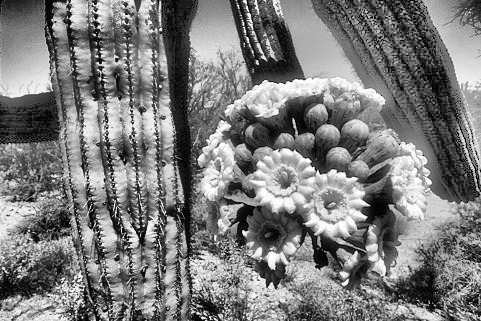}\\[0.02\textwidth]
\includegraphics[width=\imgwidtha]{209021-grey} \hfill
\includegraphics[width=\imgwidtha]{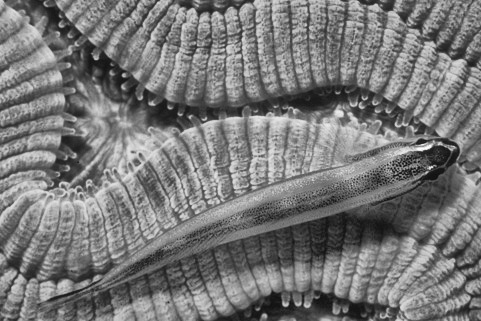} \hfill
\includegraphics[width=\imgwidtha]{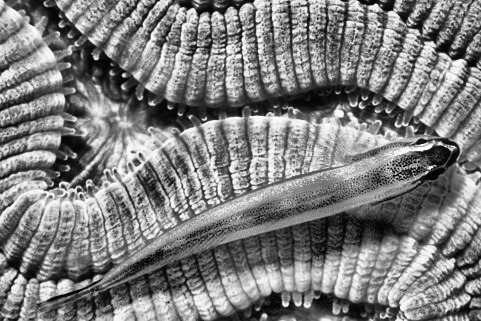}\\
\makebox[\imgwidtha]{Original image}
\hfill
\makebox[\imgwidtha]{$t = 2 \cdot 10^{-5}$}
\hfill
\makebox[\imgwidtha]{$t = 4 \cdot 10^{-5}$}
\caption{Local contrast enhancement using $\varPhi = \varPhi_{1,1}$, 
$\gamma=\gamma_1$, 
$\varrho = 60$,
and greyscale versions of images from the BSDS500 \cite{AMFM11}.}
\label{fig:app:grey_local}
\end{figure}
Again, we describe the grey value evolution with the explicit scheme 
\eref{eq:explicit_redundant}. Furthermore, we use $\gamma_1$ to model the 
influence of neighbouring grey values. As is evident from 
\Fref{fig:app:grey_local},
%(b)-(d)
increasing the values for $t$ goes along with 
enhanced local contrast.
%
% ------------------------------------------------------------------------------
%
\subsection{Colour Images}
\label{sec:app:colour}
Based on the assumption that our input data is given in sRGB colour space 
\cite{SACM96} (in the following denoted by RGB) we represent a digital 
colour image by the mapping $f: \{1,\ldots,n_x\} \times \{1,\ldots,n_y\} \to 
[0,1]^3$. Subsequently, our aim is the contrast enhancement of digital colour 
images without distorting the colour information. 
This means that we only want to adapt the \emph{luminance} but not the 
\emph{chromaticity} of a given image.
For this purpose, we convert the given image data to YCbCr colour space 
\cite[Section 3.5]{Pr01} since this representation provides a separate 
luminance channel. Next, we perform contrast enhancement on the luminance 
channel only. Just as for greyscale images we map all Y-values to the interval 
$(0,1)$ to fulfil our model requirements. After enhancing the contrast, we 
transform the colour information of the image back to RGB colour space.

At this point it is important to mention that the colour gamut of the RGB 
colour space is a subset of the YCbCr colour gamut and during the conversion 
process of colour coordinates from YCbCr to RGB colour space the so-called 
colour gamut problem may occur: Colours from the YCbCr colour gamut
may lie outside the RGB colour gamut and thus cannot be represented in RGB 
colour coordinates.
Naik and Murthy \cite{NM03} state that a simple clipping of the values to the 
bounds creates undesired shift of hue and may lead to colour artefacts.
In order to avoid the colour gamut problem we adapt the ideas presented by
Nikolova and Steidl \cite{NS14} which are based on the intensity 
representation of the HSI colour space \cite[Section 6.2.3]{GW06}.
Using the original and enhanced intensities, they define an affine colour 
mapping and transform the original RGB values. This preserves the hue and 
results in an enhanced RGB image.
It is straightforward to show that their 
algorithms are valid for any intensity $\hat{f}$ of type
\begin{equation}
\hat{f} = c_r \cdot r + c_g \cdot g + c_b \cdot b ,
\end{equation}
with $c_r + c_g + c_b = 1$ and $c_r, c_g, c_b \in [0,1]$, where $r$, $g$, and 
$b$ denote RGB colour coordinates. Thus, they are applicable to the luminance 
representation of the YCbCr colour space, too, i.e. $c_r=0.299$, $c_g=0.587$, 
$c_b=0.114$. Tian and Cohen make use of the same idea in \cite{TC2017}. 
As in \cite{NS14}, our result 
image is a convex combination of the outcomes of a multiplicative and 
an additive algorithm (see \cite[Algorithm 4 and 5]{NS14}) with coefficients 
$\lambda$ and $1-\lambda$ for $\lambda 
\in [0,1]$. During our experiments we use a fixed value of $\lambda = 0.5$ (for 
details on how to choose $\lambda$ we refer to 
\cite{NS14}). An overview of our strategy for contrast enhancement of digital 
colour value images is given in \Fref{fig:app:colour_concept}.
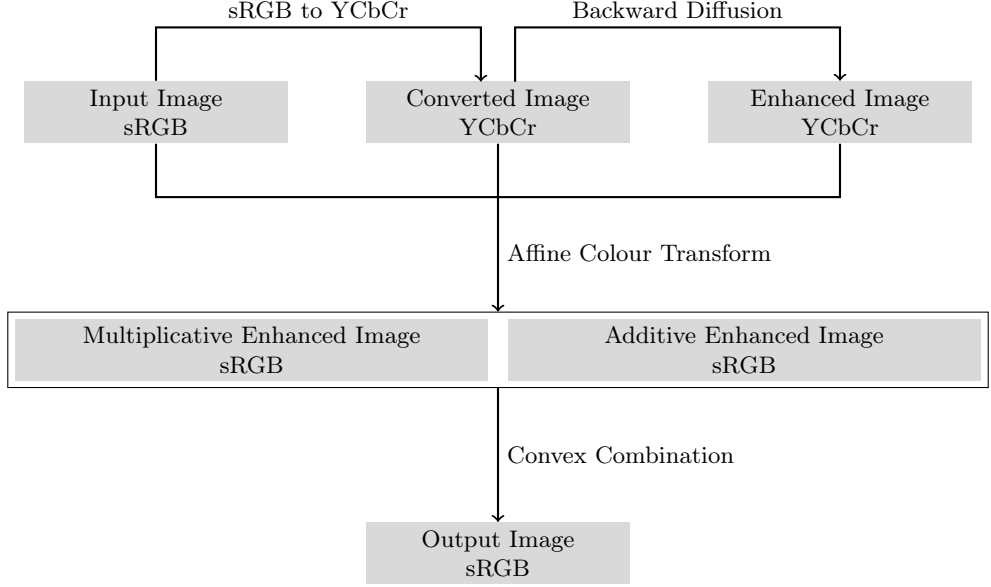
\begin{figure}
    \centering
    \begin{tikzpicture}[scale=.9]
    \tikzstyle{every node} = [rectangle, fill=gray!30]
    \node [minimum width=3.25cm, text width=3.25cm, align=center] (a) at (0, 
    6.5) 
    {Input Image\\{\footnotesize sRGB}};
    \node [minimum width=3.25cm, text width=3.25cm, align=center] (b) at (5, 
    6.5) 
    {Converted Image\\{\footnotesize YCbCr}};
    \node [minimum width=3.25cm, text width=3.25cm, align=center] (c) at (10, 
    6.5) 
    {Enhanced Image \\{\footnotesize 
            YCbCr}};
    \node [minimum width=6.0cm, text width=6.0cm, align=center] (d) at (1.4, 
    3.0) 
    {Multiplicative Enhanced Image \\{\footnotesize sRGB}};
    \node [minimum width=6.0cm, text width=6.0cm, align=center] (e) at (8.6, 
    3.0) 
    {Additive  Enhanced Image \\{\footnotesize sRGB}};
    \node [minimum width=3.25cm, text width=3.25cm, align=center] (g) at (5, 0) 
    {Output Image \\{\footnotesize sRGB}};
    \tikzstyle{every node} = []
    \node (f) at (5,3.0) [draw,minimum width=12.9cm,minimum height=1cm] {};
    \draw [->,thick] (a) -- (0, 7.75) -- (4.75, 7.75) node [pos=0.5,above] 
    {\small sRGB to YCbCr} -- (4.75, 6.95);
    \draw [->,thick] (5.25, 6.95) -- (5.25,7.75) -- (10,7.75) node 
    [pos=0.5,above] 
    {\small Backward Diffusion} -- (c);
    \draw [-,thick] (a) |- (5, 5.25) -| (c);
    \draw [->,thick] (b) -- (f) node[pos=0.65,right] {\small Affine Colour 
        Transform};
    \draw [->,thick] (f) -- (g) node[pos=0.5,right] {\small Convex Combination};
    \end{tikzpicture}\\
    
    \caption{Procedure of contrast enhancement for digital colour images 
    following \cite{NS14}.}
    \label{fig:app:colour_concept}
\end{figure}
%
% ------------------------------------------------------------------------------
%
\subsubsection{Global Contrast Enhancement}
Again, we apply the global model from \Sref{sec:global_model} in order to 
achieve global contrast enhancement. As mentioned before, we consider the $N$ 
different occurring Y-values of the YCbCr representation of the input image and 
denote them by $v_i$ (similar to \Sref{sec:app:grey:global} we neglect their 
positions in the image). Every entry of the weighting matrix 
$\tilde{w}_{i,j}$ contains the
number of occurences of the value 
$v_j$ in the Y-channel of the 
image. It becomes clear that the application of our model -- in this setting --
basically comes down to histogram equalisation of the Y-channel.
\Fref{fig:app:colour_global} shows the resulting RGB images after global 
contrast enhancement. Similar to the greyscale scenario, we can either apply 
the explicit scheme \eref{eq:explicit_redundant} or -- for $\varPhi = 
\varPhi_{a,1}$ -- 
estimate the steady state solution following \eref{eq:steady_state_linear}.
For the first case the amount of contrast enhancement grows with the positive 
time parameter $t$. The second
column of \Fref{fig:app:colour_global} 
shows the results for $\varPhi = \varPhi_{1,1}$
given time $t$.
The corresponding steady state solutions are illustrated in the last column of 
\Fref{fig:app:colour_global}.
\begin{figure}
\imgwidtha=0.32\textwidth
\centering
\includegraphics[width=\imgwidtha]{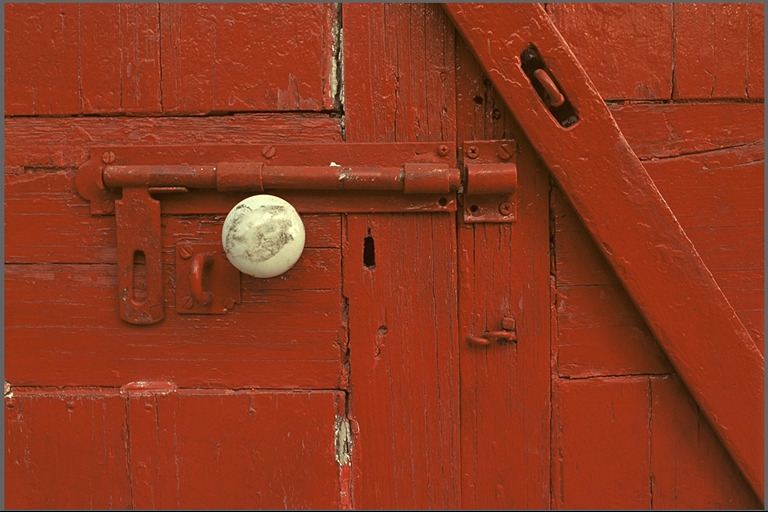} \hfill
\includegraphics[width=\imgwidtha]{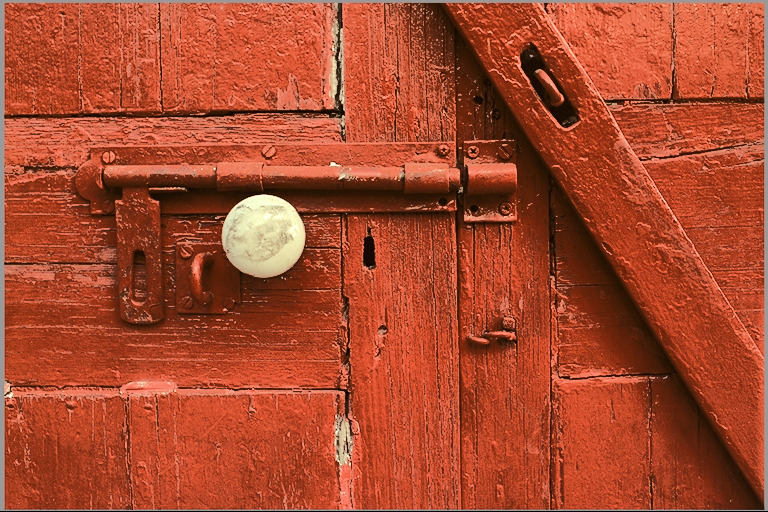} 
\hfill
\includegraphics[width=\imgwidtha]{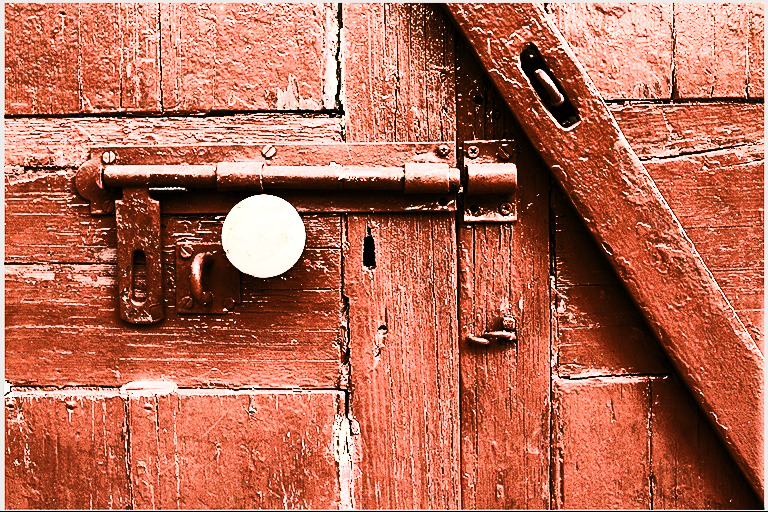}\\[0.02\textwidth]
\includegraphics[width=\imgwidtha]{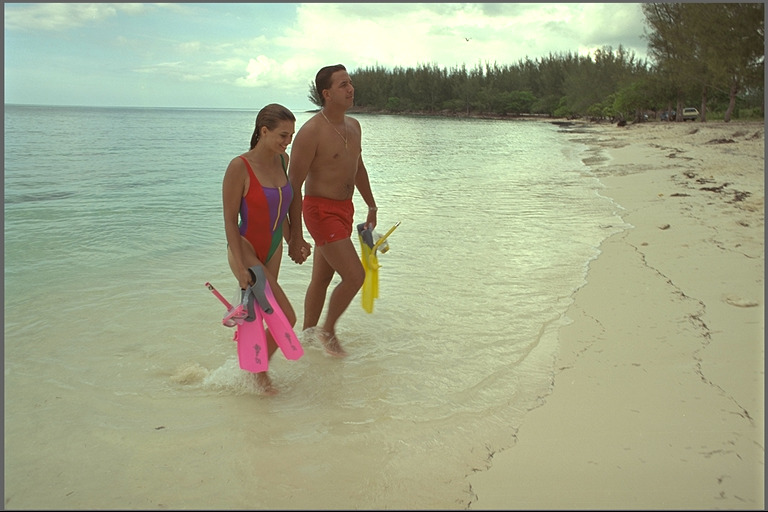} \hfill
\includegraphics[width=\imgwidtha]{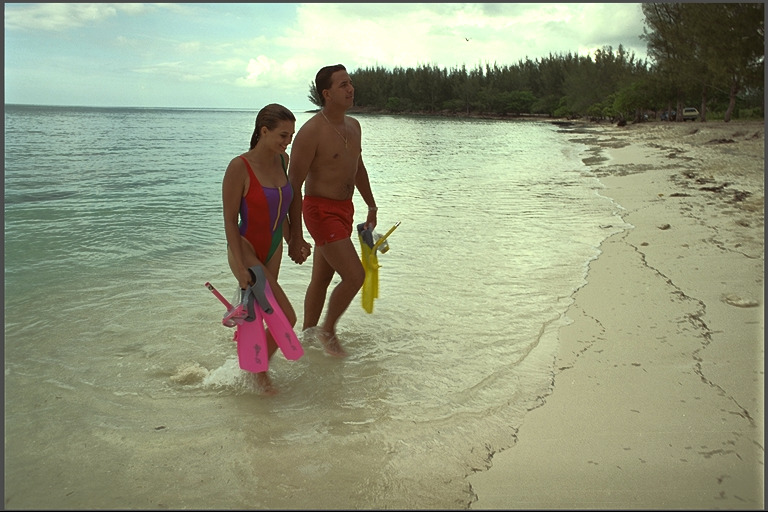} 
\hfill
\includegraphics[width=\imgwidtha]{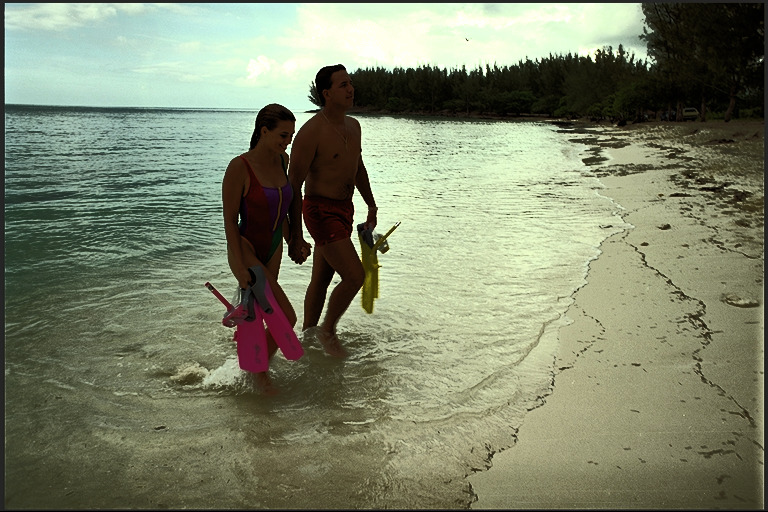}\\[0.02\textwidth]
\includegraphics[width=\imgwidtha]{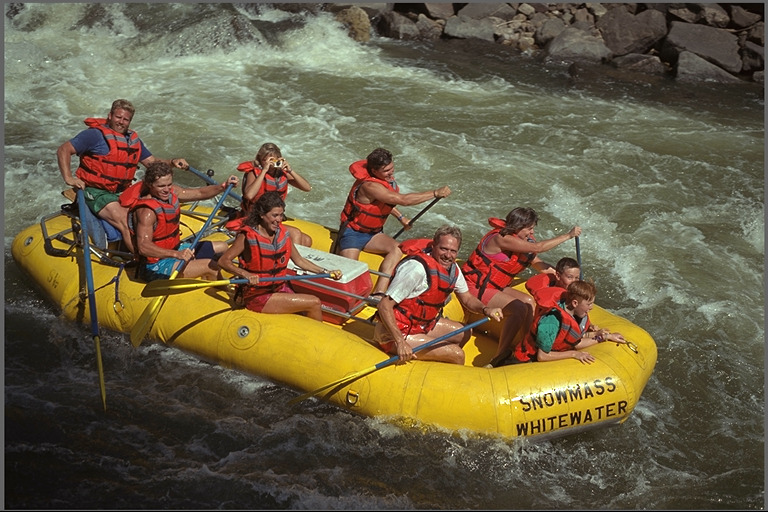} \hfill
\includegraphics[width=\imgwidtha]{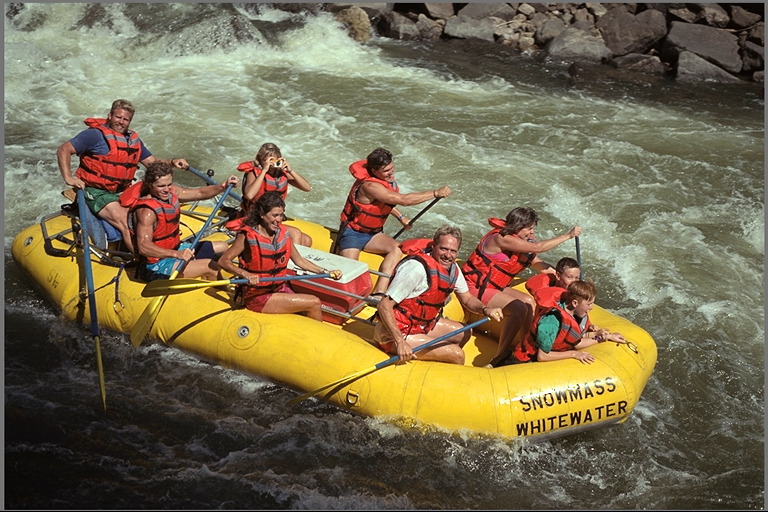} 
\hfill
\includegraphics[width=\imgwidtha]{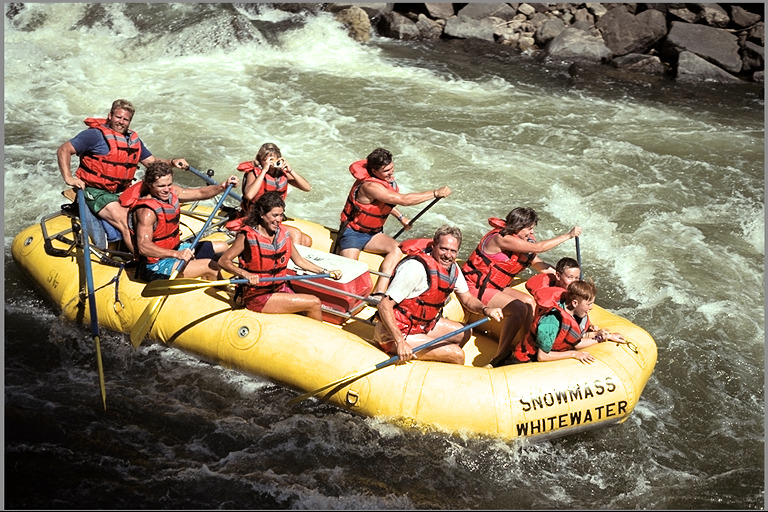}\\
\makebox[\imgwidtha]{Original image}
\hfill
\makebox[\imgwidtha]{$t = 5 \cdot 10^{-7}$}
\hfill
\makebox[\imgwidtha]{Steady state \eref{eq:steady_state_linear}}
\caption{Global contrast enhancement using $\varPhi = \varPhi_{1,1}$, 
$\lambda = 0.5$, and images from \cite{Kodak}.}
\label{fig:app:colour_global}
\end{figure}
\begin{figure}
\centering
\imgwidtha=0.32\textwidth
\includegraphics[width=\imgwidtha]{kodim02} \hfill
\includegraphics[width=\imgwidtha]{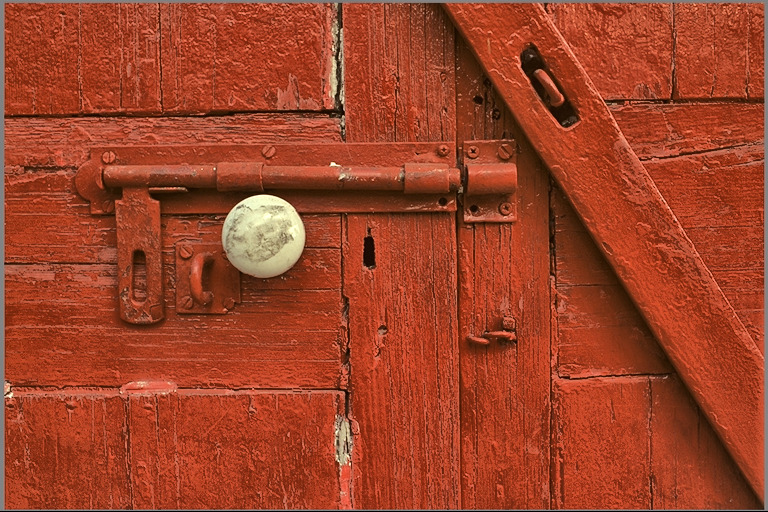}
 \hfill
\includegraphics[width=\imgwidtha]{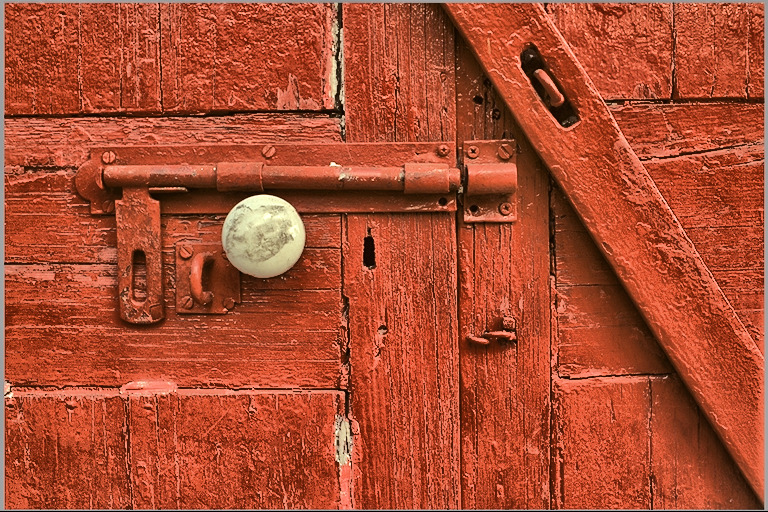}\\[0.02\textwidth]
\includegraphics[width=\imgwidtha]{kodim12} \hfill
\includegraphics[width=\imgwidtha]{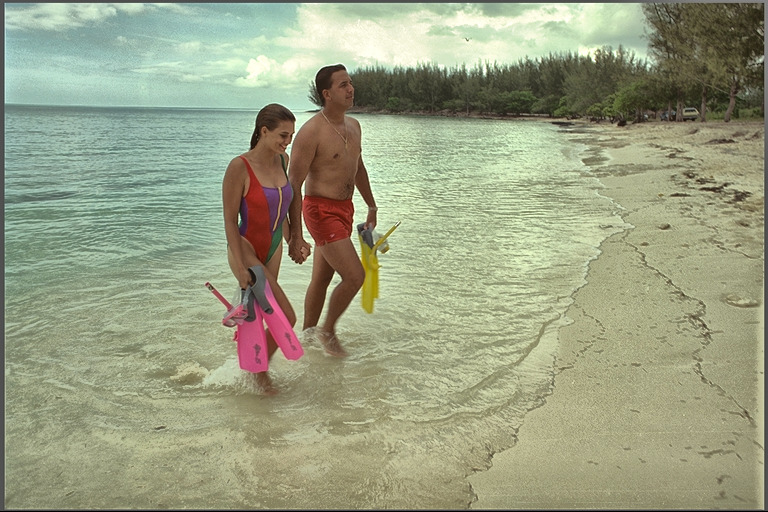}
\hfill
\includegraphics[width=\imgwidtha]{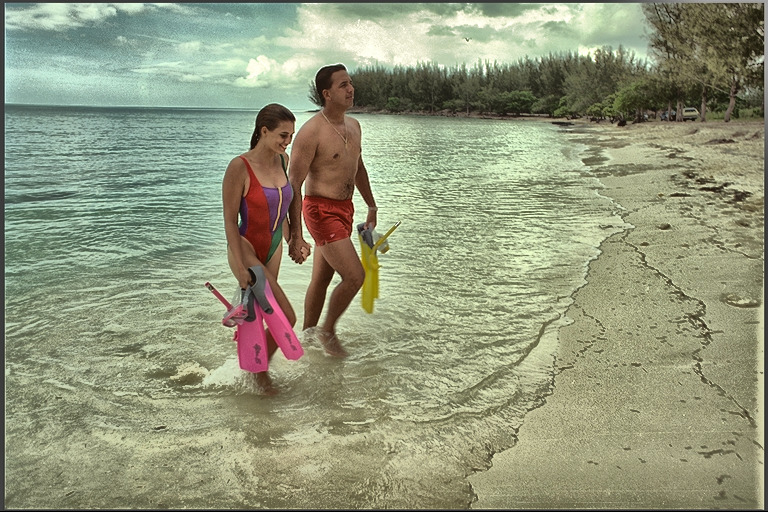}\\[0.02\textwidth]
\includegraphics[width=\imgwidtha]{kodim14} \hfill
\includegraphics[width=\imgwidtha]{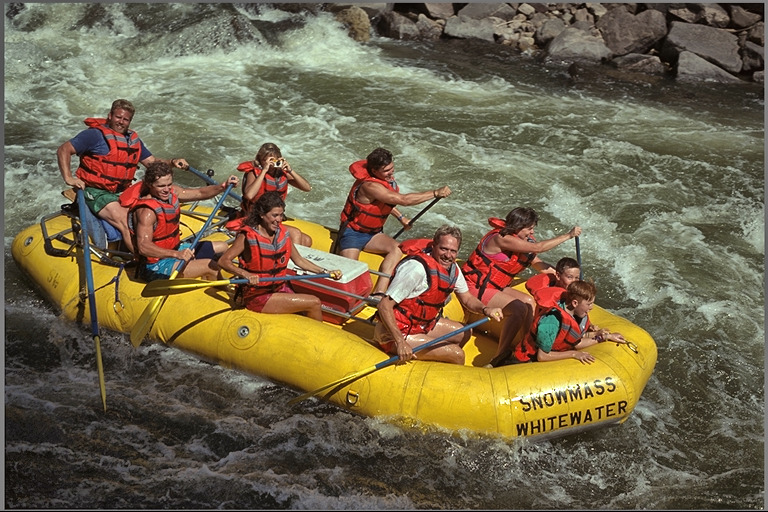}
\hfill
\includegraphics[width=\imgwidtha]{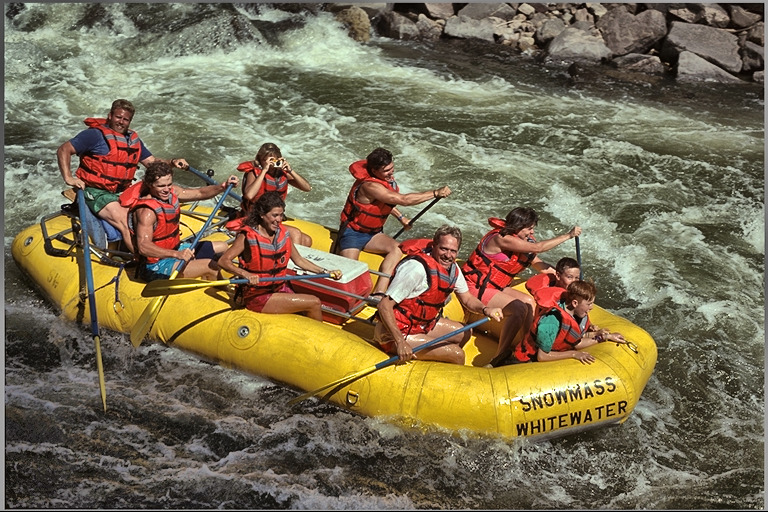}\\
\makebox[\imgwidtha]{Original image}
\hfill
\makebox[\imgwidtha]{$t = 1 \cdot 10^{-5}$}
\hfill
\makebox[\imgwidtha]{$t = 2 \cdot 10^{-5}$}
\caption{Local contrast enhancement using $\varPhi = \varPhi_{1,1}$, 
$\gamma=\gamma_1$,  
$\varrho = 60$, $\lambda = 0.5$,
and images from \cite{Kodak}.}
\label{fig:app:colour_local}
\end{figure}
%
% ------------------------------------------------------------------------------
%
\subsubsection{Local Contrast Enhancement}
In a similar manner -- and adapting the ideas from Subsection 
\ref{sec:app:local_grey} -- we achieve local contrast enhancement in colour 
images. For this purpose we describe the evolution of Y-values $v_i$ 
at all $n_x \cdot n_y$ image grid positions using a disk-shaped neighbourhood 
of radius $\varrho$ around the corresponding grid positions $\bm{x}_i$. The 
entries of the weighting matrix $\bm{\tilde{W}}$ follow 
\eref{eq:local_model_weights}. In combination with mirrored boundary conditions 
the explicit scheme \eref{eq:explicit_redundant} allows to increase the local 
contrast of an image with growing $t$. \Fref{fig:app:colour_local} shows 
exemplary results for $\varPhi = \varPhi_{1,1}$ and $\gamma = \gamma_1$ (cf. 
\eref{eq:local_model_gamma_01}). Note, how well -- in comparison to the global 
set-up in \Fref{fig:app:colour_global} -- 
the structure of the door gets enhanced while the details of the door knob are 
preserved. The differences are even larger in the second image: For both the 
couple in the foreground and the background scenery, contrast increases 
which implies visibility also for larger times $t$.
%
% ------------------------------------------------------------------------------
%
\subsection{Parameters}
In total, our model has up to six parameters: $\varPhi$, $\tau$, $t$, 
$\lambda$, $\varrho$, and $\gamma$.
During our experiments we have fixed
$\varPhi(s)$ to the linear flux function $\varPhi_{1,1}(s)$ 
and $\lambda$ to $0.5$. Valid bounds for 
the time step size $\tau$ are given in Theorems
\ref{thm:range_preserv}--\ref{thm:conv_expl}. From the theory in 
\Sref{sec:theory} and the subsequent
experiments on greyscale and colour images it becomes clear that
the amount of contrast enhancement grows with the diffusion time.
Thus, it remains to discuss the influence of the parameters $\varrho$ and 
$\gamma$.
We found out that the neighbourhood radius $\varrho$ affects the diffusion time 
and controls the amount of perceived local contrast enhancement, i.e. it steers 
the localisation of the contrast enhancement process.
Whereas small radii lead to high contrast in 
already small image areas, the size of image sections with high contrast 
increases with $\varrho$. For sufficiently large values of $\varrho$ global 
histogram equalisation is approximated.
Another interesting point is the choice of the weighting function $\gamma$.
Overall, choosing
$\gamma=\gamma_1$ leads to more homogeneous contrast enhancement resulting in 
smoother perception. For $\gamma=\gamma_2$ the focus always lies on the 
neighbourhood centre which implies even more enhancement of local structures 
than in the preceding case.
In summary,  
$\gamma_2$ leads to more enhancement which, however, also creates 
undesired effects in smooth or noisy regions. Thus, we prefer $\gamma_1$ over 
$\gamma_2$.
Further experiments which visualise the effect of the parameters can be found 
in the supplementary material.
%
% ------------------------------------------------------------------------------
%
\subsection{Related Work from an Application Perspective}
Now that we have demonstrated the applicability of our model to digital images 
we want to discuss briefly its relation to other existing theories in the 
context of image processing.

As mentioned in \Sref{sec:app:grey:global}, applying the global model -- with 
%$\bm{\tilde{W}} = \bm{1}\bm{1}\transpose$
the entries of $\bm{\tilde{W}}$ representing the grey value frequencies
-- is identical to histogram 
equalisation (a common formulation can e.g. be found in \cite{GW06}). 
Furthermore, there exist other closely 
related  histogram specification techniques -- such as 
\cite{SC97,NS14a,NWC13} -- which can have the same steady state.
If we compare our evolution with the histogram modification flow introduced by 
Sapiro and Caselles \cite{SC97}, we see that their flow
can also be translated into a combination of repulsion among grey-values and a 
barrier function. However, in \cite{SC97} the repulsive force is constant, 
and the barrier function quadratic. Thus, they cannot be derived from the same 
kind of interaction between the $v_i$ and their reflected counterparts as in 
our paper.

Referring to \Sref{sec:app:local_grey}, there also exist well-known approaches 
which aim to enhance the local image contrast such as adaptive histogram 
equalisation -- see \cite{PAACGGRZZ87} and the references therein -- or 
contrast limited adaptive histogram equalisation \cite{Zu94}. The 
latter technique tries to overcome the over-amplification of noise in mostly 
homogeneous image regions when using adaptive histogram 
equalisation. Both approaches share the basic idea 
with our approach in \Sref{sec:app:local_grey} and perform histogram 
equalisation for each pixel, i.e. the mapping function for every pixel is 
determined using a neighbourhood of predefined size and its corresponding 
histogram. 

Another related research topic is the rich field of colour image enhancement 
which we broach in \Sref{sec:app:colour}.
A short review of existing methods -- as well as two new ideas -- is presented 
in \cite{BK07}. Therein, Bassiou and Kotropoulus also mention the colour gamut 
problem for methods which perform contrast enhancement in a different colour 
space and transform colour coordinates to RGB afterwards. Of particular 
interest are the publications by Naik and Murthy \cite{NM03} and Nikolova and 
Steidl \cite{NS14} whose ideas are used in \Sref{sec:app:colour}. 
Both of them suggest -- based on an affine colour 
transform -- strategies to overcome the colour gamut problem while avoiding 
colour artefacts in the resulting image. A recent approach which also makes use 
of these ideas is presented by Tian and Cohen \cite{TC2017}. Ojo et al.
\cite{OSA16} make use of the HSV colour space to avoid the colour gamut problem 
when enhancing the contrast of colour images. A variational approach for 
contrast enhancement which tries to approximate the hue of the input image was 
recently published by Pierre et al. \cite{PABST17}.
%
% ------------------------------------------------------------------------------
% ------------------------------------------------------------------------------
%
\section{Conclusions and Outlook}
\label{sec:summary}
In our paper we have presented a mathematical model which describes pure 
backward diffusion as gradient descent of strictly convex energies.
The underlying evolution makes use of ideas from the area of 
collective behaviour and -- in terms of the latter -- our model can be 
understood as a fully repulsive discrete first order swarm model.
Not only it is surprising that our model allows backward diffusion to be 
formulated as a convex optimisation problem but also that it is sufficient to 
impose reflecting boundary conditions in the diffusion co-domain in order to 
guarantee stability.
This strategy is contrary to existing approaches which either assume forward or 
zero diffusion at extrema or add classical fidelity terms to avoid 
instabilities.
Furthermore, discretisation of our model does not require sophisticated 
numerics. We have proven that a straightforward explicit scheme is sufficient 
to preserve the stability of the time-continuous evolution.
In our experiments, we show that our model can directly be applied to contrast 
enhancement of digital greyscale and colour images.

We see our contribution mainly as an \emph{example} of stable modelling of 
backward parabolic evolutions that create neither theoretical nor numerical 
problems. We are convinced that this concept has far more widespread 
applications in inverse problems, image processing, and computer vision. 
Exploring them will be part of our future research.
%
% ------------------------------------------------------------------------------
% ------------------------------------------------------------------------------
%
%\begin{acknowledgements}
\subsection*{Acknowledgements}
Our research on intrinsic stability properties of specific backward
diffusion processes goes back to 2005, when Joachim Weickert was visiting
Mila Nikolova in Paris. Unfortunately the diffusion equation considered
at that time did not have the expected properties, such that it took
one more decade to come up with more appropriate models.
Leif Bergerhoff wishes to thank Antoine Gautier and Peter Ochs for fruitful
discussions around the topic of the paper.
This project has received funding from the DFG Cluster of Excellence
\emph{Multimodal Computing and Interaction} and from the
European Union's Horizon 2020 research and innovation programme (grant
agreement No 741215, ERC Advanced Grant {\em INCOVID}).
This is gratefully acknowledged.
%\end{acknowledgements}
%
% ------------------------------------------------------------------------------
%
% BibTeX users please use one of
%\bibliographystyle{spbasic}      % basic style, author-year citations
%\bibliographystyle{spmpsci}      % mathematics and physical sciences
%\bibliographystyle{spphys}       % APS-like style for physics
\bibliographystyle{splncs04}
\bibliography{manuscript.bib}
\end{document}